\documentclass[a4paper]{article}
\usepackage[utf8]{inputenc}
\usepackage{amsmath, amsthm}
\usepackage{amsfonts}
\usepackage{amssymb}
\usepackage{mathtools}
\usepackage{etoolbox}
\usepackage{cite}
\usepackage{a4wide}
\usepackage{cases, color, graphicx}
\usepackage[dvipsnames]{xcolor}
\usepackage{caption}
\usepackage{subcaption}
\usepackage{units}
\graphicspath{{figures/}}
\usepackage[textsize=footnotesize,backgroundcolor=yellow!20,bordercolor=red]{todonotes}
\usepackage{tikz, pgfplots}
\newtheorem{definition}{Definition}
\newtheorem{theorem}{Theorem}

\newtoggle{tikz_pic}
\togglefalse{tikz_pic}

\numberwithin{equation}{section}

\newcommand{\R}{\mathbb{R}}

\DeclareMathOperator{\sign}{sgn}

\pgfplotsset{compat=1.14}

\begin{document}
\title{Qualitative Properties of Mathematical Model For Data Flow\thanks{The United States Government retains and the publisher, by accepting the article for publication, acknowledges that the United States Government retains a non-exclusive, paid-up, irrevocable, world-wide license to publish or reproduce the published form of this manuscript, or allow others to do so, for the United States Government purposes. The Department of Energy will provide public access to these results of federally sponsored research in accordance with the DOE Public Access Plan (\texttt{http://energy.gov/downloads/doe-public-access-plan})} }
\date{\today}
\author{Cory D. Hauck\footnote{hauckc@ornl.gov} \\
	{\small\it Computational and Applied Mathematics Group} \\
	{\small\it Oak Ridge National Laboratory} \\
	{\small\it 1 Bethel Valley Road, Bldg. 5700, Oak Ridge, TN 37831-6164, USA} \\ 
	{\small and} \\
	{\small\it Department of Mathematics} \\
	{\small\it University of Tennessee} \\
	{\small\it 227 Ayres Hall. 1403 Circle Drive. Knoxville TN 37996-1320, USA \vspace{10mm}} \\
	Michael Herty\footnote{herty@igpm.rwth-aachen.de} \and
	Giuseppe Visconti\footnote{visconti@igpm.rwth-aachen.de} \\
	{\small\it Institut f\"{u}r Geometrie und Praktische Mathematik (IGPM)} \\
	{\small\it RWTH Aachen University} \\
	{\small\it Templergraben 55, 52062 Aachen, Germany\vspace{5mm}} 
}

\maketitle

\begin{abstract}
In this paper, properties of a recently proposed mathematical model for data flow in large-scale asynchronous computer systems are analyzed.
In particular, the existence of special weak solutions based on propagating fronts is established.  Qualitative properties of these solutions are investigated, both theoretically
and numerically. 

\end{abstract}


\section{Introduction} \label{sec:intro}

The increasing number and diversity of processing units in modern supercomputers presents a significant challenge to the understanding of how data is globally distributed in these systems.  This issue is critically important in emerging computer architectures for which the cost of data movement is becoming a critical driver for system design, affecting hardware and software choices and even basic algorithms \cite{Dongarra2014,vetter2018extreme,murray2018basic}.  
At the same time, these supercomputers are becoming increasingly difficult to model. Indeed, simulating the detailed state of a supercomputer requires an even larger computer, thereby making such approaches prohibitive for the largest systems.  

With the help of new mathematical models, a more efficient utilization of available processing units might be possible.  Such models are necessarily coarse-grained approximations, used to understand the gross behavior of the system in order to make design, control, and optimization more tractable.  In this spirit, a coarse-grained mathematical model has been recently derived in \cite{Barnard2019} to describe the movement of processed data in an extreme-scale computer system.   This model focuses on a idealized setting, in which a lattice of processors perform operations in parallel and communicate with nearest neighbors.  Although very simple, the model allows for local variations in processor speed, due to differences in hardware performance or local problem complexity, as well as asynchronously operation, meaning that processors manipulate and share data with their neighbors as they are able, as opposed to a staged process in which processors communicate only after they have all completed a step of assigned work.   

The model proposed in \cite{Barnard2019} takes the form of a PDE for the density of data at a particular processor location and stage in the global computation.  It is derived formally from a microscopic ODE description.  The notion of coarse-graining discrete and semi-discrete model of networks and related systems to derive PDE models is common mathematical strategy that is used in many applications, including vehicular~\cite{aw2002SIAP,DiFrancescoRosini2015,HoldenRisebro2019} and pedestrian traffic flow~\cite{DegondEtAl2011,BellomoDogbe2008,ChertockEtAl2014,Helbing1992}, supply chains~\cite{ArmbrusterEtAl2006,ArmbrusterEtAl2011}, bacterial movement~\cite{GalanteLevy2013} and machine learning applications~\cite{ChanEtAl2019}.

In the current work, we establish some qualitative properties of the coarse-grained PDE model in \cite{Barnard2019}.  We focus specifically on the phenomena of front propagation and present simulation results that highlight the theoretical findings.  In the context of the current application, these fronts describe the movement of data through the different stages of a computational task.  The reason for this analysis is two-fold.  First, it was observed in \cite{Barnard2019} that, after a transient phase, many initial conditions eventually relax to a front whose shape depends on local variations in the processor speed. Second, these fronts can be used to characterize three important features for application:  the first time at which some portion of the data reaches a state of completion, the rate at which the rest of the data reaches this state, and the time at which all data reaches this stage.  Moreover, the relatively simple form of these fronts enables the investigation of optimization and control strategies based on some application-specific metrics.  As a first step in this direction, we explore how variations in the processor speed affect front behavior.

\section{Mathematical description of the model}
  In this section, we briefly summary the microscopic ODE model from \cite{Barnard2019} and the macroscopic model that is obtained from it.  The microscopic model describes a computer system as a lattice of processors, the dimension of which can be arbitrary, but finite.  For application purposes, only one-, two-, and three-dimensional lattices are realistic, and here we focus the one-dimensional case for simplicity.  It is assumed in \cite{Barnard2019} that each processor performs the same task, which is broken into discrete stages.  The rate at which a given processor moves data from stage to stage depends not only on its intrinsic processing rate, but also on the availability of usable data in the processor and its neighbors.  
  
  Mathematically, the amount of data at time $t$ that sits in stage $k \in \{1,\dots,k_{\max}\}$ of processor $i \in \{1,\dots,i_{\max}\}$ is given by $q_{i,k}(t)$.  
In the absence of any throttling, the flow of data between stages is given by a (processor dependent) rate $a_{i}(t) \geq 0$.  However the actual flow may be reduced due to a lack of available data in a given processor or its nearest neighbors. Throttling due to a lack of usable data in the nearest neighbors is characterized by the parameter $\eta:= k_{\max} / i_{\max} >0$.  As $\eta$ increases the global effect of local slowdowns increases.
\par 
For each $i \in \{1,\dots,i_{\max}\}$ and $k \in \{1,\dots,k_{\max}\}$, the evolution of  $q_{i,k}$
is modeled by  the ordinary differential equation
\begin{align}
\dot{q}_{i,k}(t) = f_{i,k-1}(t) - f_{i,k}(t), \quad 
{q}_{i,k}(0)={q}^0_{i,k}, \quad
f_{i,0}(t)=f^{\rm{in}}_{i}(t),
\label{eq:ode}
\end{align}
where ${q}^0_{i,k}$ is the (known) initial amount of data present in each processor $i$ at each stage $k$ and $f^{\rm{in}}_{i}(t)$ is the (known) inflow of unprocessed (or raw) data
at processor $i$ and time $t$.  For $k \in \{1,\dots,k_{\max}\}$, $f_{i,k}$ is the flow of data in processor $i$ from stage $k$ to $k+1$. It is given by the product of the maximum processing rate $a_i$ and the composition of two \textit{throttling functions} $v_1$ and $v_2$; that is,
\begin{equation}
	\label{eq:flux}
	f_{i,k} 
	= a_i v_1\Big(v_2\big(q_{i,k},Q_{i+1,k}-Q_{i,k}+q_{i,k},Q_{i-1,k}-Q_{i,k}+q_{i,k}\big);q_*\Big),
\end{equation} 
where $Q_{i,k}(t)$ is the total amount of data in processor $i$ that by time $t$ has traversed the first $k-1$ stages of the computation:
\begin{align}\label{S}
	Q_{i,k}(t) =
 \sum_{j=k}^{k_{\max}} q_{i,j}(t) + \int_0^t f_{i,k_{\max}}(s) \mathrm{d}s,
\end{align}
and $Q_{i\pm1,k}-Q_{i,k}+q_{i,k}$ is the amount of usable data available from processor $i \pm1$.  The throttling functions are given by
\begin{align} \label{eq:v1}
		v_1(q;q_*) &= \max \left \{0, \min \left \{1,\frac{q}{q_*} \right\} \right\}, \\
		\label{eq:v2}
v_2(q,\Delta_+,\Delta_-) &= \min\left\{q_{i,k},\max\{\Delta_+,0\},\max\{\Delta_-,0\}\right\}.
\end{align}
The function $v_1$ is a linear ramp, sometimes referred to as roof-line model \cite{williamsroofline2009}.  The steepness of the ramp is determined by the parameter $q_\ast>0$,  the value at which the processor is fully on.  The function $v_2$ determines the amount of data ready to process based on what is available locally and from nearest neighbors.  In particular, if not enough data is available from neighboring processors, then $f_{i,k} < a_i$.%
\footnote{In \cite{Barnard2019}, there is another parameter $\beta \in [0,1]$ in the definition of $v_2$ that measures the effective strength of coupling between processors.  Here we assume $\beta=1$ and do not consider it any further.}

The macroscopic model in \cite{Barnard2019} is a  continuum approximation for \eqref{eq:ode} that is derived in the limit of infinitely many processors ($i_{\max} \to \infty$) and stages ($k_{\max} \to \infty$), subject to that constraint that $\eta$ is a fixed, positive, finite constant.   
This continuum model is a PDE for a function $\rho$, where $\rho(t,x,z)$ is interpreted as the local density of data at time $t$, processor location $x$, and stage completion variable $z$.  It takes the form
\begin{subequations}\label{eq_model}
\begin{alignat}{2}
\partial_t \rho(t,x,z) + \partial_z \Phi(\rho(t,x,z), \partial_x P(t,x,z) ) = 0 , 
     && \qquad &(t,x,z) \in \mathbb{R}^{+} \times \mathbb{T} \times (0,1), \label{eq_PDE2} \\
\rho(0,x,z)  =\rho_0(x,z),
    && \qquad &(x,z) \in \mathbb{T} \times (0,1), \label{eq_IC}\\
\Phi\big( \rho(t,x,0), \partial_x P(t,x,0) \big)  = F^{\rm{in}}(t,x),
    && \qquad &(t,x) \in \mathbb{R}^{+}\times (0,1), \label{eq_inflow}
\end{alignat} 
\end{subequations}
where $\mathbb{R}^+ = (0, \infty)$; $\mathbb{T}$ is the torus parameterized by $x \in [0,1)$;  $F^{\rm{in}}$ is a scaled, continuous version of $f^{\rm{in}}$; $P$ and $\Phi$ are given by
\begin{subequations}\label{eq:model_aux}
\begin{align}
   P(t,x,z) &= \int_z^1 \rho(t,x,y) \mathrm{d}y + \int_0^t \Phi \big( \rho(s,x,1), \partial_x P (s,x,1) \big) \mathrm{d}s, & \label{recursive} \\
\Phi(\rho, \sigma ) &= \alpha \; w_1\big(w_2(\rho, \sigma ,-\sigma  ));& \label{eq:flux1}  
\end{align}
\end{subequations}
and $w_1$ and $w_2$ are scaled version of $v_1$ and $v_2$:
\begin{subequations}\label{eq:throlling}
\begin{align}
w_1(u)&=\min\left\{ \max \left\{\frac{u}{\rho_*},0 \right \}, 1 \right \} \label{eq:w1} \\
w_2(u,v_1,v_2)&=\min \Big\{ u,\max\{u+\eta v_1,0\},\max\{u+\eta v_2,0\}\Big \}.  \label{eq:w2}
\end{align}
\end{subequations}
The model \eqref{eq_model} is a conservation law that describes the movement of data entering the system at $z=0$ and exiting at $z=1$. The flux $\Phi$ is the composition of the two scaled throttling functions:  First $w_2$ determines the amount of data available based on the local density and the density of left and right neighbors.  Once $w_2$ is specified, $w_1$ computes the effective processing speed assuming a linear ramp, where the parameter $\rho_* > 0$ is the minimum data density needed to operate at full capacity.  The parameter $\eta$ controls the degree of throttling due to the lack of available data in neighboring processors. 

At first glance, \eqref{eq_model} appears to include a recursive
definition of the flux function $\Phi$ due to equation \eqref{recursive}.  However, after integrating \eqref{eq_PDE2} over $(s,z) \in [0,t] \times [0,1]$ and applying \eqref{eq_inflow} and \eqref{eq_IC}, the result can substituted into \eqref{recursive} to find the following formula for $P$: 
\begin{equation}
    P(t,x,z) = \int_0^t F^{\rm{in}}(s,x) \mathrm{d}s + \int_0^1 \rho_0(x,z) \mathrm{d}z - \int_0^z \rho(t,x,z) \mathrm{d}z.
\end{equation}
Alternatively, since 
 $\partial_z P(t,x,z) = - \rho(t,x,z)$, \eqref{eq_PDE2} can be rewritten as a closed Hamilton-Jacobi equation for $P$: 
 \begin{align} \label{eqP}
  \partial_t P  - \Phi( -\partial_z P, \partial_x P) = 0.
 \end{align}
The reformulation  of \eqref{eq_PDE2} in terms of \eqref{eqP} is the basis for the analysis and simulation performed in \cite{Barnard2019}.  In the current work, we proceed by analyzing \eqref{eq_model} directly in terms of $\rho$.

{
}
\section{Qualitative properties of the mathematical model} \label{sec:properties}

In this section, we investigate qualitative properties of \eqref{eq_model}.  We begin by simplifying the expression for $\Phi$ in \eqref{eq:flux1} using the auxiliary function
\begin{equation}\label{eq:W2}
    W_2(s) = \min \{ \overline{w}_2(s), \overline{w}_2(-s) \},
\end{equation}
where 
\begin{equation} \label{eq:w2bar}
    \overline{w}_2(s) = \min\{1,\max\{1+\eta s,0\}\};
\end{equation}
(see Figure \ref{fig:figfunctions}).
\begin{figure}[t!]
	\centering
	\begin{subfigure}[t]{0.49\textwidth}
		\centering
		\begin{tikzpicture}[scale=0.6]
		\begin{axis}[axis lines=middle,enlargelimits,xlabel=\Large{$s$},xtick={0.3},ytick={0.3},
		xticklabels={\Large{$\rho_*$}},
		yticklabels={\Large{$1$}},legend style={at={(0.5,0.5)},anchor=west}]
		\addplot [domain=-0.3:0,samples=200,very thick,black] {0};
		\addplot [domain=0:0.3,samples=200,very thick,black] {x};
		\addplot [domain=0.3:1,samples=200,very thick,black] {0.3*x^0};
		\addplot[thin,gray,dashed] coordinates {(0,0.3) (0.3,0.3)};
		\addplot[thin,gray,dashed] coordinates {(0.3,0.3) (0.3,0)};
		\addplot [domain=0:0.3,samples=200,very thick,black] {x};
		\addplot [domain=0.3:1,samples=200,very thick,black] {0.3*x^0};
		\addlegendentry{\Large{$w_1(u)$}}
		\end{axis}
		\end{tikzpicture}
		\caption{$w_1$; see \eqref{eq:w1}}
	\end{subfigure}%
	\hfill
	\begin{subfigure}[t]{0.49\textwidth}
		\centering
		\begin{tikzpicture}[scale=0.6]
		\begin{axis}[xlabel=\Large{$u$},ylabel=\Large{$v$}]
		\addplot3[surf,
		opacity=0.8,
		samples=50, samples y=50,
		domain=0:1,domain y=-1:1
		]
		{min(x,max(x+y,0),max(x-y,0))};
		\addlegendentry{\Large{$w_2(u,v,-v)$}}
		\end{axis}
		\end{tikzpicture}
		\caption{$w_2(u,v,-v)$ when $\eta=1$; see \eqref{eq:w2}}.
	\end{subfigure}
	\begin{subfigure}[t]{0.49\textwidth}
		\centering
		\begin{tikzpicture}[scale=0.6]
		\begin{axis}[axis lines=middle,enlargelimits,xlabel=\Large{$s$},xtick={-0.3},ytick={1},
		xticklabels={\Large{$-\frac{1}{\eta}$}},
		yticklabels={\Large{$1$}},legend style={at={(0.7,0.5)},anchor=west}]
		
		\addplot [domain=-0.3:0,samples=200,very thick,black] {3.33*x+1};
		\addplot [domain=-0.65:-0.3,samples=200,very thick,black] {0*x^0};
		\addplot [domain=0:0.35,samples=200,very thick,black] {1*x^0};
		\addlegendentry{\Large{$\overline{w}_2(s)$}}
		\end{axis}
		\end{tikzpicture}
		\caption{Function $\overline{w}_2$ defined by equation \eqref{eq:w2bar}.}
	\end{subfigure}
	\hfill
	\begin{subfigure}[t]{0.49\textwidth}
		\centering
		\begin{tikzpicture}[scale=0.6]
		\begin{axis}[axis lines=middle,enlargelimits,xlabel=\Large{$s$},xtick={-0.3,0.3},ytick={1},
		xticklabels={\Large{$-\frac{1}{\eta}$},\Large{$\frac{1}{\eta}$}},
		yticklabels={\Large{$1$}},legend style={at={(0.65,0.5)},anchor=west}]
		
		\addplot [domain=-0.3:0,samples=200,very thick,black] {3.33*x+1};
		\addplot [domain=-0.65:-0.3,samples=200,very thick,black] {0*x^0};
		\addplot [domain=0:0.3,samples=200,very thick,black] {-3.33*x+1};
		\addplot [domain=0.3:0.65,samples=200,very thick,black] {0*x^0};
		\addlegendentry{\Large{$W_2(s)$}}
		\end{axis}
		\end{tikzpicture}
		\caption{Function $W_2$ defined by equation \eqref{eq:W2}.}
	\end{subfigure}
	\caption{Shape of the throttling functions used in the definition of the flux $\Phi$ in ~\eqref{flux2}. } 
	\label{fig:figfunctions}
\end{figure}
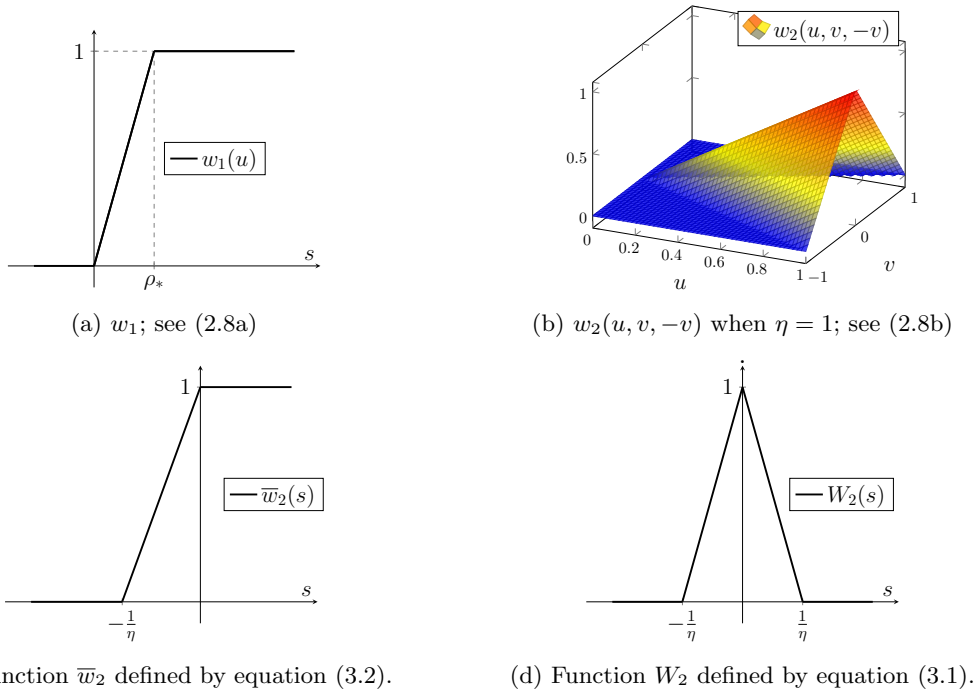
In terms of $w_1$ and $W_2$, $\Phi$ takes the form 
\begin{equation} \label{flux2}
\Phi(\rho, \sigma ) =  
\begin{cases}
\alpha \; w_1\left(\rho W_2\left(\displaystyle{\frac{\sigma }{\rho}}\right) \right) , & \rho \ne 0 \\
0 , & \rho = 0.
\end{cases}
\end{equation}

Next for simplicity, we artificially extend the domain in $z$ to $[0,\infty)$ and assume initial data with compact support.  This assumption allows us to ignore the contribution of the outflow in  \eqref{recursive}.  Under this assumption, the definition of $P$ is modified as follows:
\begin{align}\label{eqNEWP}
{\bf Assumption:} \quad P(t,x,z) &= \int_z^{\infty} \rho(t,x,y) \mathrm{d}y \quad \forall t\geq0, x \in \mathbb{T}, z \in [0,\infty).
\end{align}
Hence, 
in the following we consider the  model \eqref{eq_model} on the  extended domain 
\begin{equation}
    \mathbf{D}:= \R^+ \times \mathbb{T} \times \R^+
\end{equation} 
and introduce the auxiliary variable
\begin{equation}
\label{eq:def_sigma}
    \sigma (t,x,z)  := \partial_x \int_z^{\infty} \rho(t,x,y) \mathrm{d}y, \quad (t,x,z) \in \mathbf{D}.
\end{equation}
The resulting system for $(\rho,\sigma) \colon \mathbf{D}^2 \to \R^2$ is then 
\begin{subequations}
\label{model2} 
\begin{align}
\partial_t \rho + \partial_z \Phi(\rho,\sigma) =0, & \qquad (t,x,z) \in {\bf D}
, \label{model2_rho}  \\ 
\partial_x \rho + \partial_z \sigma = 0,  & \qquad  (t,x,z)\in {\bf D}, \label{model2_sigma} \\
\rho(0,x,z)=\rho_0(x,z), & \qquad  (x,z) \in \mathbb{T} \times \R^+,
\\ \rho(t,x,0)=\rho_b(t,x), \; \sigma(t,x,0)=\sigma_b(t,x), & \qquad (t,x) \in \R^+ \times \mathbb{T}.
\end{align} 
\end{subequations}
The relation of the quantities $(\rho,\sigma)$ to the 
original model \eqref{eq_model}, modified by the previous Assumption, is as follows:
\begin{subequations}
	\begin{align} 
	\sigma_b(t,x) &= \partial_x \int_0^\infty \rho(t,x,y) \mathrm{d}y, \label{eq:sigma_b_consistency}\\
	\Phi \big(\rho_b(t,x),\sigma_b(t,x) \big) &=F^{\rm{in}}(t,x), \\
	\lim_{z \to \infty} \sigma(t,x,z) &= 0. \label{eq:sigma_infty}
	\end{align}
\end{subequations}

\begin{definition}[Weak solution] \label{def:weaksol}
Given  initial 
data $\rho_0 \in L^\infty(  \mathbb{T} \times \R^+ )$ and boundary data $(\rho_b,\sigma_b)\in L^\infty( \R^+ \times  \mathbb{T})^2$, we call $(\rho,\sigma) \in [L^\infty({\bf D})]^2$ a weak solution of \eqref{model2}
if for all smooth compactly supported functions $\varphi:{\bf D}\to \R$ and $\psi:{\bf D}\to \R$, 
\begin{subequations}\label{weak sol 1}
\begin{align}
\int_{{\bf D}} \rho \; \partial_t \varphi 
    + \Phi(\rho,\sigma) \, \partial_z \varphi \; \mathrm{d}z \, \mathrm{d}x \, \mathrm{d}t \label{eq:weak_rho}
    + \int_{\mathbb{T} \times \R^+ }  \rho_0(x,z) \, \varphi(0,x,z) \; \mathrm{d}z \, \mathrm{d}x \nonumber \\ 
+\int_{\R^+ \times \mathbb{T} }  \Phi(\rho_b(t,x),\sigma_b(t,x)) \, \varphi(t,x,0) \; \mathrm{d}x \, \mathrm{d}t  &=0,  \\
 \int_{ {\bf D} }  \rho \; \partial_x \psi + \sigma \; \partial_z \psi \; \mathrm{d}z \mathrm{d}x \mathrm{d}t 
+ \int_{ \R^+ \times \mathbb{T} } \sigma_b(t,x)\, \psi(t,x,0) \; \mathrm{d}x \mathrm{d}t \label{eq:weak_sigma}
&= 0.
 \end{align}
\end{subequations}
\end{definition}

We do not know whether a solution in the sense of Definition  \ref{def:weaksol} exists for \eqref{model2} with general initial and boundary data or whether such a solution is unique and depends continuously on the data.%
\footnote{In \cite{Barnard2019}, a continuous vanishing viscosity solution is established using known results from the mathematical literature on Hamilton-Jacobi equations.  However, translating such a result to an $L^1$-theory for $\rho$ would require $P$ to be absolutely continuous.  While possible, this additional regularity has not yet been determined.}
However, for the particular initial and boundary data given in \eqref{eq:front1}, we establish the existence of special solutions below.

The weak formulation \eqref{weak sol 1} gives rise to a Rankine--Hugoniot jump condition \cite{Dafermos2005}.  Consider
the surface 
\begin{align}
S:=\{ (t,x,z) \subset {\bf D} : z=\zeta(t,x) \} 
\end{align}
for a differentiable function $\zeta \colon \R^+_0 \times \mathbb{T} \to \R^+_0$ and assume that the functions
$\rho$ and $\sigma$ satisfy \eqref{model2} point-wise in the interior of ${\bf D} \backslash S$.
Then standard arguments (see for example \cite[Section 11.1.1]{Evans2010}) can be used to establish that
\begin{equation}\label{RH}
\Phi(\rho_\ell,\sigma_\ell) - \Phi(\rho_r,\sigma_r) = \partial_t \zeta(t,x)  \;  \left( \rho_\ell-\rho_r \right),
\end{equation}
where 
\begin{gather}
    \rho_\ell = \rho(t,x,\zeta(t,x)^-), \qquad \sigma_\ell=\sigma(t,x,\zeta(t,x)^-), \\
    \rho_r = \rho(t,x,\zeta(t,x)^+), \qquad  \sigma_r=\sigma(t,x,\zeta(t,x)^+).
\end{gather}

\subsection{Special solutions for front propagation} \label{sec:weaksolutions}

In this section, we investigate solutions to \eqref{model2} that take the form of fronts;  see Figure \ref{fig1}. As discussed in the introduction, such solutions are important from the point of view of applications.

\begin{figure}
\centering
	\begin{tabular}{cc}
		\includegraphics[width=0.40\textwidth]{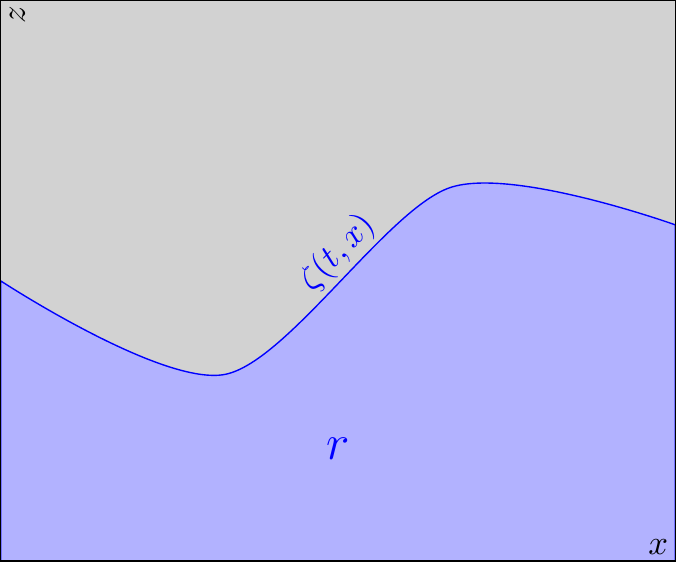} \quad & \quad 
		\includegraphics[width=0.40\textwidth]{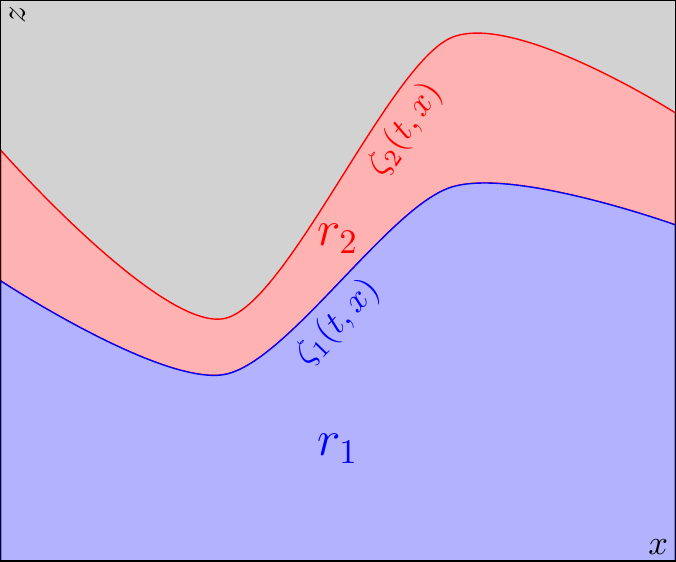} \\
		(a) & (b) 
	\end{tabular}
	\caption{Left: Contour plot of the density $\rho(t,x,z)$ for a fixed time $t.$  Processed data with constant density $r>0$ is depicted in blue up to a stage of completion  $z=\zeta(t,x).$ Zero data (grey) is prescribed for completion stages $z> \zeta(t,x)$, c.f. equation \eqref{eq:front1}. Right: Similar plot of a density with constant values $r_1$ and $r_2$ and regions separated by functions $\zeta_1(t,x)$ and $\zeta_2(t,x)$. }\label{fig1}
\end{figure}

\subsubsection{Existence of Solutions}
In the case of a single front, the initial and boundary conditions for $\rho$ take the form
\begin{equation}\label{eq:front_data_rho}
	\rho_0(x,z) = r  H( \zeta_0(x)-z) \quad \text{and} \quad \rho_b(t,x)=r,
\end{equation}
where $H$ denotes the Heaviside function and $r>0$ is a positive constant.  In addition, we enforce the consistency of $\sigma_b$ as prescribed in \eqref{eq:sigma_b_consistency}:
\begin{equation}\label{eq:front_data_sigma}
    	\sigma_b(t,x)= r \partial_x \zeta(t,x).
\end{equation}

\begin{theorem}\label{lem1}
Let $\rho_*,\eta>0$ be positive constants and let  $\alpha=\alpha(t,x) \in C^1(\R^+ \times \mathbb{T} )$ be a strictly positive function.  Given initial and boundary conditions of the form \eqref{eq:front_data_rho} and \eqref{eq:front_data_sigma}, where
\begin{equation}
\label{ass1}
    0 < r < \rho_*
\end{equation}
and $\zeta_0 \in C^1(\mathbb{T})$ is non-negative, 
assume that there is a $C^1$ solution $\zeta: \R^+_0\times \mathbb{T} \to \R^+$ that satisfies 
\begin{equation}
\label{movement of front}
\partial_t \zeta(t,x)  = \frac{\alpha(t,x)}{\rho_*}
W_2(\partial_x \zeta(t,x)), \quad \zeta(0,x) = \zeta_0(x).
\end{equation}
Then there pair ($\rho$, $\sigma$) given by
\begin{equation} \label{eq:front1}
	\rho(t,x,z) = r H( \zeta(t,x)-z), \quad  \sigma(t,x,z)= r H( \zeta(t,x)-z) \partial_x \zeta(t,x) 
\end{equation}
is a weak solution for \eqref{model2} in the sense of Definition~\ref{def:weaksol} with initial and boundary conditions given by \eqref{eq:front_data_rho} and \eqref{eq:front_data_sigma}.
\end{theorem}
Before proving this result, some remarks are in order. 
\begin{enumerate}
    \item The condition in \eqref{ass1} guarantees that $\frac{r}{\rho_*} W_2( \frac{\sigma}{r} ) < 1$. In this case the flux function $\Phi$ (c.f. \eqref{flux2}) is given by the simplified formula 
    \begin{equation} \label{note:phi}
        \Phi(\rho,\sigma) = 
        \begin{dcases}
         \dfrac{\alpha \rho}{\rho_*}\ W_2\left(\dfrac{\sigma}\rho\right), &
            \rho \ne 0, \\
        0, & \rho = 0.
        \end{dcases}
    \end{equation}
    This formula and the Rankine--Hugoniot condition in \eqref{RH} together form the key components of the existence proof.
    \item While the condition $r<\rho_*$ is sufficient to obtain \eqref{note:phi}, it is not necessary.  Indeed $W_2(\sigma/\rho)=0$ for $\sigma$ sufficiently large; in particular, \eqref{note:phi} holds at any $(x,t)$ such that $\partial_x \zeta(t,x)> \eta^{-1}$.
    \item It is currently open as to whether equation \eqref{movement of front} has a  $C^1$ solution. However, since $\alpha, \rho_*$, and $W_2$ are non-negative, it is clear that any such solution will be non-decreasing.  This fact is consistent with the notion that data is always processed toward completion.
    \item A result similar to Theorem \ref{lem1} can be obtained for initial data separated by a finite number of non-intersecting fronts. Consider for example the case of two fronts (see Figure \ref{fig1}-b):
    \begin{equation} \label{eq:front2_ic}
	    \rho_0(x,z) =r_1 H \big( \zeta_{1,0}(x) -z \big)
	        + r_2 \left[H \big(\zeta_{2,0}(x) -z \big)   
	            - H  \big(\zeta_{1,0}(x) -z \big)\right]
    \end{equation}
    with positive constants $r_1 \not = r_2$ and $r_1<\rho_*$ and $r_2<\rho_*$.  If $0<\zeta_{1,0}(x)<\zeta_{2,0}(x)$ for all $x \in \mathbb{T}$, then there is a weak solution $(\rho,\sigma)$ in the sense of Definition \ref{def:weaksol} of the form  
    \begin{subequations} \label{eq:front2}
        \begin{align} 
	    \rho(t,x,z) &= (r_1 - r_2) H \big( \zeta_1(t,x) -z \big)
	        + r_2 H \big(\zeta_2(t,x) -z \big)   
	    \label{eq:front2_rho}\\
	    \sigma(t,x,z) &=  
	        (r_1 -r_2 ) \partial_x \zeta_1(t,x) H \big( \zeta_1(t,x) -z \big) \\
	        & \qquad + r_2  \partial_x \zeta_2(t,x) H \big(\zeta_2(t,x) -z \big),
	    \label{eq:front2_sigma}
    \end{align}
    \end{subequations}
    provided $\zeta_1(0,x)= \zeta_{1,0}(x)$, $\zeta_2(0,x)= \zeta_{2,0}(x)$, and
    \begin{subequations}
    \begin{align}
	    \partial_t \zeta_1(t,x) 
	    &= \frac{\alpha(t,x)}{\rho_* (r_2-r_1)} 
	        \Big[ r_2 W_2 \Big(\partial_x \zeta_2(t,x) \Big) 
	        - r_1 W_2\Big( \Big(\frac{r_1 - r_2}{r_1}\Big)\partial_x \zeta_1(t,x) + \frac{r_2}{r_1} \partial_x \zeta_2(t,x)  \Big) \Big],\\
	\partial_t \zeta_2(t,x) &= \frac{\alpha(t,x)}{\rho_*} W_2\Big(\partial_x \zeta_2(t,x) \Big).
\end{align}
\end{subequations}
\end{enumerate}  

\begin{proof}[Proof of Theorem \ref{lem1}]
With the $C^1$ assumption on $\zeta$,
$(\rho,\sigma)$ given by equation \eqref{eq:front1} clearly belongs to
$L^\infty({\mathbf{D}})^2$. 
By construction the given solution \eqref{eq:front1} fulfills boundary and initial 
condition  pointwise and hence also in weak form. It therefore remains to verify that $(\rho,\sigma)$ is a weak solution in the 
interior of $\mathbf{D}.$ 
\par
We first verify that $(\rho,\sigma)$ fulfills \eqref{eq:weak_sigma}. For $\psi$ smooth and compactly supported in the interior of $\mathbf{D}$,  we need to show that:
\begin{align}
	\int_{\bf D} \big[ r H(\zeta(t,x)-z) \partial_x \psi(t,x,z) + r H(\zeta(t,x)-z) \partial_x \zeta(t,x) \partial_z \psi(t,x,z) \big] \mathrm{d}z \mathrm{d}x \mathrm{d}t = 0.
\end{align}
By  definition of the Heaviside function, 
\begin{multline}
\int_{\bf D} \big[ r H(\zeta(t,x)-z) \partial_x \psi(t,x,z) + r H(\zeta(t,x)-z) \partial_x \zeta(t,x) \partial_z \psi(t,x,z) \big] \mathrm{d}z \mathrm{d}x \mathrm{d}t \\
 = r \int_{ \R^+_0\times \mathbb{T} } \int_0^{\zeta(t,x)} \big[ \partial_x \psi(t,x,z) + \partial_x \zeta(t,x) \partial_z \psi(t,x,z) \big] \mathrm{d}z \mathrm{d}x \mathrm{d}t  \\
    =  r \int_{ \R^+_0 \times \mathbb{T} } \left[ \int_0^{\zeta(t,x)} \partial_x \psi(t, x,z) \mathrm{d}z + \partial_x \zeta(t,x) \psi(t,x,\zeta(t,x))  \right]  \mathrm{d}x \mathrm{d}t \\
    = r \int_{\R^+_0 \times \mathbb{T} } \frac{\partial}{\partial x} \left( \int_0^{\zeta(t,x)} \psi(t, x,z) dz \right) \mathrm{d}x \mathrm{d}t = 0,
\end{multline}
 where the integral in the last line vanishes due to periodicity of the domain with respect to $x.$

We next verify that $(\rho,\sigma)$ fulfills \eqref{eq:weak_rho}. Let $S=\{ (t,x,z)\subset \mathbf{D}: z= \zeta(t,x) \} \subset \mathbf{D}$ be the surface of the front described by $\zeta$, and let $\mathbf{S}_+:=\{ (t,x,z) \subset \mathbf{D}: z>\zeta(t,x) \}$
and  $\mathbf{S}_-:=\{ (t,x,z) \subset \mathbf{D}: z<\zeta(t,x) \}.$ For $z<\zeta(t,x)$ the explicit form of $(\rho,\sigma)$ implies the following relation 
\begin{align}
	\frac{\sigma(t,x,z)}{\rho(t,x,z)}  = \partial_x \zeta(t,x).
\end{align}
Hence according to \eqref{note:phi},
\begin{align}
\Phi(\rho(t,x,z),\sigma(t,x,z)) =
\begin{dcases}
\frac{\alpha(t,x) r}{\rho_*} W_2(\partial_x \zeta(t,x)) &  (t,x,z) \in \mathbf{S}_- \\
0 & (t,x,z) \in \mathbf{S}_+.
\end{dcases}
\end{align}
Let $\varphi$ be any smooth, compactly supported function in the interior of $\mathbf{D}.$  In $\mathbf{S}_{\pm}$, $\Phi(\rho,\sigma)$ and $\rho$ are constant, and in $\mathbf{S}_{+}$, they both vanish.  Hence, 
\begin{multline}
    \int_{\mathbf{D}} \rho \; \partial_t \varphi 
    + \Phi(\rho,\sigma) \, \partial_z \varphi \; \mathrm{d}z \, \mathrm{d}x \, \mathrm{d}t = 
    \int_{\mathbf{S}_-} r \; \partial_t \varphi 
    +  \frac{ \alpha r }{\rho_*} W_2(\partial_x \zeta) \, \partial_z \varphi \; \mathrm{d}z \, \mathrm{d}x \, \mathrm{d}t  \\
= \int_{\mathbb{T}} \int_{ \{ (t,z): z=\zeta(t,x) \} } 
\varphi \left( - r  \partial_t \zeta +
\frac{ \alpha r }{\rho_*} W_2(\partial_x \zeta)   \right) \frac{1}{
\| (1,\partial_t \zeta) \| } \mathrm{d} A(t,z) \mathrm{d} x,
\end{multline}
where $\mathrm{d} A(t,z)$ denotes the surface measure 
on the curve $\{ (t,z): z=\zeta(t,x) \}$. The final expression above vanishes  due to the definition of $\zeta$ given in \eqref{movement of front}. This finishes the proof.
\end{proof}

\subsubsection{Properties of fronts and their associated weak solutions } \label{sec:discussion}
It is the equation for the front $\zeta$ in \eqref{movement of front} that essentially determines the behavior of the special solutions \eqref{eq:front1}. We assume again that $0 < r<\rho_*$. With the definition of $W_2 $ in \eqref{eq:W2}, this equation takes the form

\begin{equation}\label{eq:front_equation} \partial_t \zeta(t,x) = \frac{\alpha(t,x)}{\rho_*} \times
\begin{dcases}
1 -  \eta \, | \partial_x\zeta(t,x)|,  
    &    | \partial_x\zeta(t,x)| < \eta^{-1}, \\
0,  &  | \partial_x\zeta(t,x)| \geq \eta^{-1}.
\end{dcases}
\end{equation}
In general, we do not know if there exists a solution to  \eqref{eq:front_equation},  due to the discontinuous flux. 
However, we may consider particular solutions related to special  initial conditions $\zeta_0(x).$ Those special solutions are used as test cases for numerical simulations in the next section.

\begin{itemize}

\item 
If for all $(t,x) \in \R^+ \times \mathbb{T}$,
$|\partial_x \zeta(t,x)| < \eta^{-1}$ and $\alpha(t,x)=\alpha$ is constant, then formally $\partial_x \zeta(t,x)$ fulfills the conservation law
\begin{equation}\label{PDEform-x}
\partial_t (\partial_x \zeta)(t,x) +  \frac{\alpha\eta}{\rho_*} \partial_x  | (\partial_x \zeta)(t,x) | = 0, \quad  (\partial_x \zeta)(0,x)=\partial_x \zeta_0(x)
\end{equation}
with flux function $f(u)=C |u|$ and $C= \frac{\alpha\eta}{\rho_*}.$ It is known that there exists a weak (entropic) solution to \eqref{PDEform-x} as long as $\partial_x \zeta_0(x)$ is in $L^\infty$, has locally bounded variation, and is continuous from the left \cite[Proposition 3.1]{Dafermos1972}.
Moreover, if $\partial_x \zeta_0(x)$ is piecewise constant, this solution is given by piecewise constant states separated by a finite number of traveling discontinuities \cite[Lemma 3.2]{Dafermos1972}.  We investigate such piece-wise solutions in Examples 2 and 3 in the numerical results of Section \ref{sec:numerics}.

\item  If  $\alpha$ is constant, and if $0 \leq \partial_x \zeta_0(x) < \eta^{-1}$ locally in $x$, then for sufficiently small $t$,
\begin{equation} \label{eq:solThrottled1}
\zeta(t,x) 
= \zeta_0 \left(x - \eta \frac{\alpha}{\rho_*}
 t \right) + \frac{\alpha}{\rho_*}  t. 
\end{equation}
Similarly if $-\eta^{-1} < \partial_x \zeta_0(x) \leq 0$ locally, then for sufficiently small $t$,
\begin{equation} \label{eq:solThrottled2}
    \zeta(t,x)=\zeta_0\left(x + \eta \; \frac{\alpha}{\rho_*} \; t\right) + \frac{\alpha}{\rho_*} \; t.
\end{equation}
We use these formulas to make comparisons with numerical results in Examples 2-4 in Section~\ref{sec:numerics}.
\item If locally $|\partial_x \zeta_0(x)| \geq \eta^{-1}$, then $\zeta(t,x) = \zeta_0(x)$ for $t$ sufficiently small.
In this case the model predicts that all processors are stalled.  However, the condition on $\zeta_0$ cannot hold globally due to the periodicity assumption in $x$.  We explore the behavior of the model for this scenario in Example 5 in Section~\ref{sec:numerics}.

\end{itemize}

\section{Numerical scheme and computational simulations } \label{sec:numerics}

In this section, we  construct a numerical discretization for the system.  Results of numerical simulations using this discretization are presented
to illustrate the behavior of special front-type solutions discussed in the previous section.   

\subsection{Numerical Discretization} \label{sec:scheme}

Our strategy for approximating \eqref{model2} is based on the relaxed system of equations 
\begin{subequations}\label{model2equivalent}
\begin{align}
\partial_t \rho + \partial_z \Phi(\rho,\sigma) &= 0, \label{model2equivalent_rho}\\
\varepsilon \partial_t \sigma + \partial_x \rho + \partial_z \sigma &= 0, \label{model2equivalent_sigma}
\end{align}
\end{subequations}
where $\epsilon > 0$ is a small parameter.  Formally, \eqref{model2_rho} and \eqref{model2_sigma} are obtained from \eqref{model2equivalent} in the limit as $\epsilon \to 0$. This limit is considered later after a suitable discretization of \eqref{model2equivalent}. 

Before discretizing \eqref{model2equivalent}, we investigate hyperbolicity in the simplified case that
$\alpha(t,x)  > 0$ is constant.%
\footnote{Hyperbolicity will not change if $\alpha$ is space or time--dependent, provided it is sufficiently smooth.} In this case, \eqref{model2equivalent} takes the
non-conservative form 
\begin{equation} \label{eq:nonconservative_form}
\partial_t U + B^x(U) \partial_x U + B^z(U) \partial_z U = 0 ,
\end{equation}
where $U=(\rho,\sigma)$ and 
\begin{equation}
	B^x(U) = \begin{pmatrix}
    	0 & 0 \\[2ex] \displaystyle{\frac{1}{\epsilon}} & 0
    \end{pmatrix}, \quad
    B^z(U) = \begin{pmatrix}
    	\partial_\rho \Phi & \partial_\sigma \Phi \\[2ex] 0 & \displaystyle{\frac{1}{\epsilon}}
    \end{pmatrix}.
\end{equation}
Under the assumption \eqref{ass1} and provided that $\rho>0$, 
the simplified form of $\Phi$ in \eqref{note:phi} implies that
\begin{equation} 
\partial_\rho \Phi(\rho,\sigma) = \frac{\alpha}{\rho_*} \big( W_2(s) - s W^\prime_2(s) \big), \quad \partial_\sigma \Phi(\rho,\sigma) = \frac{\alpha}{\rho_*} W^\prime_2(s),
\end{equation}
where  $s=\sigma/\rho$. 

 The system \eqref{eq:nonconservative_form} is hyperbolic if for any $\xi=(\xi_1, \xi_2) \in \mathbb{R}^2$ the matrix $B(U, \xi) = \xi_1 B^x(U) + \xi_2 B^z(U)$ is diagonalizable in the field of real numbers \cite{Evans2010}.   Because the function $W_2$ is not smooth, $B(U,\xi)$ is not differentiable when $s=0$ or $s = \pm 1/\eta$.  However, away from these points,   
\begin{align}
B(U,\xi)= \left\{  \begin{array}{ll} 
\begin{pmatrix}
	0 & 0 \\[2ex] \displaystyle{\frac{\xi_1}{\epsilon}} & \displaystyle{\frac{\xi_2}{\epsilon}}
	\end{pmatrix}, & |s| > \dfrac{1}\eta , \\[5ex]] 
 \begin{pmatrix}
	\xi_2 \displaystyle{\frac{\alpha}{\rho_*}} & -\xi_2  \sign(s) \eta \dfrac{\alpha}{\rho_*} \\[2ex] \displaystyle{\frac{\xi_1}{\epsilon}} & \displaystyle{\frac{\xi_2}{\epsilon}}
	\end{pmatrix}, & 0< |s| < \dfrac{1}\eta	,	
\end{array} \right.
\end{align}
and the eigenvalues of $B(U,\xi)$ are 
\begin{align}
(\lambda_1,\lambda_2) = \left\{  \begin{array}{ll} 
\left(0, \dfrac{\xi_2}\epsilon \right),
& |s| > \dfrac{1}\eta ,\\[5ex] 
(\lambda_1^*, \lambda_2^*), & 0< |s| < \dfrac{1}\eta,	
\end{array}	\right.
\end{align}
where $(\lambda_1^*,\lambda_2^*)$ are the  roots of the characteristic polynomial
\begin{equation}\label{eq:char_poly}
	p(\lambda)= 
	    \lambda^2 
	    - \xi_2 \left( \frac{\alpha}{\rho_*} + \frac{1}{\epsilon} \right) \lambda 
	    + \frac{\xi_2 \alpha }{\epsilon \rho_*} \left( \xi_2  + \sign(s)  \xi_1 \eta \right).
\end{equation}
For $|s| > \eta^{-1}$, the system is strictly hyperbolic.   For $ 0 < |s| < \eta^{-1}$, the  polynomial $p$ in \eqref{eq:char_poly} has discriminant 
\begin{equation}
    d_\epsilon(\xi_1,\xi_2) = \left(\frac{\xi_2}{\epsilon \rho_*}\right)^2 
    \left[ (\epsilon\alpha - \rho_*)^2 - 4 \sign(s) \epsilon \alpha \rho_* \eta \frac{\xi_1}{\xi_2}\right]
\end{equation}
which, for $\xi_2 \ne 0$, is negative (so that $\lambda_1^*$ and $\lambda_2^*$ have non-zero imaginary component) if and only if
\begin{equation}
    \sign(s)\frac{\xi_1}{\xi_2} >  \frac{(\epsilon\alpha - \rho_*)^2}{4 \epsilon \alpha \rho_* \eta}.
\end{equation}
For any $\epsilon >0$, the non-hyperbolic region $\{ (\xi_1,\xi_2) : d_\epsilon(\xi_1,\xi_2)  < 0 \}$ is nontrivial; 
see Figure \ref{fig:hyp}.  However, for any fixed $(\xi_1,\xi_2)$, there exists an $\epsilon^*(\xi_1,\xi_2)$ 
small enough that $d_\epsilon(\xi_1,\xi_2)  \geq 0$ for all $\epsilon < \epsilon^*(\xi_1,\xi_2)$.  Hence as $\epsilon \to 0^+$, the non-hyperbolic region vanishes.  The numerical method below is constructed by discretizing the relaxation system in \eqref{model2equivalent} and then formally setting $\epsilon=0$.  This fact partially justifies the use of \eqref{model2equivalent} even though it is not everywhere hyperbolic when $\epsilon > 0$.
\begin{figure}\centering
	\begin{tabular}{cc}
		\includegraphics[width=0.45\textwidth]{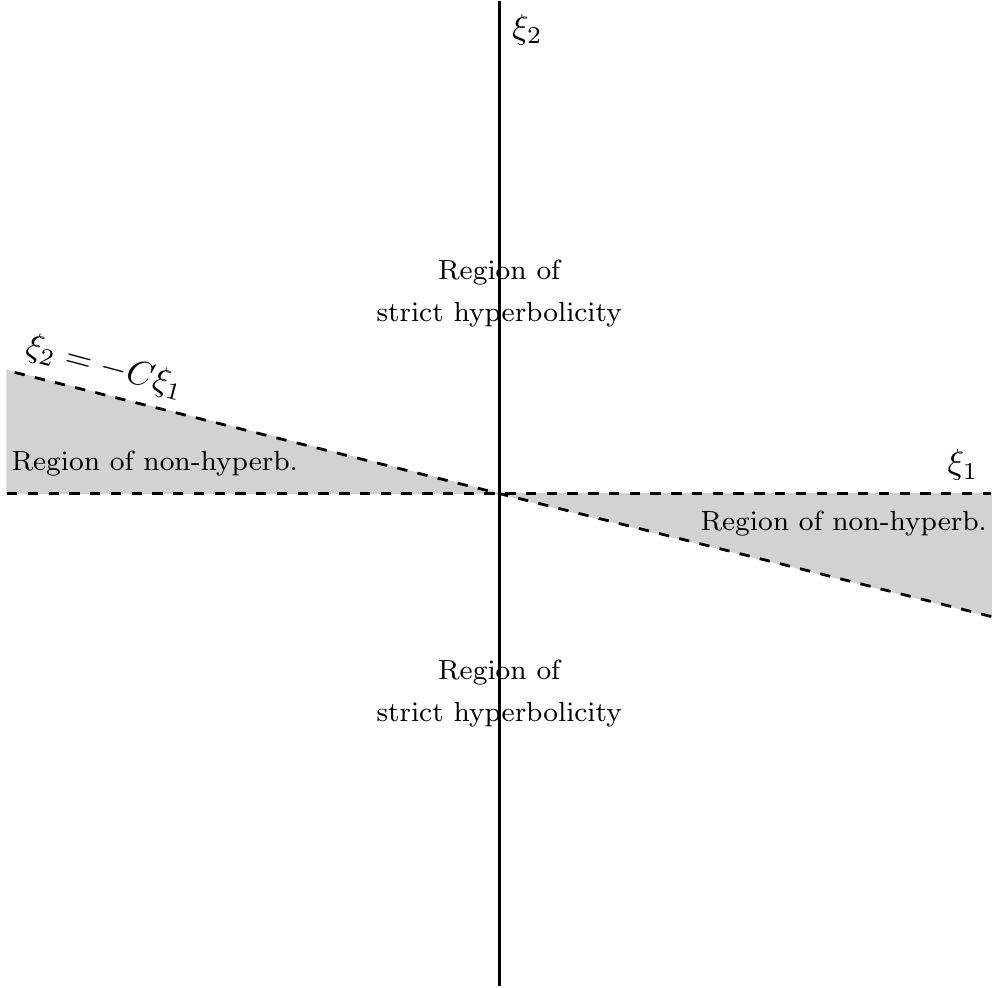} &
		\includegraphics[width=0.45\textwidth]{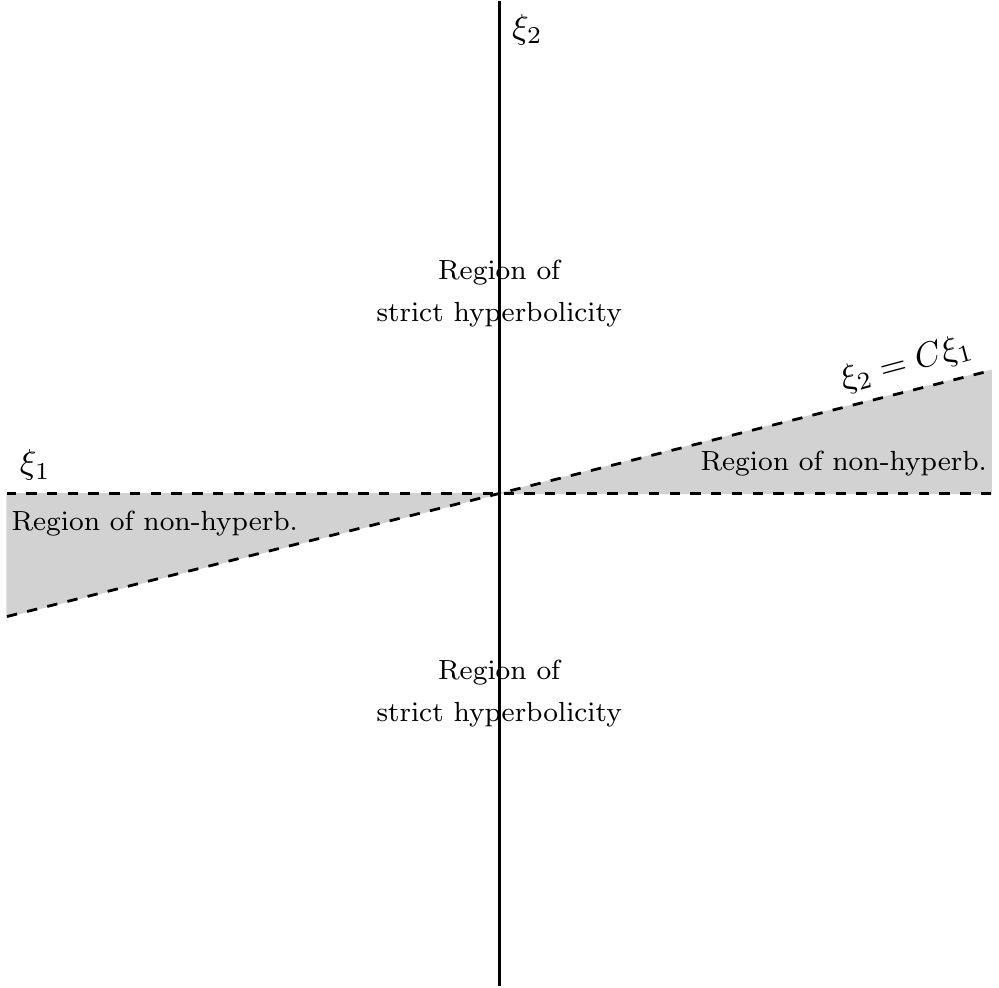} \\
		(a) Case: $-\frac{1}{\eta} < s < 0$ & (b) Case: $0 < s < \frac{1}{\eta}$
	\end{tabular}
	\caption{
	Domain of hyperbolicity of system~\eqref{model2equivalent}. In both cases, the non-hyperbolic region (gray) is bounded by the $\xi_1$-axis and by $\xi_2= \pm C \xi_1$ with slope $C=\frac{4 \epsilon \alpha \rho_* \eta}{\left(\epsilon \alpha -\rho_*\right)^2} \to 0$ as $\epsilon \to 0^+$.}\label{fig:hyp}
\end{figure} 

To derive a numerical scheme for \eqref{model2equivalent}, we first discretize in time using a first-order method. We are eventually interested in the $\epsilon \to 0$ 
limit of the fully discretized scheme; hence the equation for $\rho$ in \eqref{model2equivalent_rho} is treated explicitly, but the equation for $\sigma$ in \eqref{model2equivalent_sigma} is treated implicitly. Let
$\rho^n(x,z) \simeq \rho(n \Delta t,x,z)$ and $\sigma^n(x,z) \simeq \sigma(n \Delta t,x,z)$ be the discrete
approximations of $\rho$ and $\sigma$, respectively, at time $n \Delta t.$ We set
\begin{subequations}\label{model2equivalent_time_disc}
	\begin{align}
\rho^{n+1} &=  \rho^{n} -  \Delta t \, \partial_z \Phi(\rho^n,\sigma^n),  \label{model2equivalent_time_disc_rho} \\
	\sigma^{n+1} &=  \sigma^{n}  -  \frac{\Delta t}{\varepsilon} \left( \partial_x \rho^{n+1} + \partial_z \sigma^{n+1} \right) \label{model2equivalent_time_disc_sigma}.
\end{align}
\end{subequations}

To  discretize \eqref{model2equivalent_time_disc} in $x$ and $z$, we introduce a bounded  computational domain $(x,z) \in [0,1)\times [0,1]$, which is divided into $N^x\times N^z$  uniform cells of size $\Delta x \times \Delta z$. The cell  centers are denoted by  $(x_i,z_j)$, for $i \in \{1,\dots,N^x\}$ and $j \in \{1,\dots,N^z\}$; and the cell edges are denoted by $x_{i+\frac{1}2} = (i + \frac{1}2) \Delta x$  and $z_{j+\frac{1}2}= (j + \frac{1}2) \Delta z$, respectively, for $i \in \{0,\dots,N^x\}$ and $j \in \{0,\dots,N^z\}$.  The approximate cell averages of $\rho^n$ and $\sigma^n$ on the 
 cell centered at $(x_i,z_j)$ are denoted by
 \begin{align}
 \overline{R}^n_{ij}  \simeq \frac{1}{\Delta x \Delta z} \int_{x_{i-1/2}}^{x_{i+1/2}} \int_{z_{j-1/2}}^{z_{j+1/2}} \rho^n(x,z) \mathrm{d}x \mathrm{d}z 
 	\intertext{and}
 \overline{S}^n_{ij} \simeq \frac{1}{\Delta x \Delta z} \int_{x_{i-1/2}}^{x_{i+1/2}} \int_{z_{j-1/2}}^{z_{j+1/2}} \sigma^n(x,z) \mathrm{d}x \mathrm{d}z,
 \end{align}
 respectively. 
 
 We use  a Lax-Friedrichs approximation of  $\partial_z \Phi$ in  \eqref{model2equivalent_time_disc_rho} and in \eqref{model2equivalent_time_disc_sigma} a centered discretization for $\partial_x \rho$ and a one-sided discretization of $\partial_z \sigma$.  The resulting fully discrete scheme is 
\begin{subequations}
\label{model2equivalent_fully_discrete}
	\begin{align}
		\overline{R}_{ij}^{n+1} &= \overline{R}_{ij}^{n}  - \frac{\Delta t}{\Delta z} \left(F_{i,j+1/2}^n - F_{i,j-1/2}^n \right)+ \frac14  \left(\overline{R}_{i+1,j}^{n}-2\overline{R}_{ij}^{n}+\overline{R}_{i-1,j}^{n}\right) ,\\
  \overline{S}_{ij}^{n+1} &= \overline{S}_{ij}^n - \frac{\Delta t}{2\epsilon\Delta x}
  \left(\overline{R}_{i+1,j}^{n+1}-\overline{R}_{i-1,j}^{n+1}\right) - \frac{\Delta t}{\epsilon \Delta z}
  \left(\overline{S}_{i,j+1}^{n+1} - \overline{S}_{ij}^{n+1}\right) , \label{eq:discr sigma 2}
	\end{align}
\end{subequations}
where
 
\begin{equation}
F_{i,j+\frac{1}2}^n = \frac12\left( \Phi(\overline{R}_{i,j+1}^n,\overline{S}_{i,j+1}^n)+\Phi(\overline{R}_{ij}^n,\overline{S}_{ij}^n) - a (\overline{R}_{i,j+1}^n-\overline{R}_{ij}^n)\right).
\end{equation}

The monotonicity of the numerical flux $F_{i,j+\frac{1}2}^n $ is guaranteed provided that 
\begin{align}\label{eq:CFL}
\max_{(R,S) }  \Bigl\| \partial_\rho \Phi(R,S) \Bigr\| =: a \leq \frac{\Delta z}{\Delta t}. 
\end{align} 
This condition motivates a dynamic choice of $\Delta t$: At each time level $t^n$ we set  $\Delta t=\frac{\Delta z}{a^n}
$ for 
\begin{align}
  a^n = \max\limits_{(i,j): 1\leq i\leq N^x, 1 \leq j\leq N^z } 
\Bigl\| \partial_\rho \Phi(\overline{R}^n_{ij}, \overline{S}^n_{ij})   \Bigr\| .  
\end{align}

 The right-biased stencil in the discretization of $\partial_z \sigma$  in equation \eqref{eq:discr sigma 2} warrants some discussion.   Indeed, given that the boundary condition for $\sigma$ is given at $z=0$, it seems more natural to use a left-biased stencil.  However, such an approach does not allow the boundary condition to be computed implicitly via the consistency relation in \eqref{eq:sigma_b_consistency}.  However, if $\rho$ has compact support, then \eqref{eq:sigma_infty} holds.  As long as the support of the solution does not reach the  boundary of the computational domain at $z=1$, this condition can be enforced there, in an implicit fashion.  This approach is consistent with the fact that $\sigma$ at a given $z_*$ is determined entirely by $\rho$ at values of $z > z_*$.  Moreover, numerical calculations confirm that enforcing the boundary condition in this way yields stable results.

The final scheme used in Section~\ref{sec:simulations} is  given by the formal limit of \eqref{model2equivalent_fully_discrete} in the limit $\epsilon \to 0$:
\begin{subequations}
\label{eq: fully_discrete_scheme}
\begin{align}
\overline{R}_{ij}^{n+1} - \overline{R}_{ij}^{n} &= - \frac{\Delta t}{\Delta z} \left[ F_{i,j+1/2}^n - F_{i,j-1/2}^n \right] + \frac14 \left(\overline{R}_{i+1,j}^{n}-2\overline{R}_{ij}^{n}+\overline{R}_{i-1,j}^{n}\right), \\
\overline{S}_{ij}^{n+1} &= \frac{\Delta z}{2\Delta x} \left(\overline{R}_{i+1,j}^{n+1} - \overline{R}_{i-1,j}^{n+1} \right) + \overline{S}_{i,j+1}^{n+1}.
\end{align}
\end{subequations}
This scheme above must be accompanied by boundary and initial conditions. The initial values  $\overline{R}_{ij}^0$ are
obtained by integration of the initial data $\rho_0$:
\begin{align}
\overline{R}_{ij}^0 = \frac{1}{\Delta x \Delta z} \int_{x_{i-1/2}}^{x_{i+1/2}} \int_{z_{j-1/2}}^{z_{j+1/2}} \rho_0(x,y) \mathrm{d}x \mathrm{d}z.
\end{align} 
The cell averages $\overline{S}_{ij}^0$ are obtained by applying a midpoint rule to the consistency relation \eqref{eq:sigma_b_consistency}: 
\begin{align}
\overline{S}_{ij}^0 &= \frac{1}{\Delta x} \sum\limits_{\ell=j}^{N^z}  \rho_0(x_{i+1/2},z_j) - \rho_0(x_{i-1/2},z_j).
\end{align}
For $\overline{S}_{ij}^n$, zero boundary conditions are prescribed at $j=N^z$ consistent with the  right-based stencil in the discretization of $\partial_z \sigma$: 
\begin{align}
\overline{S}_{i,N^z}^n &= 0.
\end{align}
 For $\overline{R}_{ij}^n$,  boundary conditions at $z=0$ are computed from  $\rho_{b}(t).$ For the special solutions discussed in the previous section $\rho_{b}=r$ and hence  
 \begin{align}
 \overline{R}_{i, 1}^n &= r.
 \end{align}
Due to the central discretization of $\partial_z \rho$, boundary conditions for $\overline{R}^n_{ij}$ at $j=N^z$ 
need to be prescribed.  Since in the simulation no data reaches this boundary we implement
 Neumann boundary conditions 
 \begin{align}
\overline{R}_{i, N^z}^n &=  \overline{R}_{i, N^z-1}^n. 
\end{align}

\subsection{Simulation results } \label{sec:simulations}

\paragraph{Example 1: validation with the microscopic model.}

We consider a test problem that has been used to to validate the macroscopic model \eqref{model2} and the numerical method \eqref{eq: fully_discrete_scheme} by comparing with a simulation of the microscopic model introduced in  \cite{Barnard2019}.  For this problem, $\rho_*=1.0$ and $\eta=1.0$.  The initial and boundary conditions are given by 
\begin{equation}
\rho_0(x,z) = 1.5\left(\sin( 2\pi z) \right)^6 \,\chi_{[0,0.5]}(z),
 \quad \rho_{b}(t,x)=0, 
 \quad \sigma_{b}(t,x)=0,
\end{equation}
and the processor speed is
\begin{equation}
    \alpha(x) = 1-0.4(\sin(\pi x))^2.
\end{equation}
Both models are simulated up to a final time $T_{\text{fin}} = 0.5.$

The $(x,z)$ domain for the macroscopic model is discretized using $N^x=N^z=800$ cells, and simulated with the algorithm in \eqref{eq: fully_discrete_scheme}.  In order to be consistent with the choice of $\eta$, the microscopic model uses $i_{\max} = 800$ processors and $k_{\max} = 800$ stages.    The microscopic model is integrated
in time using the explicit Euler method.

In Figure~\ref{fig:smooth} we show the initial density and the final density profiles for both models at time $T_{\text{fin}}$.  We observe qualitative agreement between the microscopic and the macroscopic solutions.   In both cases, the data around $x=0.5$ is processed more slowly than the data in the rest of the domain, due to a slower processing rate $\alpha$ there.

\begin{figure}[ht!]
\centering
    \begin{subfigure}[t]{0.24\textwidth}
            \centering
            \includegraphics[width=\textwidth]{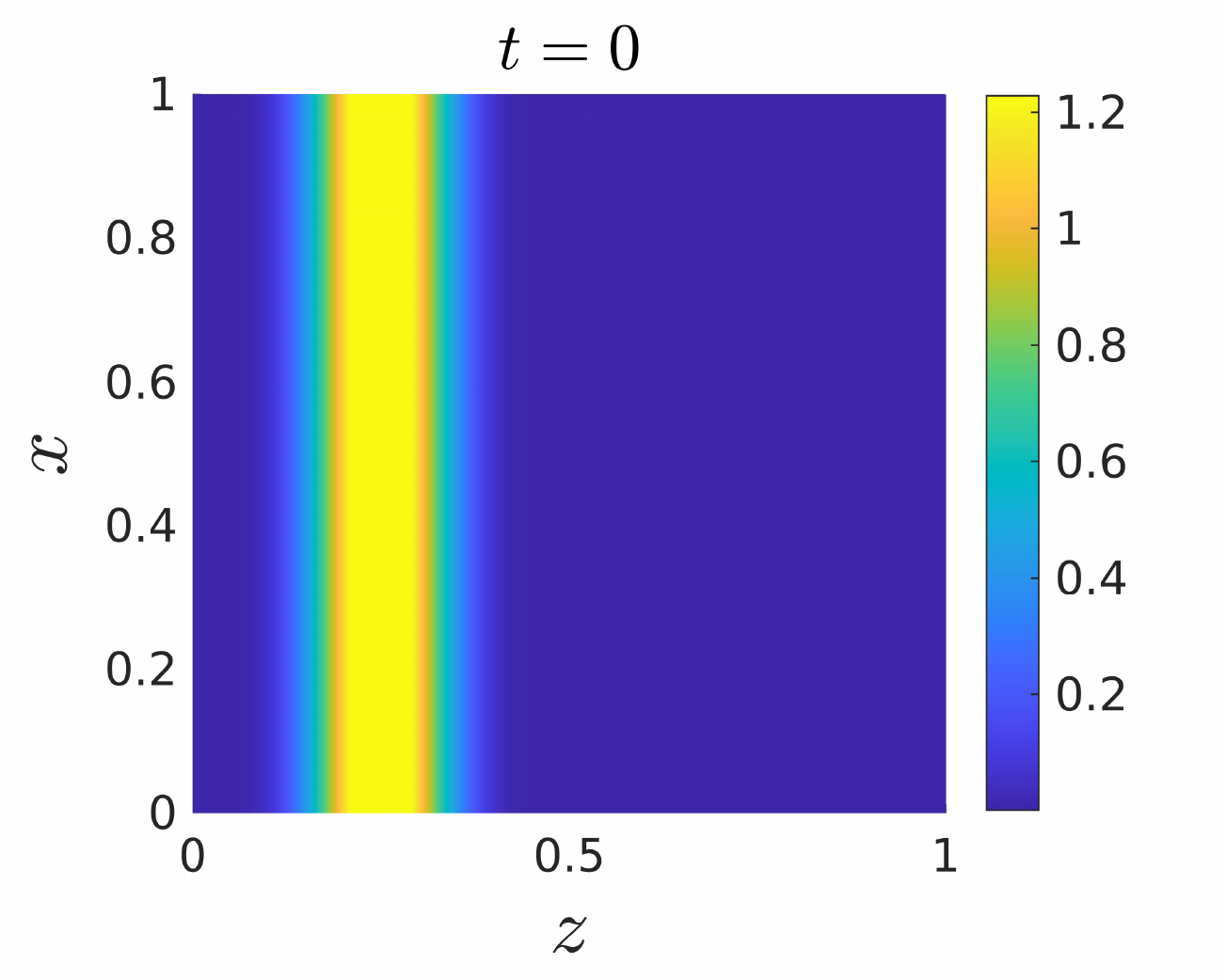}
    \label{fig:smooth_ic}
    \caption{initial density profile}
    \end{subfigure}
    \begin{subfigure}[t]{0.24\textwidth}
            \centering
            \includegraphics[width=\textwidth]{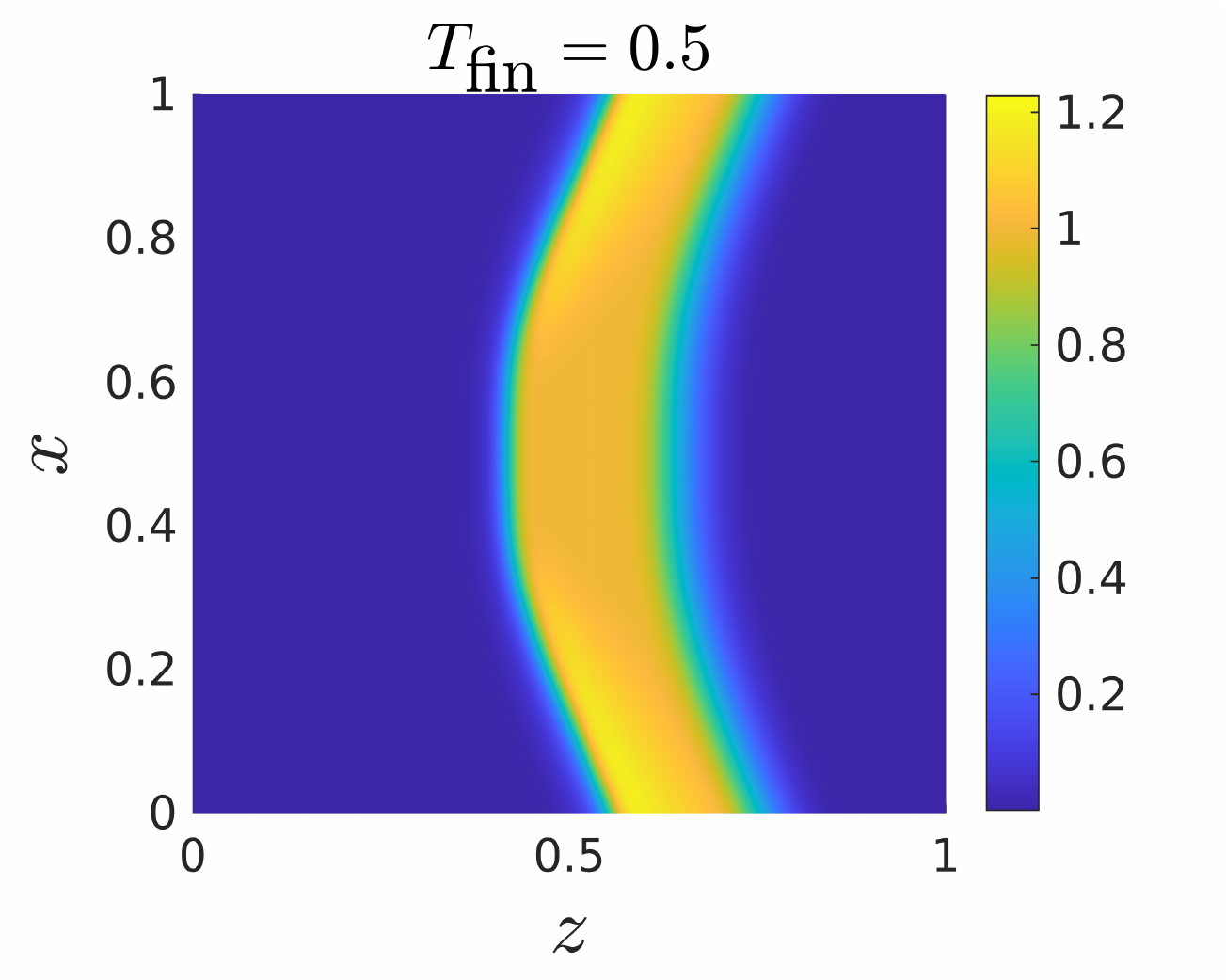}
    \label{fig:smooth_macro}
    \caption{macroscopic model
    }
    \end{subfigure}
    \begin{subfigure}[t]{0.24\textwidth}
            \centering
            \includegraphics[width=\textwidth]{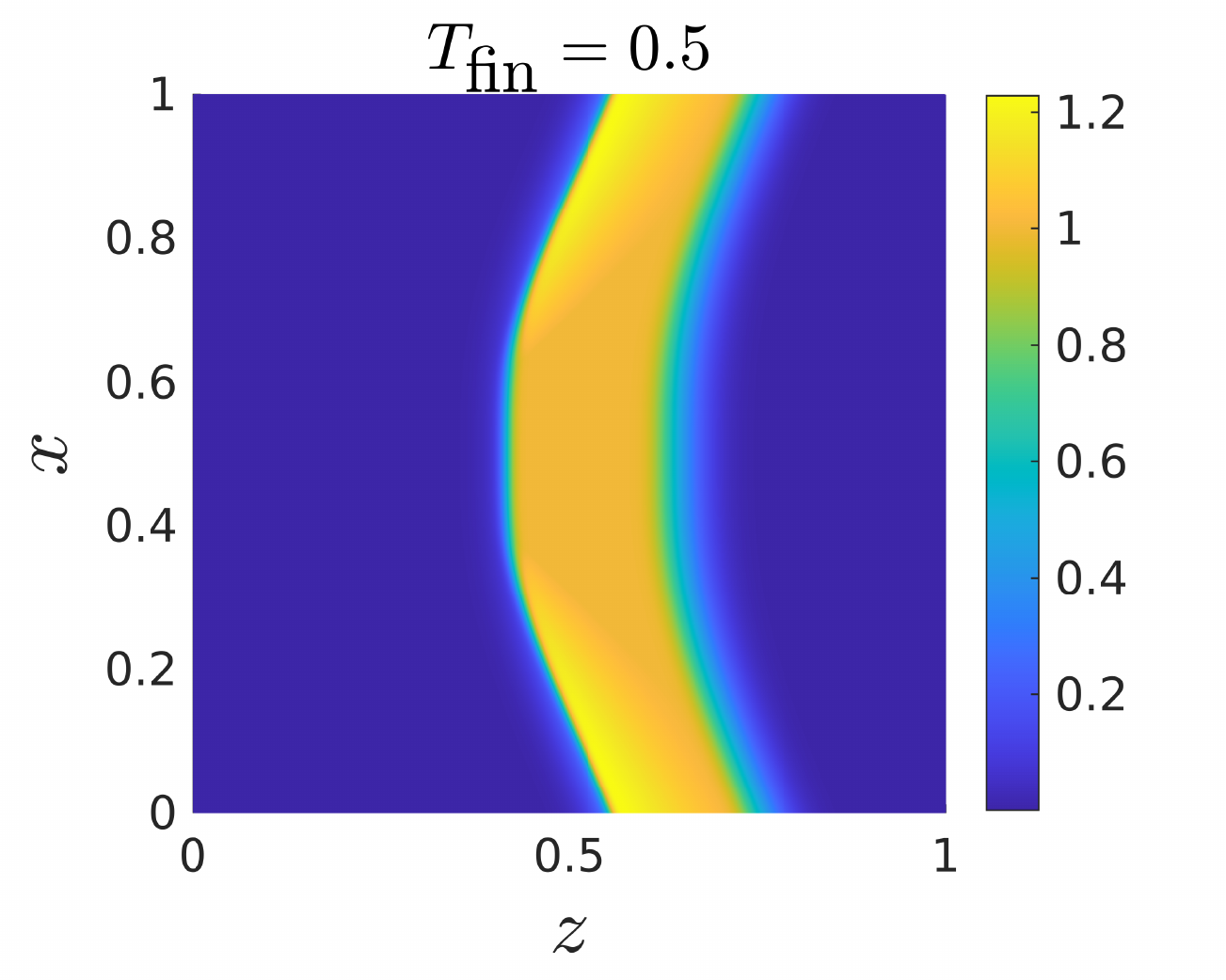}
    \label{fig:smooth_micro}
    \caption{microscopic model
    }
    \end{subfigure}
    \begin{subfigure}[t]{0.24\textwidth}
	\centering
	\includegraphics[width=\textwidth]{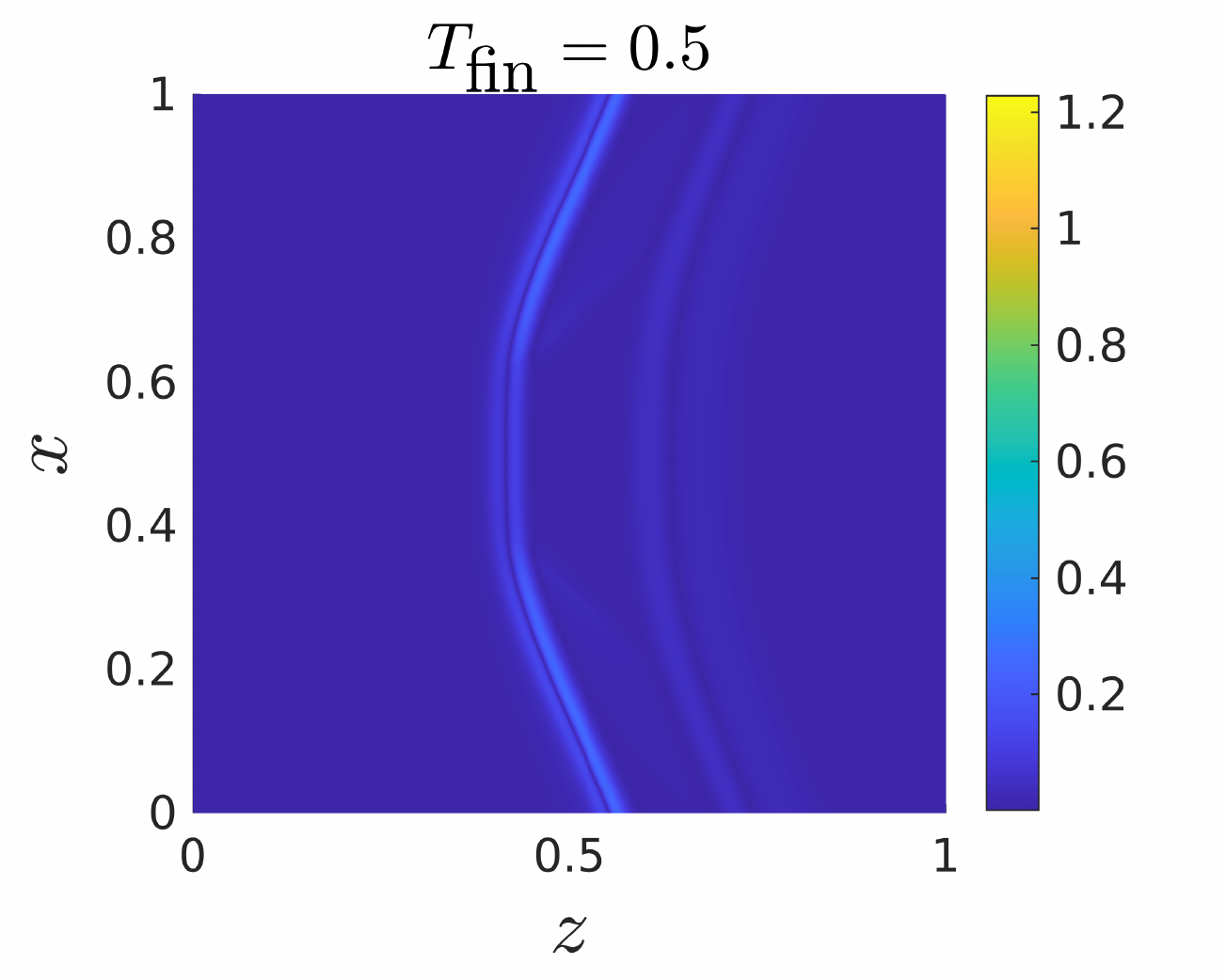}
	\label{fig:smooth_macro-micro}
	\caption{macro minus micro}
	\end{subfigure}
	\caption{Density profiles for Example 1. }
	\label{fig:smooth}
\end{figure}

In the remaining examples of this subsection, we explore the behavior of front-type solutions.  Unless otherwise stated the following parameters are used in all simulations:  
\begin{align} 
r = 0.5, \quad \rho_* = 0.8, \quad \alpha = 0.1, \quad T_{\text{fin}} =2.0.
\end{align}
The value of $T_{\text{fin}}$ ensures that data does not reach $z=1$; in particular the assumption in \eqref{eqNEWP} holds.  
These parameter choices, along with the initial and boundary conditions ensure that
\begin{equation}
   0 \leq \rho(t,x,z)  < \rho_*, \quad   \text{for all} ~(t,x,z) \in \mathbf{D},
\end{equation}
in which case the flux $\Phi$ is given by the simplified formula in \eqref{note:phi}.

\paragraph{Example 2:  front with constant profile.}
\label{ex:const_front}We consider the case of an initial constant front, namely
\begin{equation}
	\zeta_0(x) = \zeta_0  =0.2, \quad \forall\,x\in \mathbb{T},
\end{equation}
which provides  a trivial example for a piece-wise constant solution $\partial_x \zeta$ to \eqref{PDEform-x}.
Moreover, the evolution of this front is given explicitly in \eqref{eq:solThrottled1} and \eqref{eq:solThrottled2}:
\begin{equation}
    \label{eq: constant_front_soln}
    \zeta(t,x) = \zeta_0 + \frac{\alpha}{\rho_*}t.
\end{equation}
In this example $\eta$ is set to $0.5$ although the solution in \eqref{eq: constant_front_soln} is independent of this choice.

Numerical results are shown in Figure~\ref{fig:constfront}.  Since the processing rate $\alpha$ is constant the front moves with the same speed for each $x\in[0,1)$, as expected. 
Moreover, the analytical and the numerical solutions provide the same position of the front at final time.

\begin{figure}[ht!]
\centering
    \begin{subfigure}[t]{0.45\textwidth}
            \centering
            \includegraphics[width=\textwidth]{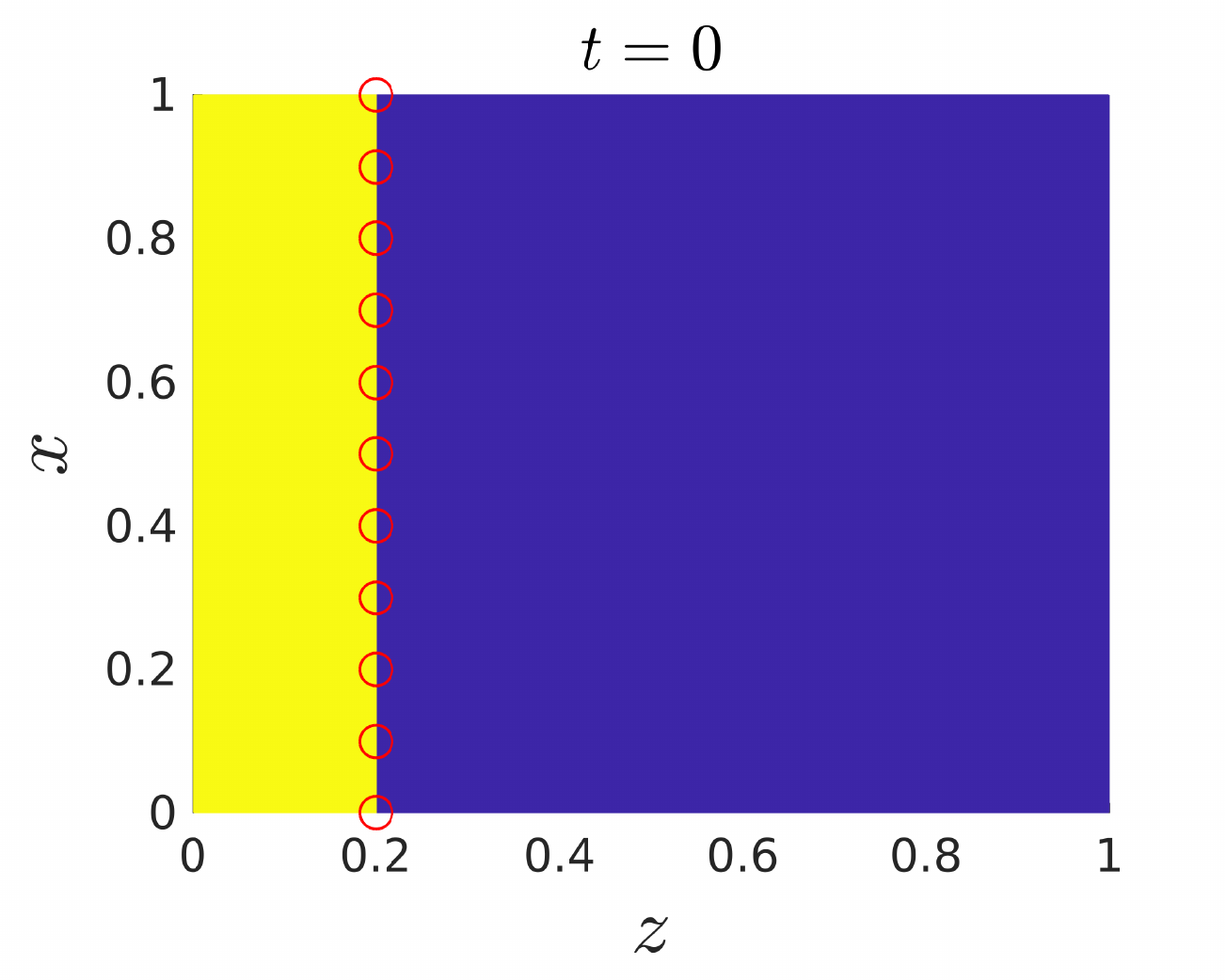}
    \label{fig:constfrontA}
    \caption{initial density profile}
    \end{subfigure}
    \begin{subfigure}[t]{0.45\textwidth}
            \centering
            \includegraphics[width=\textwidth]{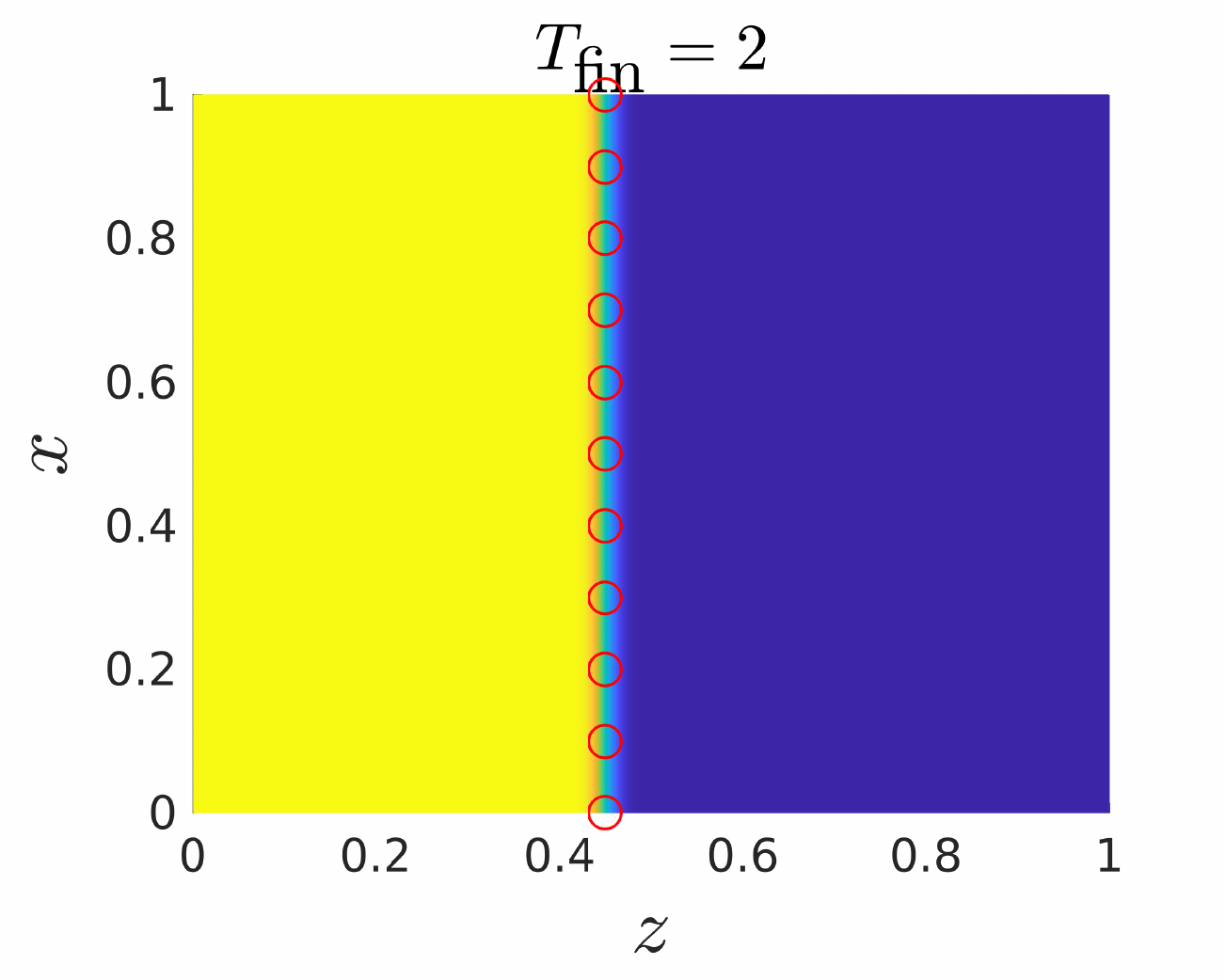}
    \label{fig:constfrontB}
    \caption{density profile at $t = T_{\text{fin}}=2$}
    \end{subfigure}\\
    [3ex]
    \begin{subfigure}[t]{0.45\textwidth}
            \centering
            \includegraphics[width=\textwidth]{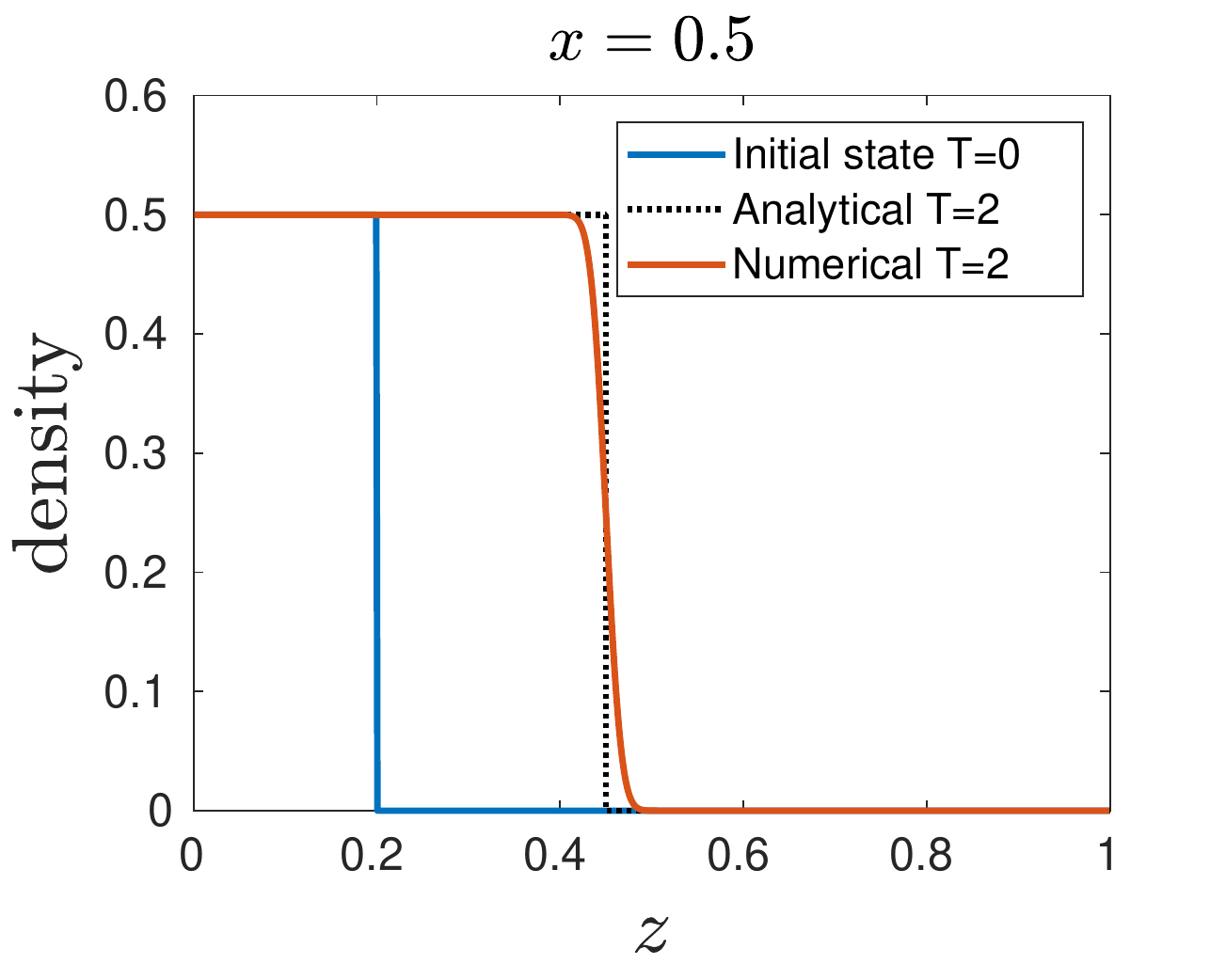}
    \label{fig:constfrontC}
    \caption{slice at $x=0.5$}
    \end{subfigure}
    \begin{subfigure}[t]{0.45\textwidth}
            \centering
            \includegraphics[width=\textwidth]{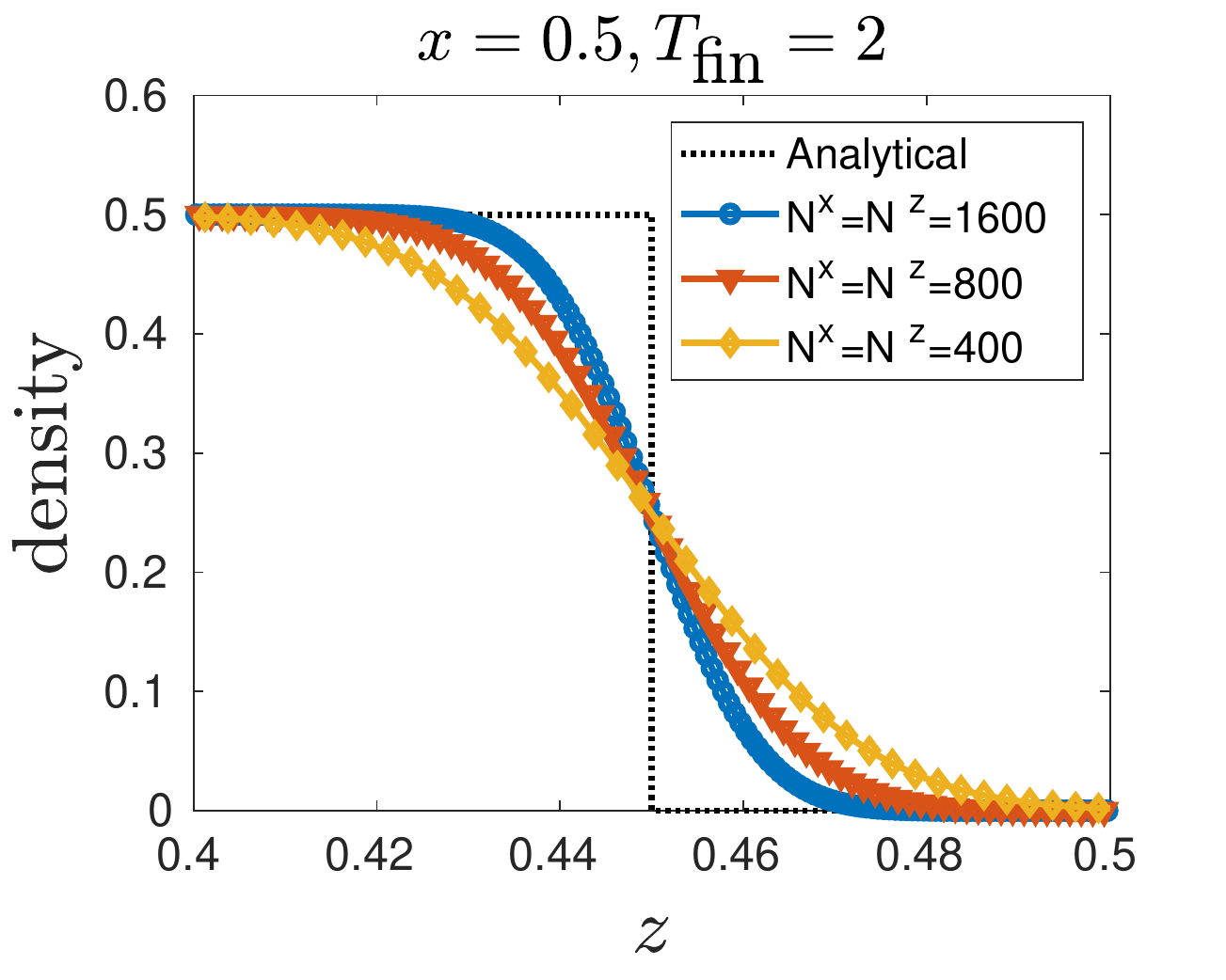}
    \label{fig:constfrontD}
    \caption{slice at $x=0.5$, $t = T_{\text{fin}}=2$}
    \end{subfigure}\\
    [3ex]
	\caption{ Density profiles based on the a simulation of the constant front solution \eqref{eq: constant_front_soln} with the scheme in \eqref{eq: fully_discrete_scheme}. The simulations in panels (a)-(c) use $N^x=N^z=800$ computational cells, while in panel (d), numerical solutions are computed at different refinement levels.  In the top row, the analytical front is marked by circles. }  
	\label{fig:constfront} 
\end{figure}

 \paragraph{Example 3: \textsf{V}-shaped front with small $\eta$.}  We consider an initial condition characterized by a \textsf{V}-shape front under the condition that $\eta | \partial_x \zeta | <1$; cf. \eqref{eq:front_equation}. This condition models the physical situation in which processors are  not initially at the same stage of completion.  Such a situation may be realized if a local group of processors is, at some previous time, slower than the rest \cite{Barnard2019}.  Specifically, we let
\begin{equation} \label{eq:frontV}
	\zeta_0(x) = (1-2z_0)\left |x - 0.5 \right| + z_0
\end{equation}
where $z_0 = \zeta_0(0.5) < 0.5$ is a fixed parameter.

Based on the theoretical findings in \cite{Dafermos1972} for \eqref{PDEform-x}, we expect that the front at later times will be piecewise linear.  Moreover, the formulas in \eqref{eq:solThrottled1} and  \eqref{eq:solThrottled2} can be used to generate the local solution analytically.  Indeed, under the assumption $\eta | \partial_x \zeta | <1$,
a small calculation shows that the local, short time solution away from $x=0.5$ is given by
\begin{equation}
\label{eq:zeta_vfront_local}
\zeta_{\rm{loc}}(t,x) = \zeta_0(x) + \frac{\alpha}{\rho_*} \Big(1 - \eta |\partial_x \zeta_0| \Big) t.
\end{equation}
where $|\partial_x \zeta_0|=(1-2z_0)$.
However at  $x=0.5$,  Theorem~\ref{lem1} does not apply, since the front is not in differentiable there;  hence  \eqref{eq:solThrottled1} and  \eqref{eq:solThrottled2} are not valid. We conjecture that at this point, 
	the front moves forward at full speed $\alpha / \rho_*$ and, as it catches up with neighboring points in the front, these points also move forward at full speed.  This means that the global front solution  is given by
\begin{equation}
\label{eq:zeta_vfront_global}
\zeta(t,x) = \max\left\{  \zeta_{\rm{loc}}(t,x), \, z_0 + \frac{\alpha}{\rho_*} t, \right \} .
\end{equation}

\begin{figure}[ht!]
\centering
    \begin{subfigure}[b]{0.32\textwidth}
            \centering
            \includegraphics[width=\textwidth]{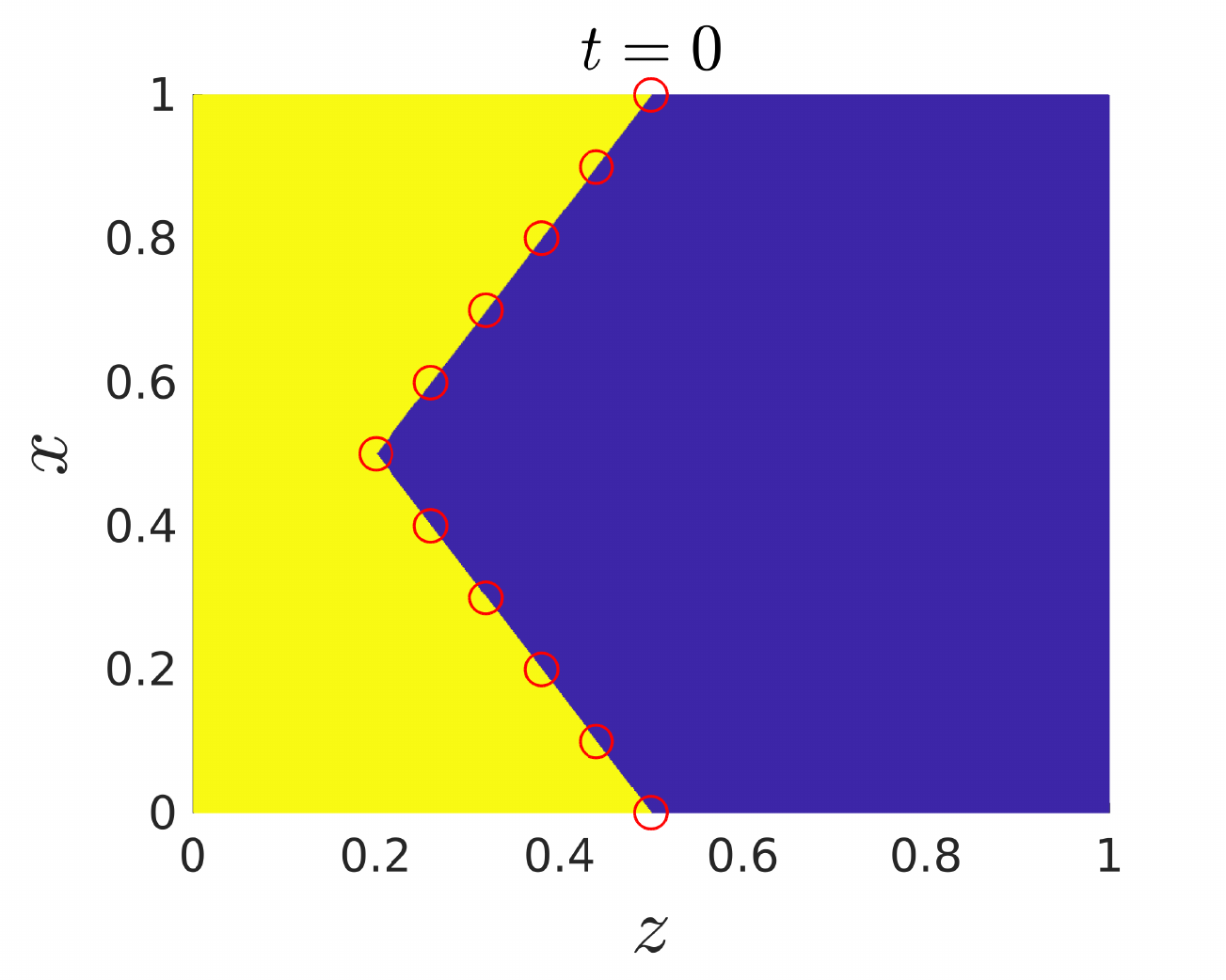}
    \label{fig:nonconstfrontA}
    \caption{initial density profile}
    \end{subfigure}
    \begin{subfigure}[b]{0.32\textwidth}
            \centering
            \includegraphics[width=\textwidth]{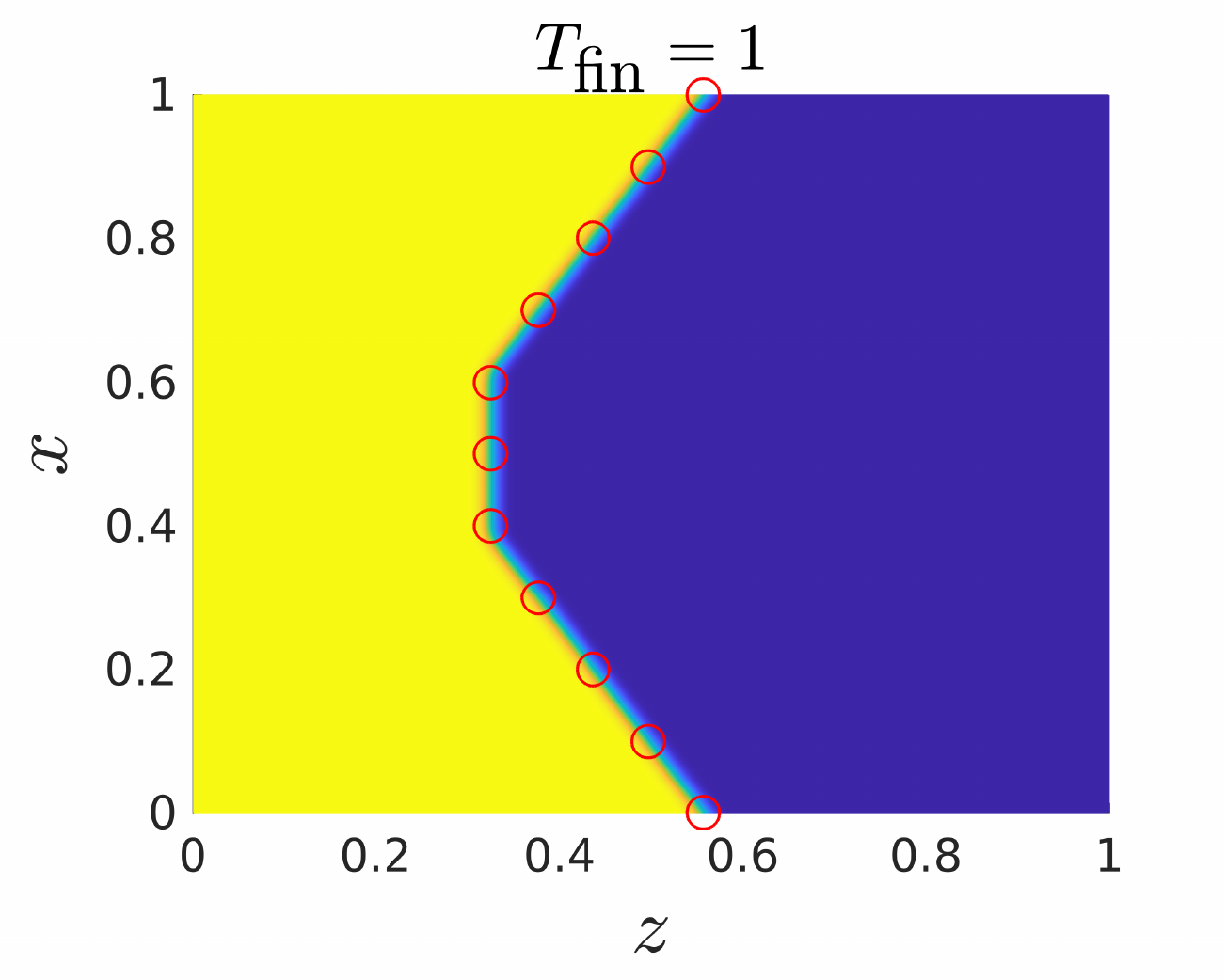}
    \label{fig:nonconstfrontB}
    \caption{density profile at $t=1$}
    \end{subfigure}
    \begin{subfigure}[b]{0.32\textwidth}
            \centering
            \includegraphics[width=\textwidth]{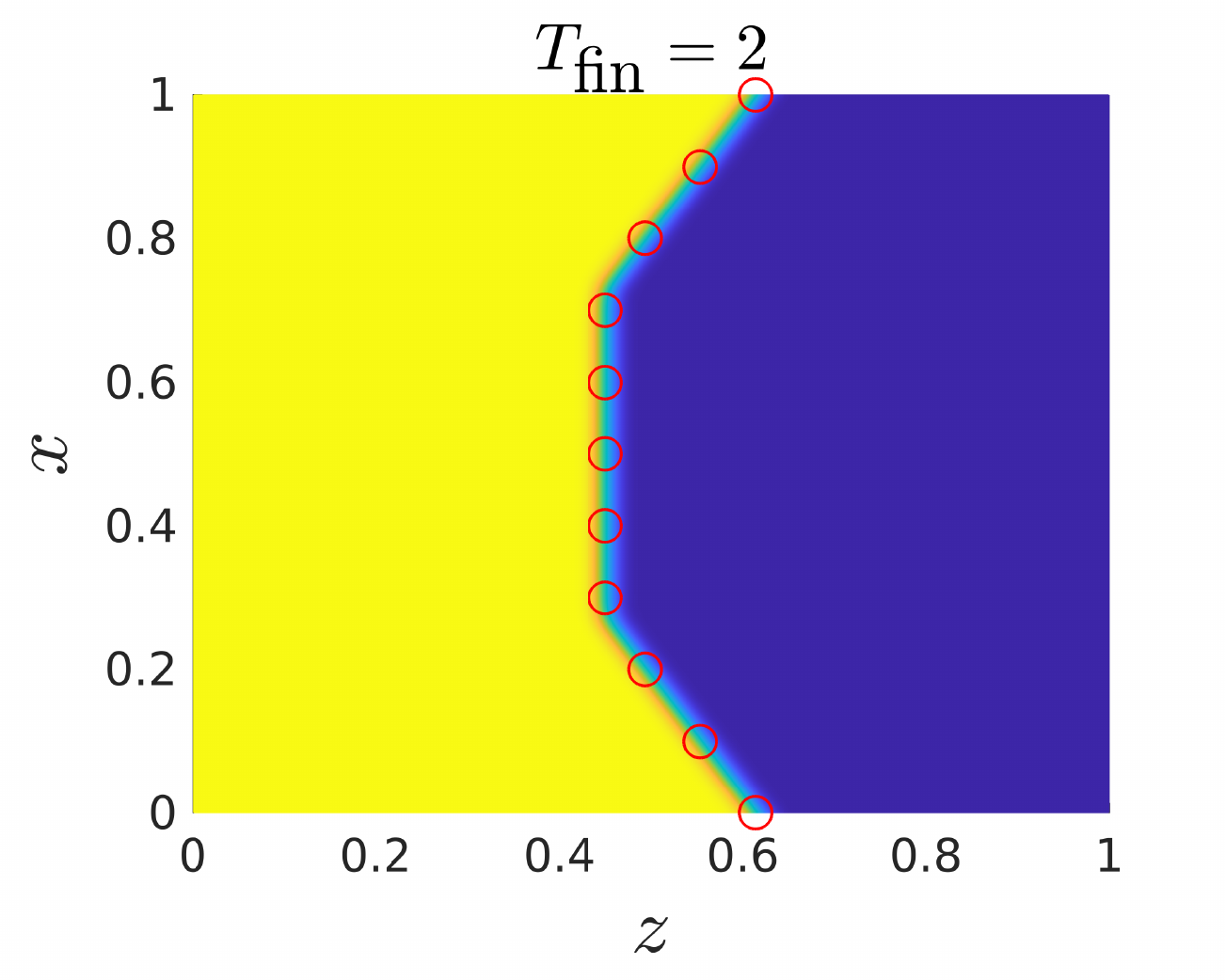}
    \label{fig:nonconstfrontC}
    \caption{density profile at $T_{\text{fin}}=2$}
    \end{subfigure} \\
    [3ex]
    \begin{subfigure}[b]{0.32\textwidth}
            \centering
            \includegraphics[width=\textwidth]{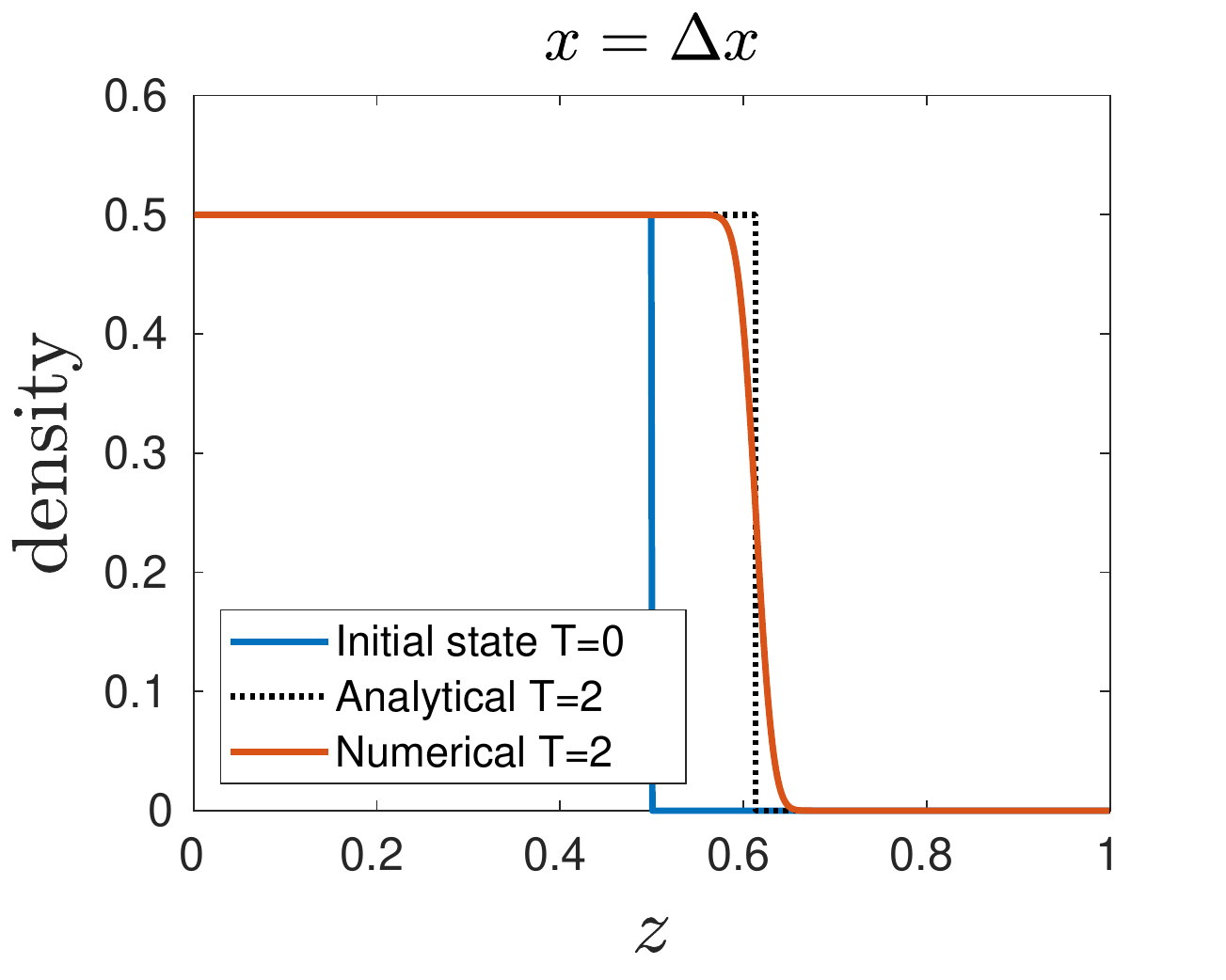}
    \label{fig:nonconstfrontD}
    \caption{slice at $x = \Delta x$}
    \end{subfigure}
    \begin{subfigure}[b]{0.32\textwidth}
            \centering
            \includegraphics[width=\textwidth]{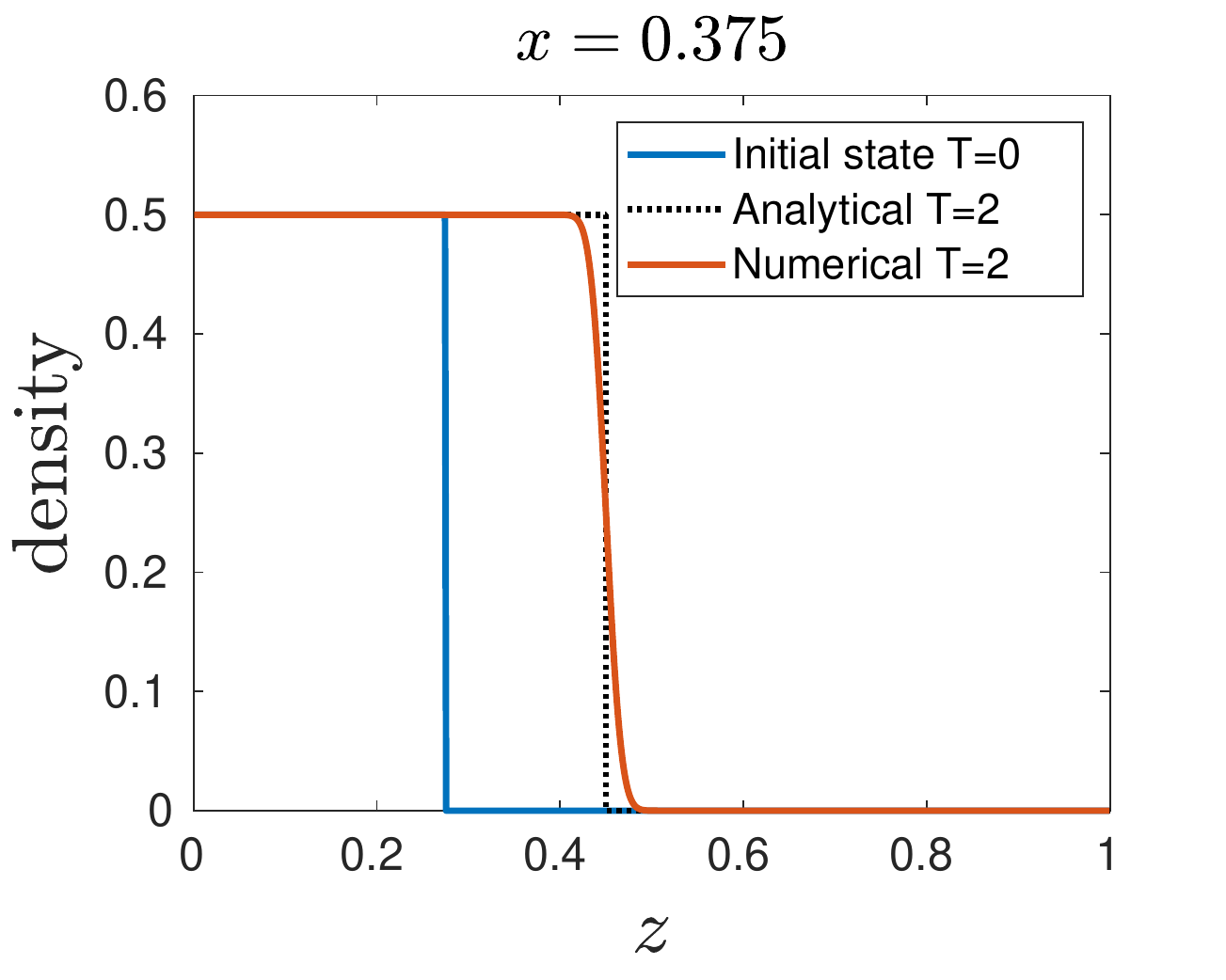}
    \label{fig:nonconstfrontE}
    \caption{slice at $x = 0.375$}
    \end{subfigure}
    \begin{subfigure}[b]{0.32\textwidth}
            \centering
            \includegraphics[width=\textwidth]{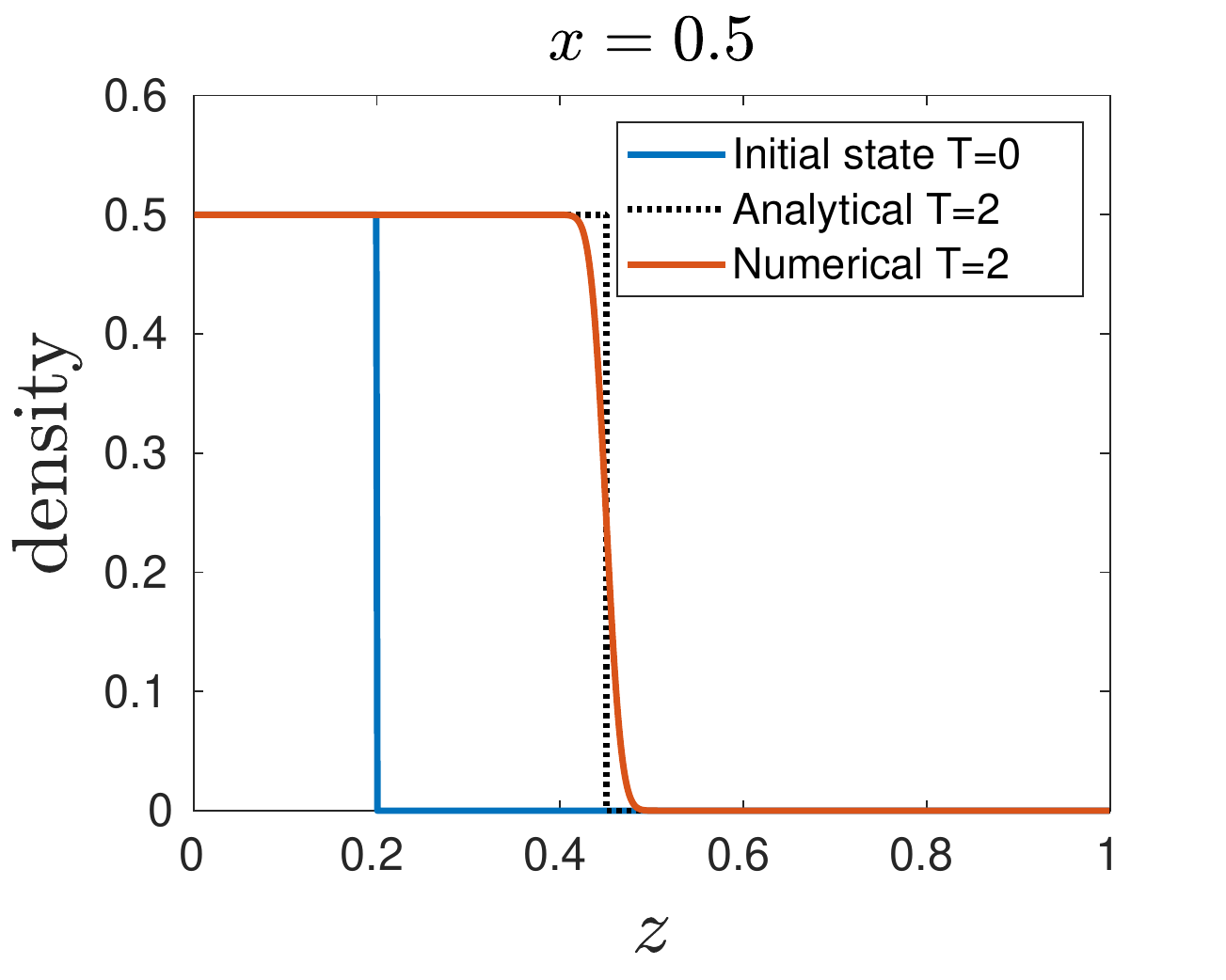}
    \label{fig:nonconstfrontF}
    \caption{slice at $x = 0.5$}
    \end{subfigure}\\
    [3ex]
	\caption{ Density profiles based on the simulation of the \textsf{V}-shaped initial front \eqref{eq:frontV} with the scheme in \eqref{eq: fully_discrete_scheme} using $N^x=N^z=800$.  The analytic solution  is computed using  \eqref{eq:zeta_vfront_global}. In the top row, the analytical front is marked by circles.  } 
	\label{fig:nonconstfront} 
\end{figure}

In Figure~\ref{fig:nonconstfront} we show numerical results for this test problem using the algorithm in \eqref{eq: fully_discrete_scheme}..  We use a grid with $N^x=N^z=800$ computational cells.  We set $z_0=0.2$ (so that $|\partial_x \zeta _0(x)| = 0.6)$ and let $\eta =  1/1.1$.  It is clear from these plots that the analytic solution \eqref{eq:zeta_vfront_global} and the numerical solution provide the same position of the front. 

\paragraph{Example 4: Smooth front with small $\eta$.}

We repeat the test problem above, except now the initial front is given by a smooth function:
\begin{equation}
\zeta_0(x) =  \frac14 \cos(2\pi x)+0.3.
\end{equation}
In order to fulfill the condition $\eta |\partial_x \zeta|<1$ we choose $\eta^{-1} = \max\limits_{x \in\mathbb{T}}  \zeta_0(x) + 1 = {\pi}/2 +1$.

In Figure~\ref{fig:smooth_front} we show numerical results for this test problem using the algorithm in \eqref{eq: fully_discrete_scheme}.   We use a grid with $N^x=N^z=800$ computational cells and simulate the solution to a time horizon $T_{\text{fin}}=2$.  As in the previous example, \eqref{eq:zeta_vfront_global} can be used to evaluate the location of the front analytically.  The results in Figure~\ref{fig:smooth_front} demonstrate that the numerical and analytical solutions agree even when $\partial_x \zeta_0$ is not piece-wise constant.
\begin{figure}[t!]
	\centering
	\begin{subfigure}[b]{0.32\textwidth}
		\centering
		\includegraphics[width=\textwidth]{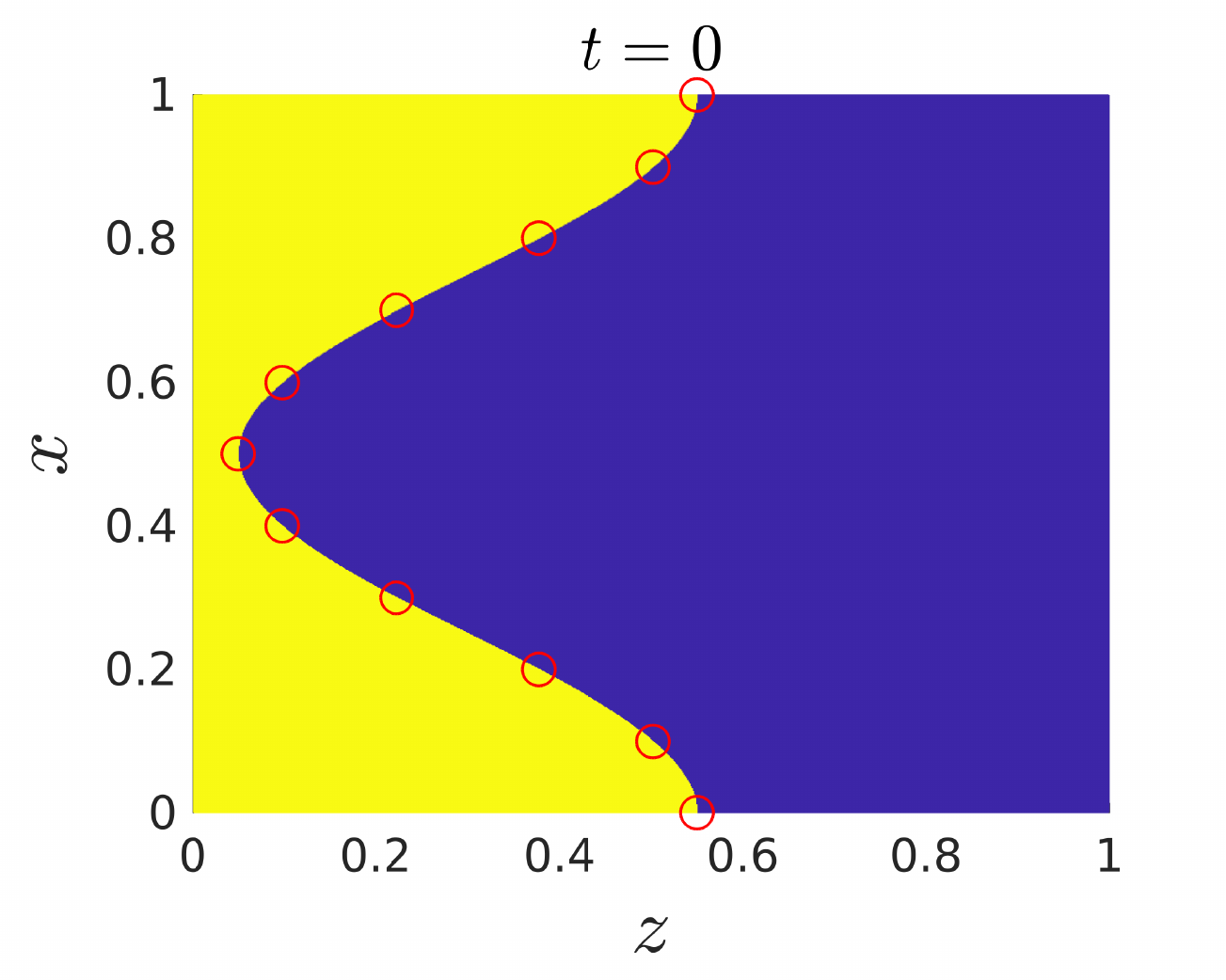}
		\label{fig:smooth_front_A}
		\caption{initial density profile}
	\end{subfigure}
	\begin{subfigure}[b]{0.32\textwidth}
		\centering
		\includegraphics[width=\textwidth]{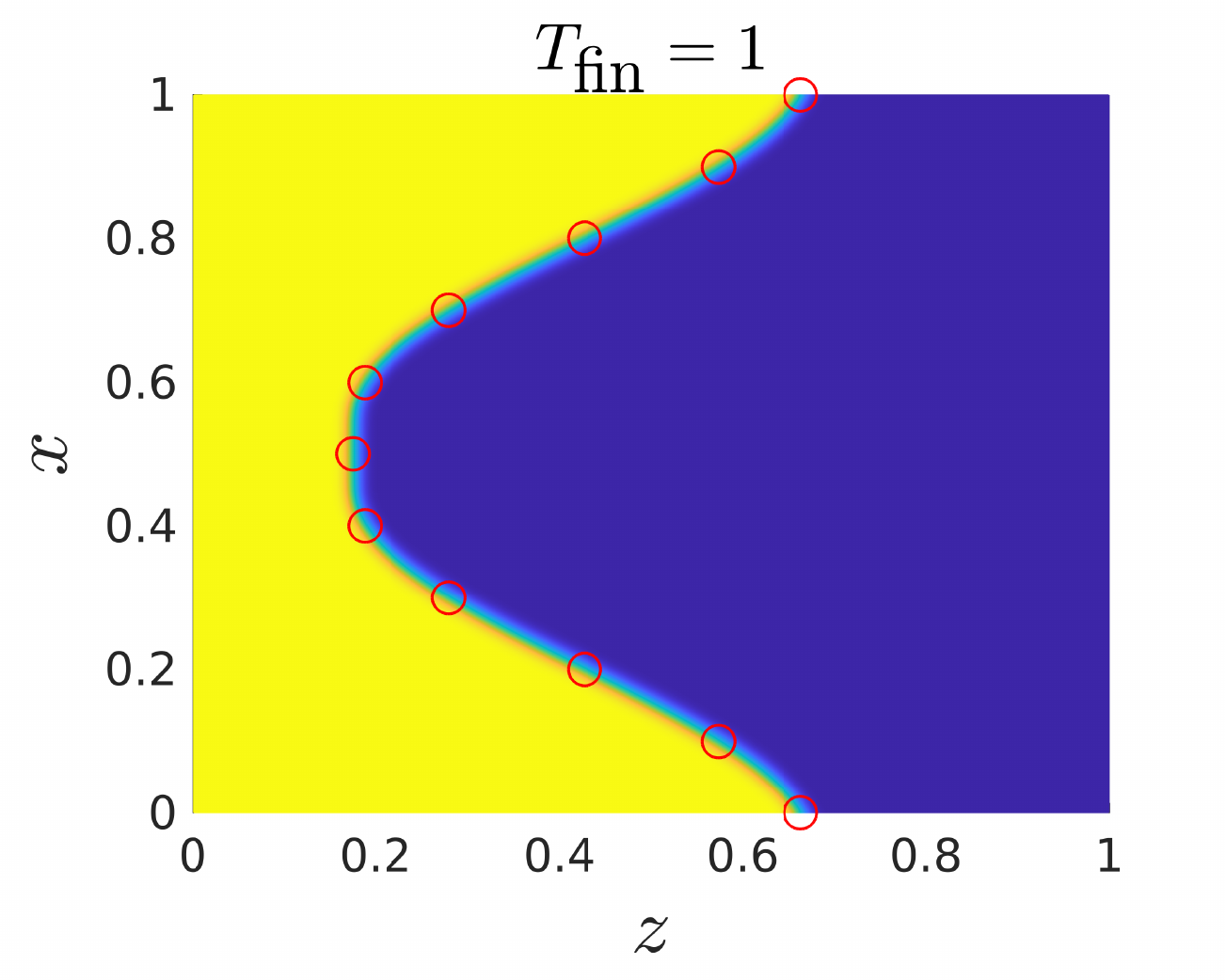}
		\label{fig:smooth_front_B}
		\caption{density profile at $t=1$}
	\end{subfigure}
	\begin{subfigure}[b]{0.32\textwidth}
		\centering
		\includegraphics[width=\textwidth]{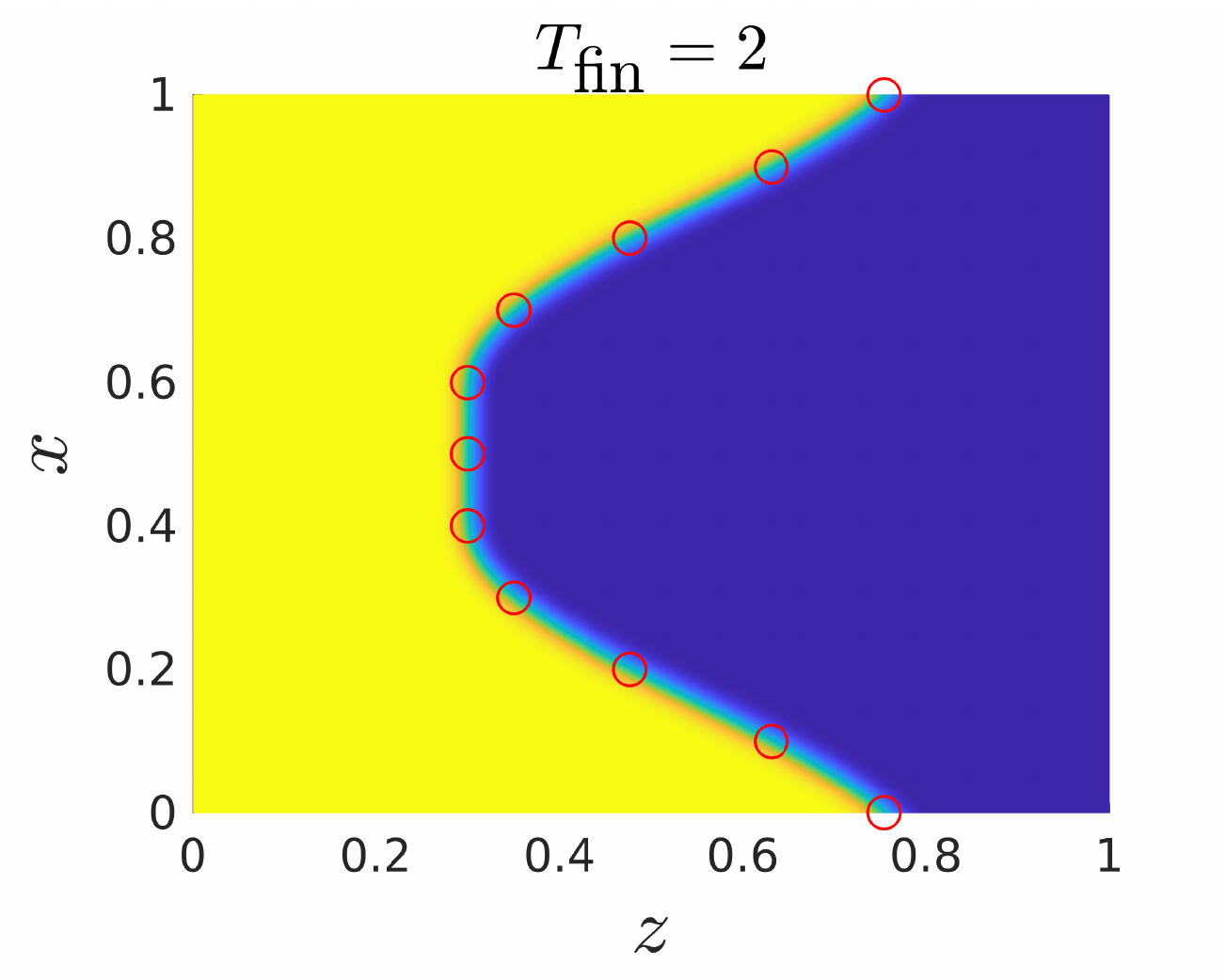}
		\label{fig:smooth_front_C}
		\caption{density profile at $T_{\text{fin}}=2$}
	\end{subfigure} \\
	[3ex]
	\begin{subfigure}[b]{0.32\textwidth}
		\centering
		\includegraphics[width=\textwidth]{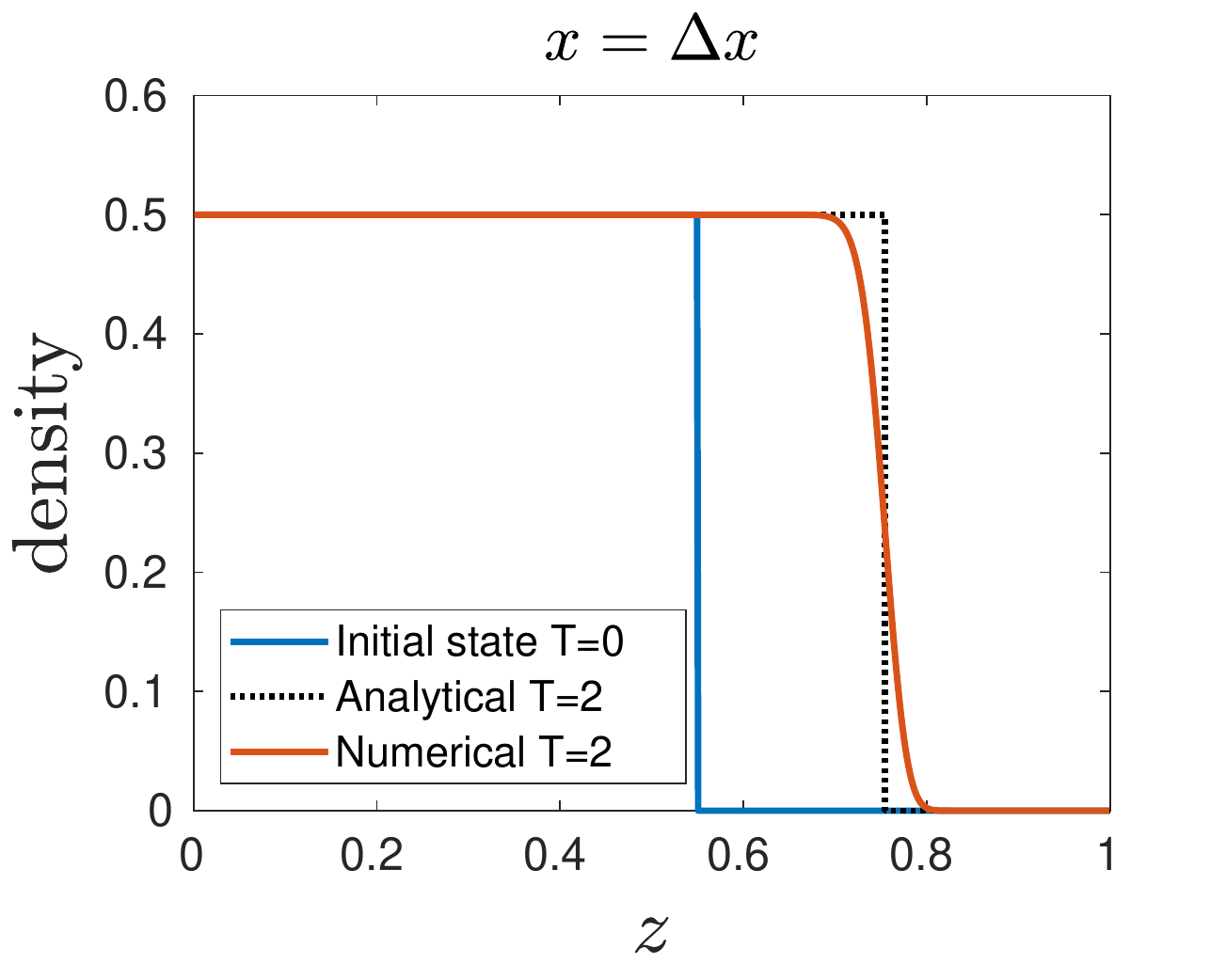}
		\label{fig:smooth_front_D}
		\caption{slice at $x = \Delta x$}
	\end{subfigure}
	\begin{subfigure}[b]{0.32\textwidth}
		\centering
		\includegraphics[width=\textwidth]{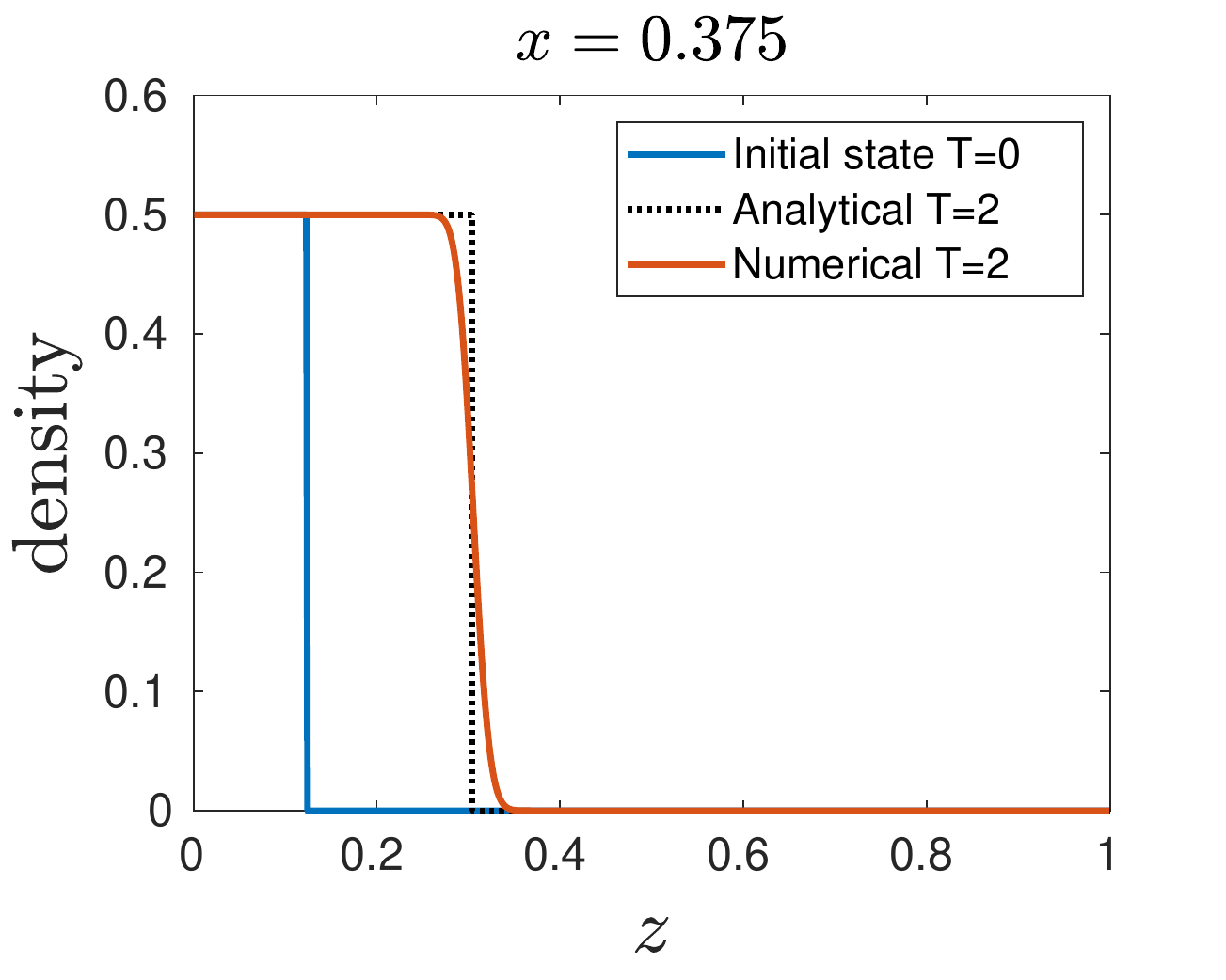}
		\label{fig:smooth_front_E}
		\caption{slice at $x = 0.375$}
	\end{subfigure}
	\begin{subfigure}[b]{0.32\textwidth}
		\centering
		\includegraphics[width=\textwidth]{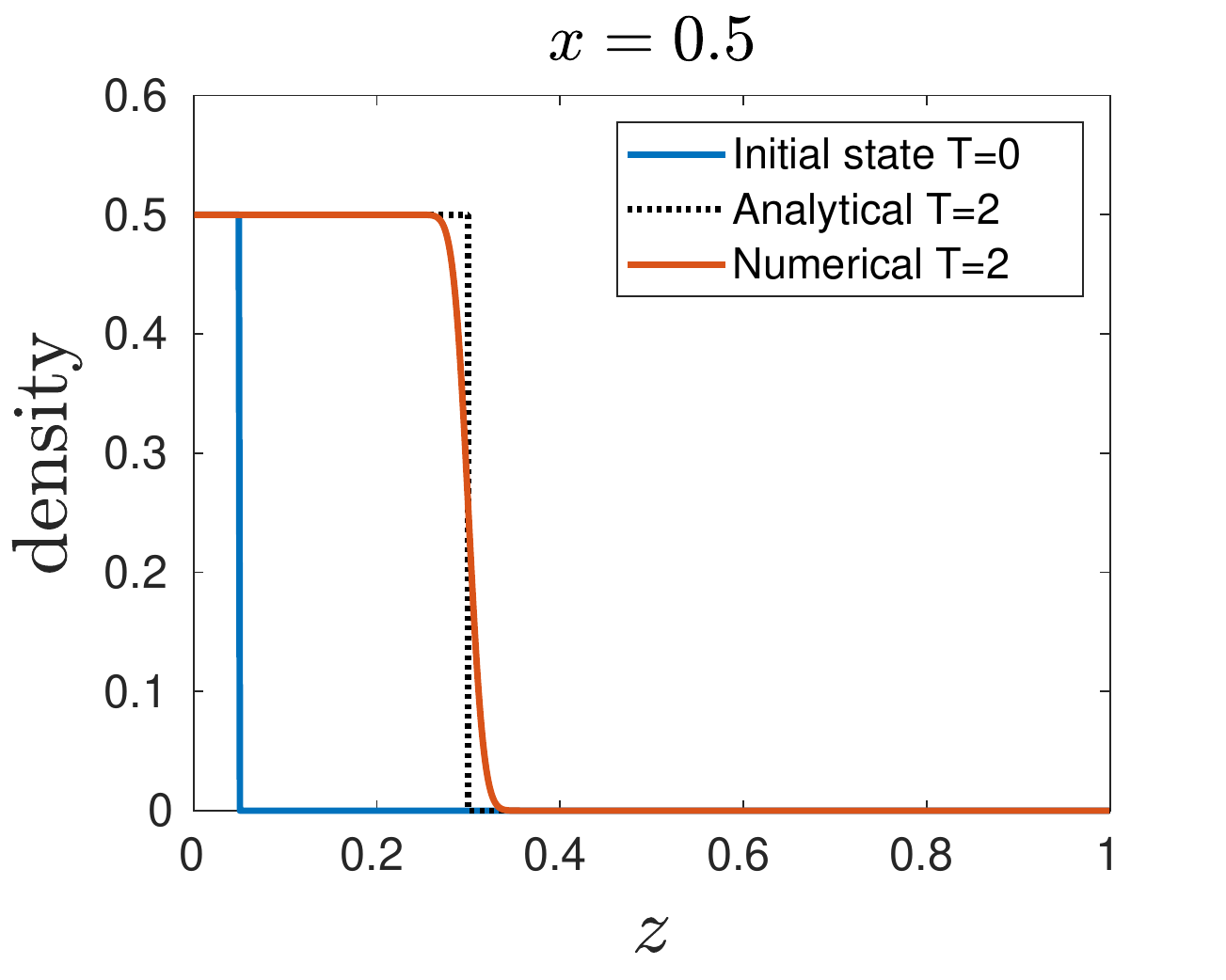}
		\label{fig:smooth_front_F}
		\caption{slice at $x = 0.5$}
	\end{subfigure}\\
	[3ex]
	\caption{Density profiles based on the simulation of the \textsf{V}-shaped initial front in \eqref{eq:frontV}, using the scheme in \eqref{eq: fully_discrete_scheme} with $N^x=N^z=800$. The analytic solution is computed using  \eqref{eq:zeta_vfront_global}. In the top row, the analytical front is marked by circles.}  
	\label{fig:smooth_front} 
\end{figure}

\paragraph{Example 5: \textsf{V}-shaped front with large $\eta$.} 
We consider again the initial condition~\eqref{eq:frontV} characterized by a \textsf{V}-shape front.  The model parameters are the same as in Example 3, except that $\eta$ is now large enough to ensure that $\eta | \partial_x \zeta | >1$.  In particular $|\partial_x \zeta |=0.6$, but $\eta = 10$.  In this case, we expect stalling for those processors $x\neq 0.5$; see  \eqref{eq:front_equation}.  Again at $x=0.5$ the solution is not described by Theorem~\ref{lem1}, since there the front is not $C^1$.  We expect that the front at this point will proceed at the maximum speed $\alpha/\rho_*$ and that the stalled points away front $x=0.5$ will also begin to move at this rate after they have been overtaken.  In other words, we conjecture that the global formula for the solution front is 
\begin{equation}
\label{eq:zeta_vfront__stall_global}
\zeta(t,x) = \max\left\{  \zeta_0(x) , \, z_0 + \frac{\alpha}{\rho_*} t\right\} .
\end{equation}

\begin{figure}[t!]
\centering
    \begin{subfigure}[b]{0.32\textwidth}
            \centering
            \includegraphics[width=\textwidth]{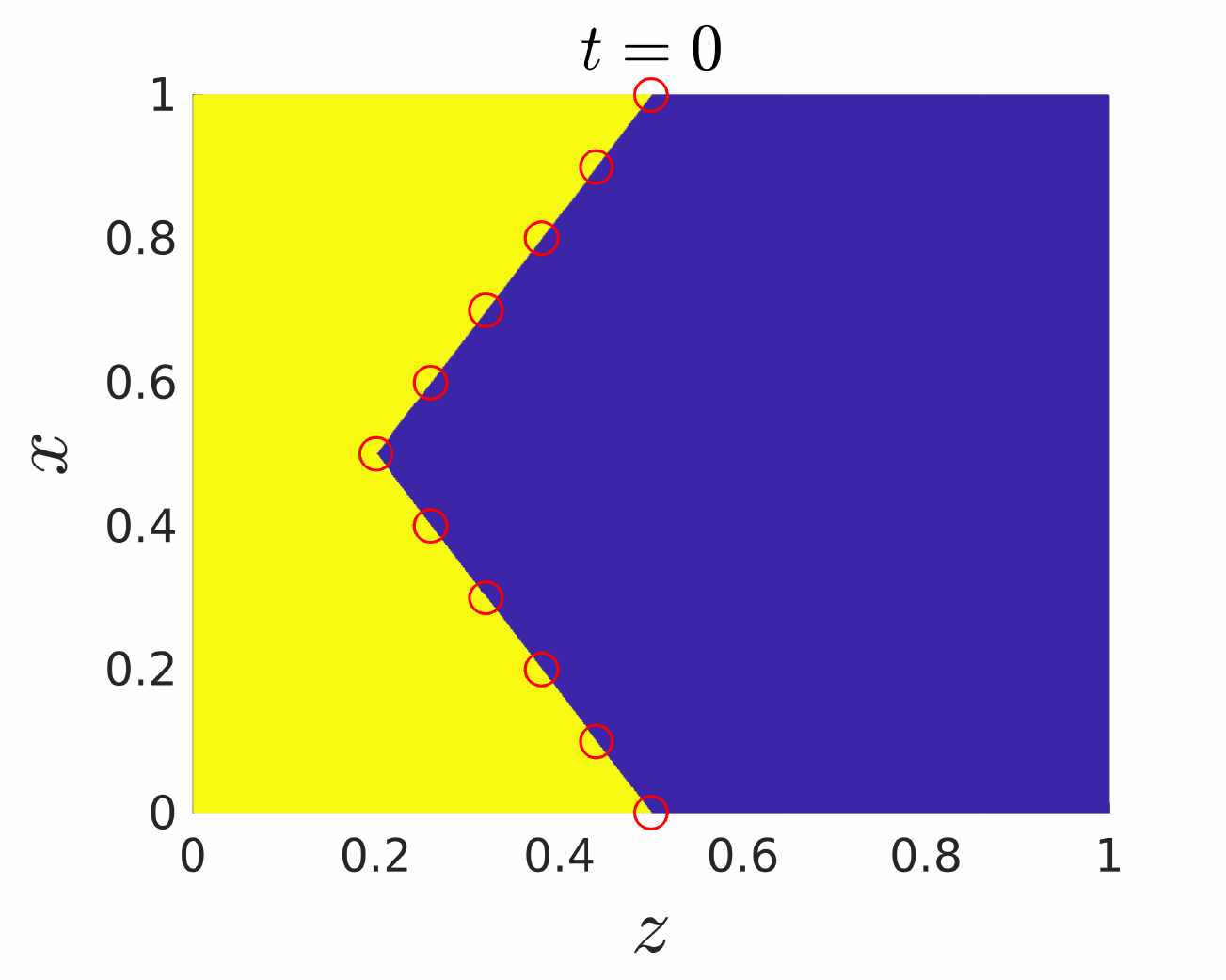}
    \label{fig:nonconstfront-etaHigh_A}
    \caption{initial density profile}
    \end{subfigure}
    \begin{subfigure}[b]{0.32\textwidth}
            \centering
            \includegraphics[width=\textwidth]{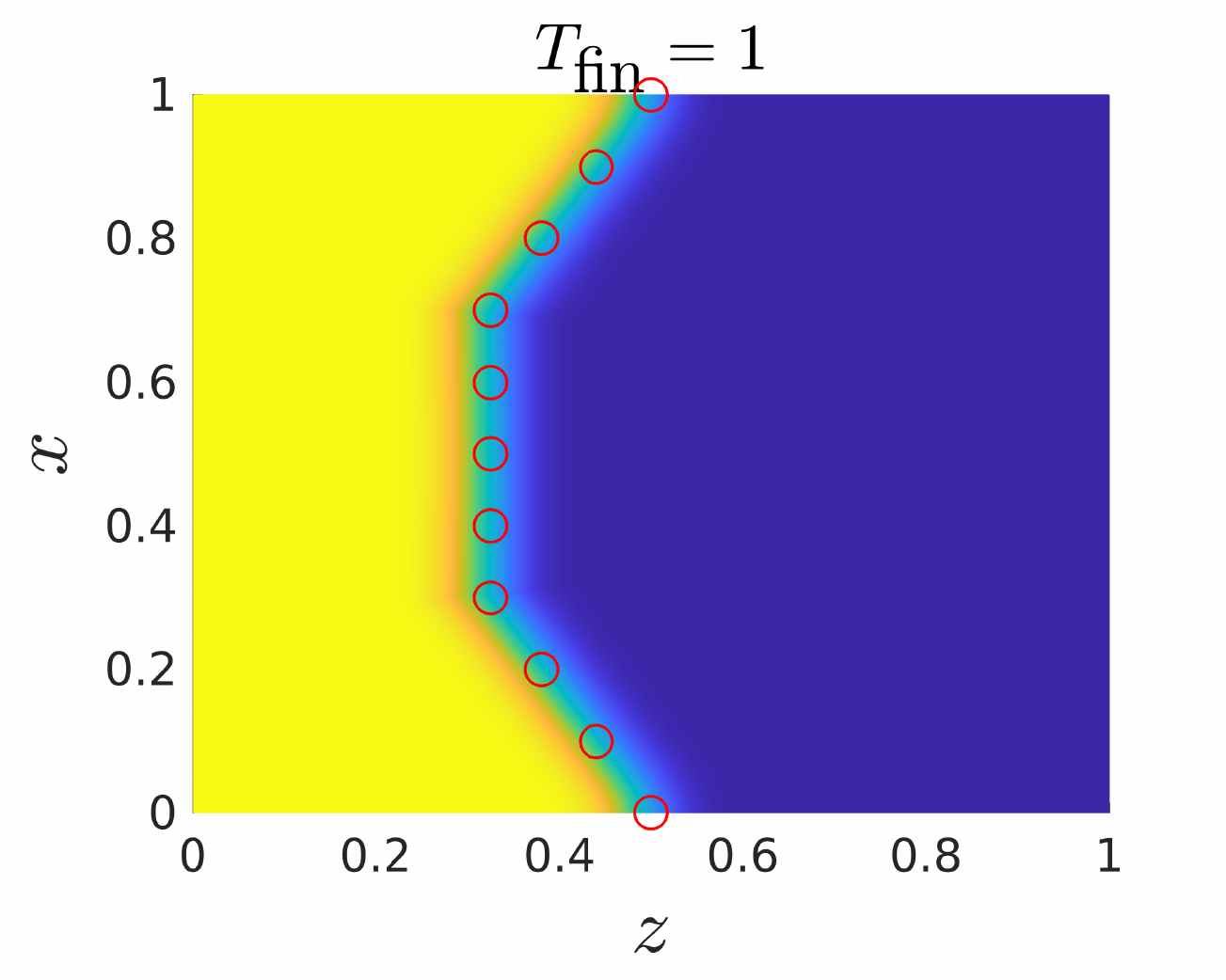}
    \label{fig:nonconstfront-etaHigh_B}
    \caption{density profile at $t=1$}
    \end{subfigure}
    \begin{subfigure}[b]{0.32\textwidth}
            \centering
            \includegraphics[width=\textwidth]{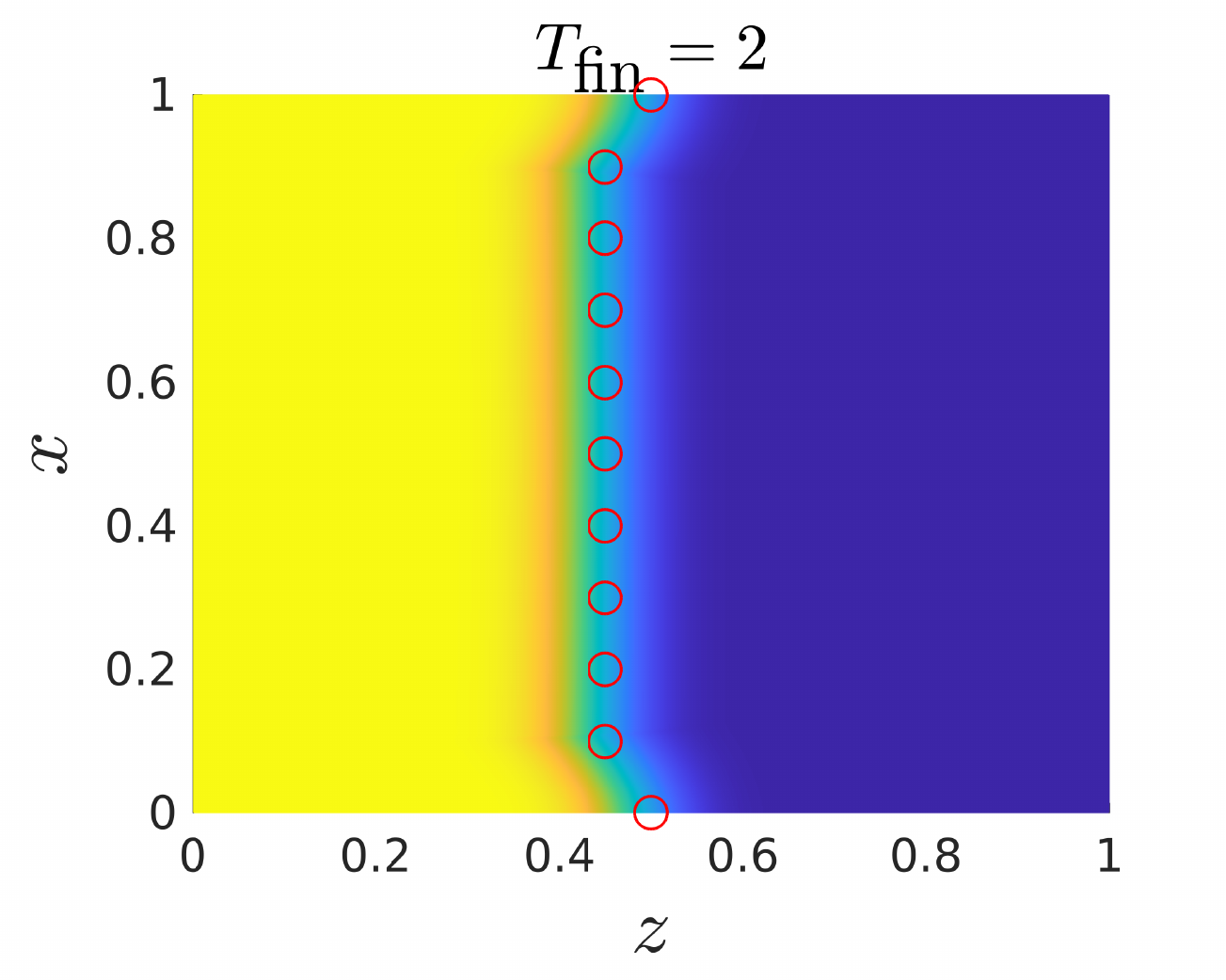}
    \label{fig:nonconstfront-etaHigh_C}
    \caption{density profile at $T_{\text{fin}}=2$}
    \end{subfigure} \\
    [3ex]
    \begin{subfigure}[b]{0.32\textwidth}
            \centering
            \includegraphics[width=\textwidth]{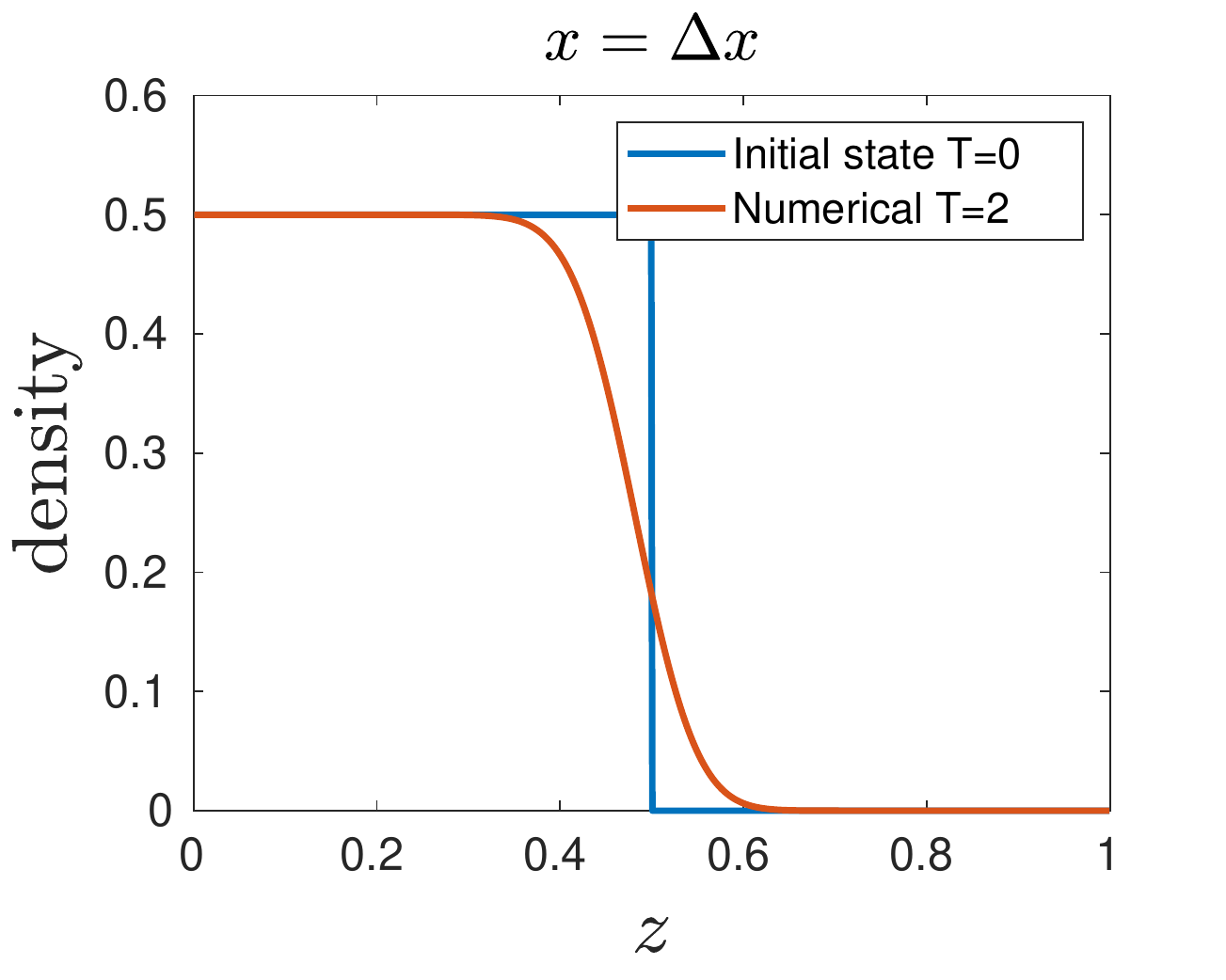}
    \label{fig:fig:nonconstfront-etaHigh_D}
    \caption{slice at $x = \Delta x$}
    \end{subfigure}
    \begin{subfigure}[b]{0.32\textwidth}
            \centering
            \includegraphics[width=\textwidth]{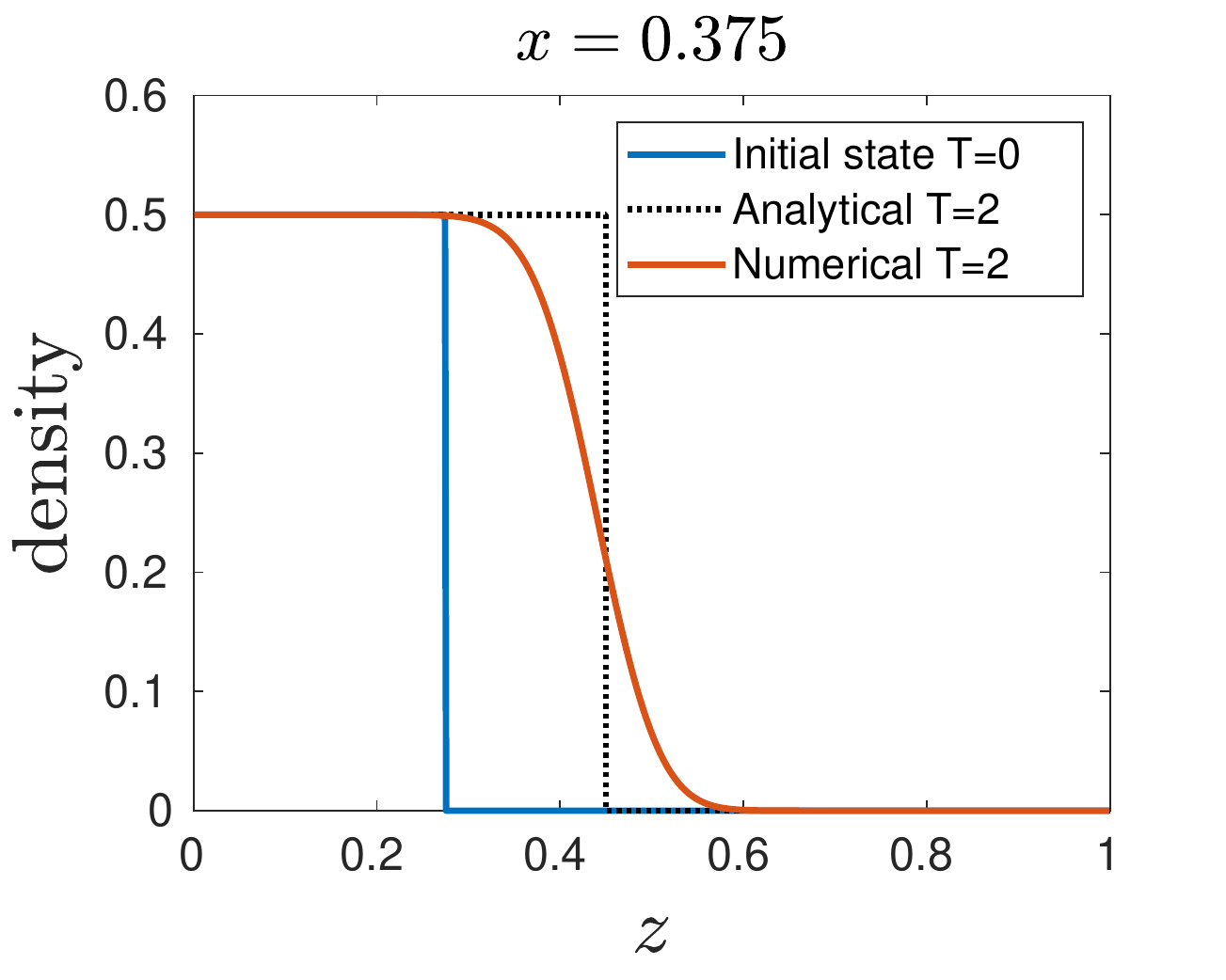}
    \label{fig:nonconstfront-etaHigh_E}
    \caption{slice at $x = 0.375$}
    \end{subfigure}
    \begin{subfigure}[b]{0.32\textwidth}
            \centering
            \includegraphics[width=\textwidth]{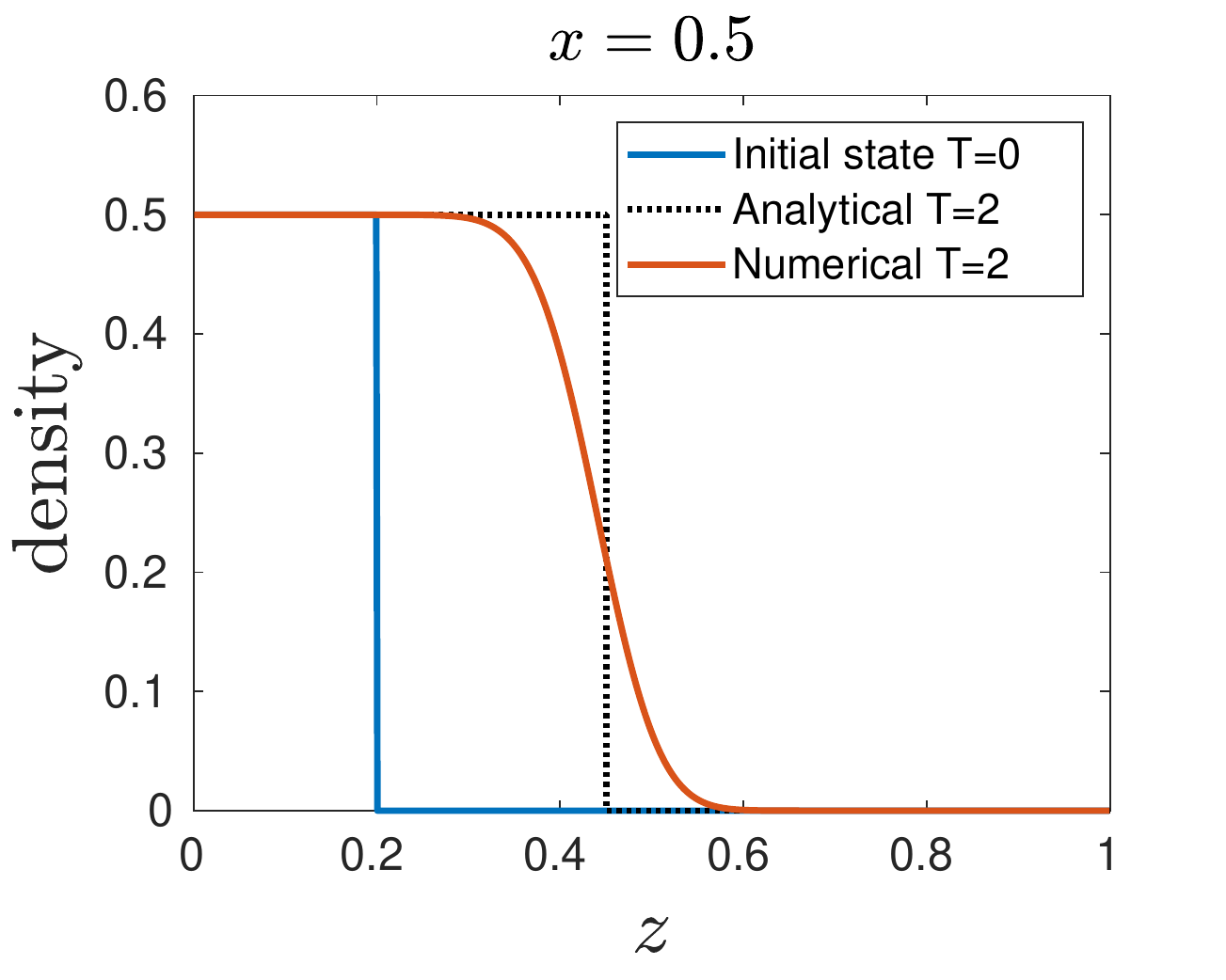}
    \label{fig:nonconstfront-etaHigh_F}
    \caption{slice at $x = 0.5$}
    \end{subfigure}\\
    [3ex]
	\caption{Density profiles based on the simulation of the \textsf{V}-shaped initial front \eqref{eq:frontV} with the scheme in \eqref{eq: fully_discrete_scheme} using $N^x=N^z=800$. The analytical solution is computed using \eqref{eq:zeta_vfront__stall_global}. In the top row, the analytical front is marked by circles.}  
	\label{fig:nonconstfront-etaHigh} 
\end{figure}

In Figure~\ref{fig:nonconstfront-etaHigh} we show numerical results for a grid with $N^x=N^z=800$ computational cells and a time horizon $T_{\text{fin}}=2$. 
These results agree with the analytical solution in \eqref{eq:zeta_vfront__stall_global}.

\subsection{Variations in the processor rate. } \label{sec:control}

The discrete model \eqref{eq:ode} can be used  to predict how variations in the processor rate affect computer performance, and to understand whether these rates can be controlled to achieve a prescribed objective, such as faster throughput.  For the continuum model in \eqref{model2}, such a control is specified via the function $\alpha$.  In this section, we briefly explore the effects of adapting $\alpha$ with the intent of a more extensive study in later work.   

In order to normalize the control action, we assume that for all $t$, 
\begin{equation} \label{alphabar} \int_{\mathbb{T}} \alpha(t,x) \mathrm{d}x = \bar{\alpha}.
\end{equation}
Since $\alpha$ is non-negative, so too is $\bar{\alpha}$; we assume further that $\bar{\alpha}>0$.  Then several choices for $\alpha$ can be envisioned.  The simplest choice is to set $\alpha$ equal to a constant. In this case, the normalization in \eqref{alphabar} implies that  $\alpha= \bar{\alpha}$.
For a front-type solution, this simple choice will be compared with a policy that depends on the initial data
 \begin{equation}
    \rho_0(x,z)=r H( \zeta_0(x)-z), \quad r<\rho_*, \quad  \text{and} \quad  \zeta_0(x)>0,
\end{equation}
Since the slope $\partial_x \zeta(t,x)$ controls the throttling due to neighboring processors, the condition $\partial_x \zeta(t,x)=0$ is desirable.  We therefore propose a choice for $\alpha$ that gives priority to processors that are trailing in the initial configuration.  Specifically, we consider
\begin{equation}\label{choice1}
\alpha(x)=  C_\alpha \, \rho_* \, \left( \zeta_{\rm{max}} - \zeta_0(x) \right) , 
\end{equation}
where $\zeta_{\rm{max}}>\max_{x \in \mathbb{T}} \zeta_0(x)$ is a fixed constant and $C_\alpha$, which depends on $\zeta_0$, is chosen  so that \eqref{alphabar} holds.
This choice of $\alpha$ will increase the speed of trailing processors

 In order to compare the temporal and spatial evolution of $\rho$ based on the policy \eqref{choice1} versus the constant $\alpha = \bar{\alpha}$ policy, we compute some quantities of interest
\begin{equation} \label{quantities}
\begin{aligned}
\omega_1(t,z) &:= \int_0^t  \int_{\mathbb{T}} \Phi(\rho(s,x,z), \sigma(s,x,z)) \mathrm{d}x \mathrm{d}s, \\
\omega_2(t,z) &:= \int_{\mathbb{T}} \Phi(\rho(t,x,z), \sigma(t,x,z) \mathrm{d}x, \\
\omega_3(t,z) &:= \int_0^t  \int_{\mathbb{T}} \rho(s,x,z) \mathrm{d}x \mathrm{d}s.
\end{aligned}
\end{equation}
The quantity $\omega_1$ measures the cumulative outflow at stage of completion $z$ up to time $t$; the quantity $\omega_2$ measures the current outflow at stage of completion $z$; and the quantity $\omega_2$ measures the cumulative part density at stage of completion $z$. The first two quantities are indicators of the processed
data whereas the last indicator shows the load of the processors, which may be related to the consumption of energy during computation.
In each case, larger values of $\omega_i$ are preferable.

For our numerical tests, we set  $T_{\rm{fin}}=1$ and $\bar{\alpha}=0.5$. The initial front $\zeta_0(x)$ is given as in~\eqref{eq:frontV} with $z_0=0.1$, in which case $C_\alpha \approx 0.5/0.575$. The quantities of interest are evaluated for  $z=0.5$ and $z=0.75$. 
  Figure~\ref{fig:control} shows the initial density profile and the profiles at $T_{\rm{fin}}=1$ for each of the two policies.  For the policy based on equation \eqref{choice1}, the control has been been active long enough to push the processors that were initially trailing ahead of the others.  This means the control has been on for too long and should have been modified or turned off at an some earlier time.  While the results of the two policies appear similar, the quantities of interest $w_i(t)$, $i=1,2,3$ displayed in Figure \ref{fig:wfigures} demonstrate the differences. In particular, with respect to $w_1$, the  policy in  \eqref{choice1} outperforms the constant $\alpha$ policy.
  
\begin{figure}[t!]
\centering
	 \begin{subfigure}[b]{0.24\textwidth}
		\centering
		\includegraphics[width=\textwidth]{nonconstantfrontT0-etaHigh}
		\label{fig:control_ic}
		\caption{initial data }
	\end{subfigure}
    \begin{subfigure}[b]{0.24\textwidth}
            \centering
            \includegraphics[width=\textwidth]{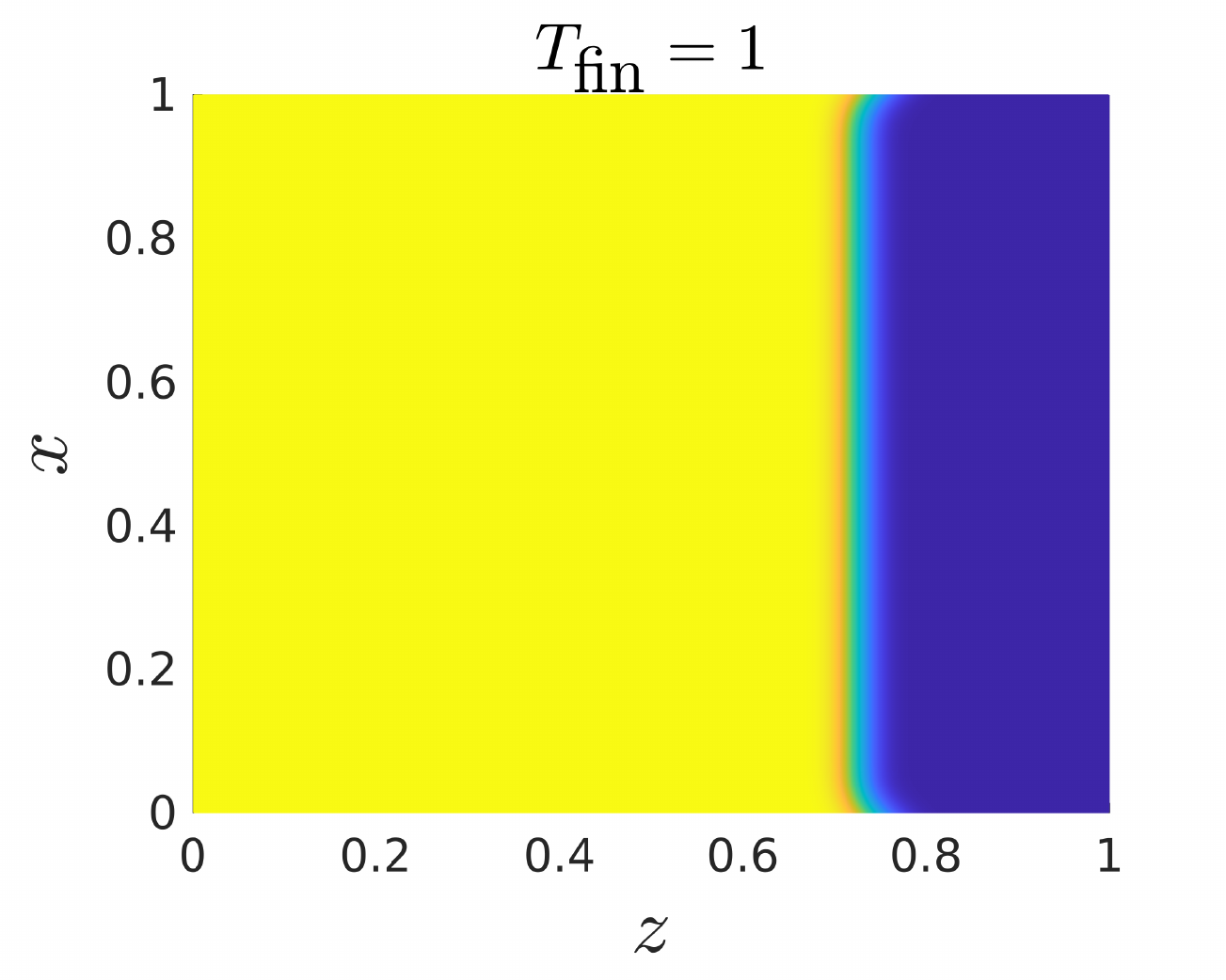}
    \label{fig:control_const_alpha}
    \caption{$\alpha = 0.5$}
    \end{subfigure}
    \begin{subfigure}[b]{0.24\textwidth}
            \centering
            \includegraphics[width=\textwidth]{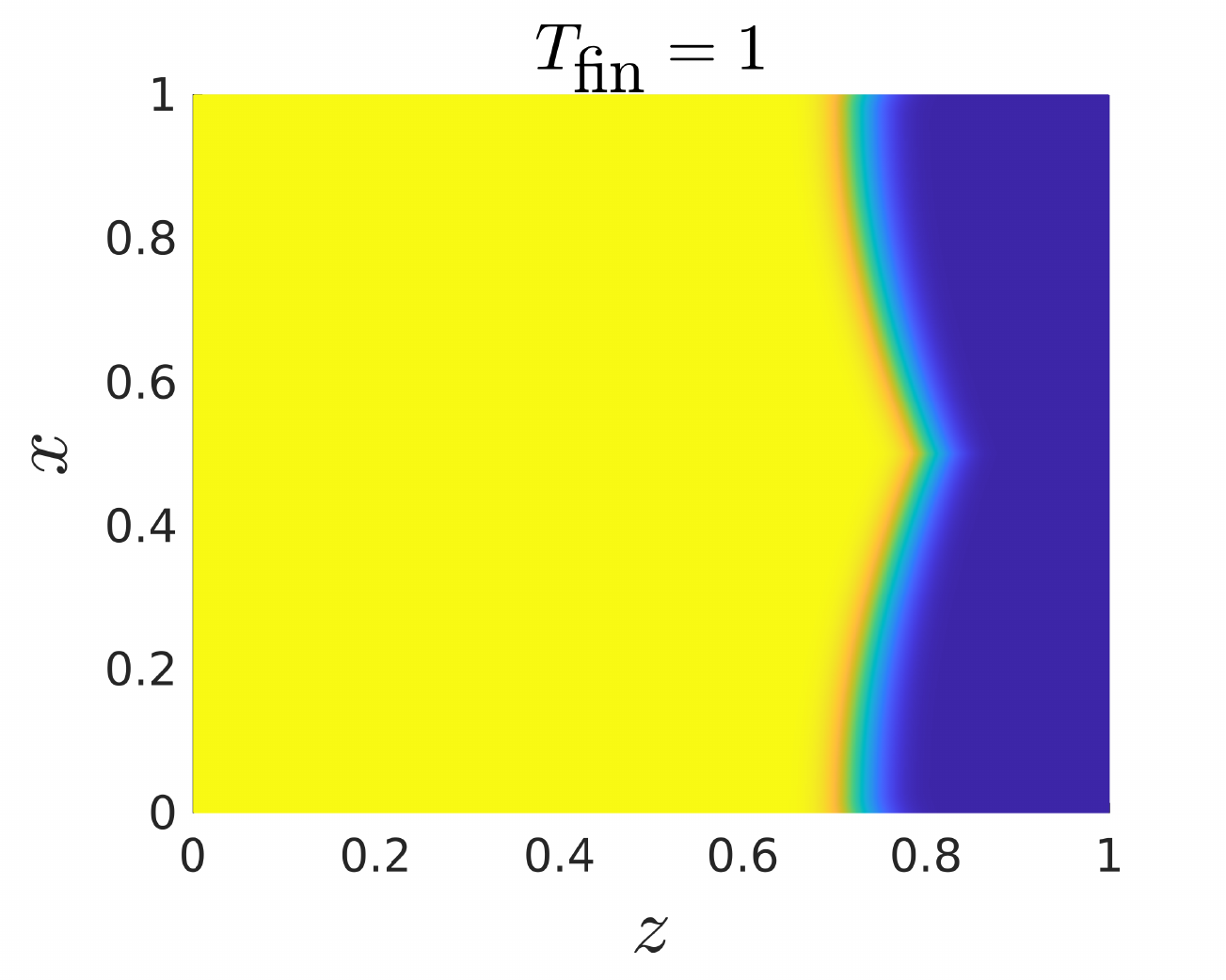}
    \label{fig:control_policy_alpha}
    \caption{$\alpha$ in \eqref{choice1}}
    \end{subfigure}
	\begin{subfigure}[b]{0.24\textwidth}
			\centering
			\includegraphics[width=\textwidth]{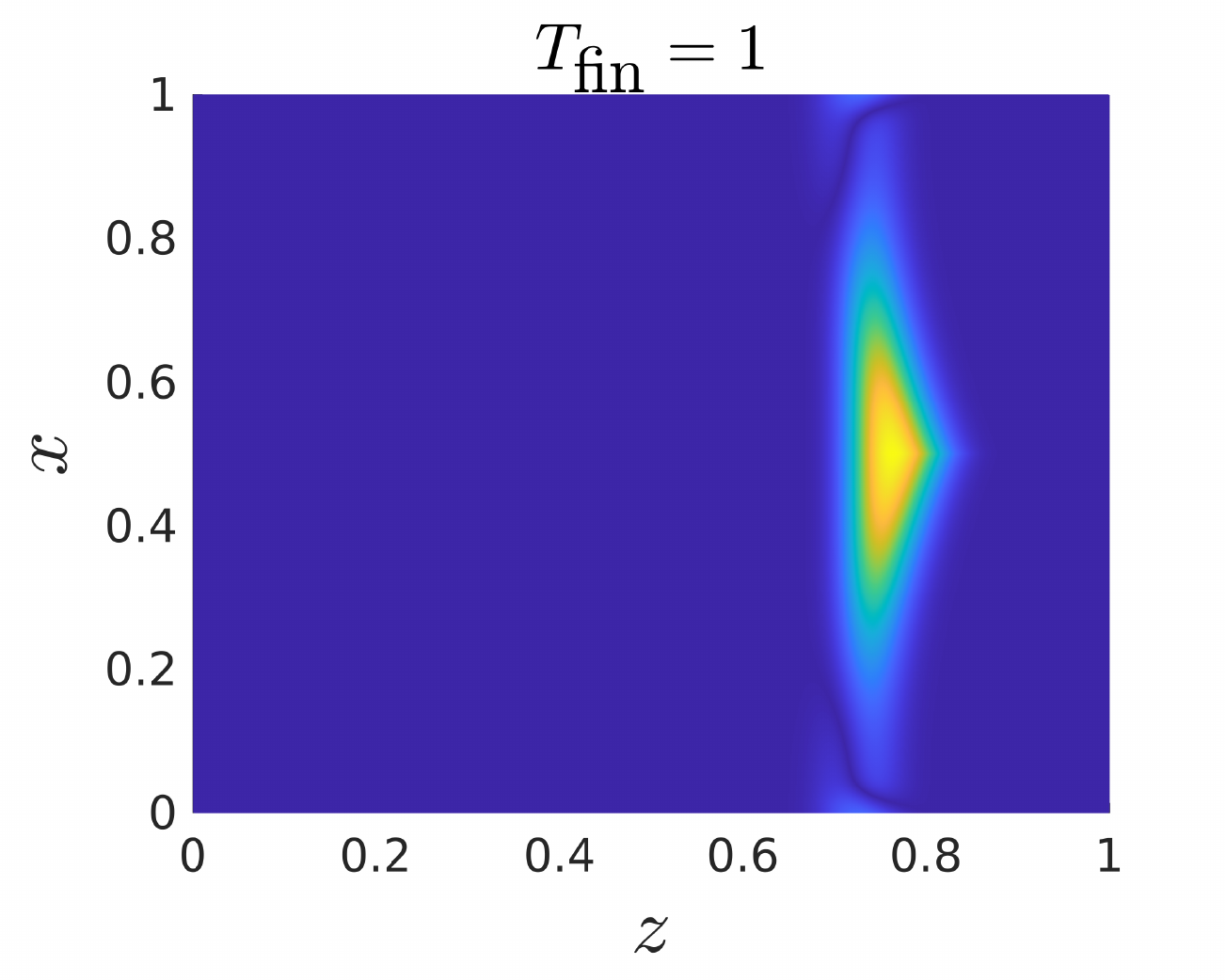}
			\label{fig:differencepolicies}
			\caption{Difference}
	\end{subfigure}
	\caption{Density profiles for two different policies choices of $\alpha$ with the scheme in \eqref{eq: fully_discrete_scheme} using $N^x=N^z=800$.  }  
	\label{fig:control}
\end{figure}

\begin{figure}[t!]
\centering
    \begin{subfigure}[t]{0.32\textwidth}
            \centering
            \includegraphics[width=\textwidth]{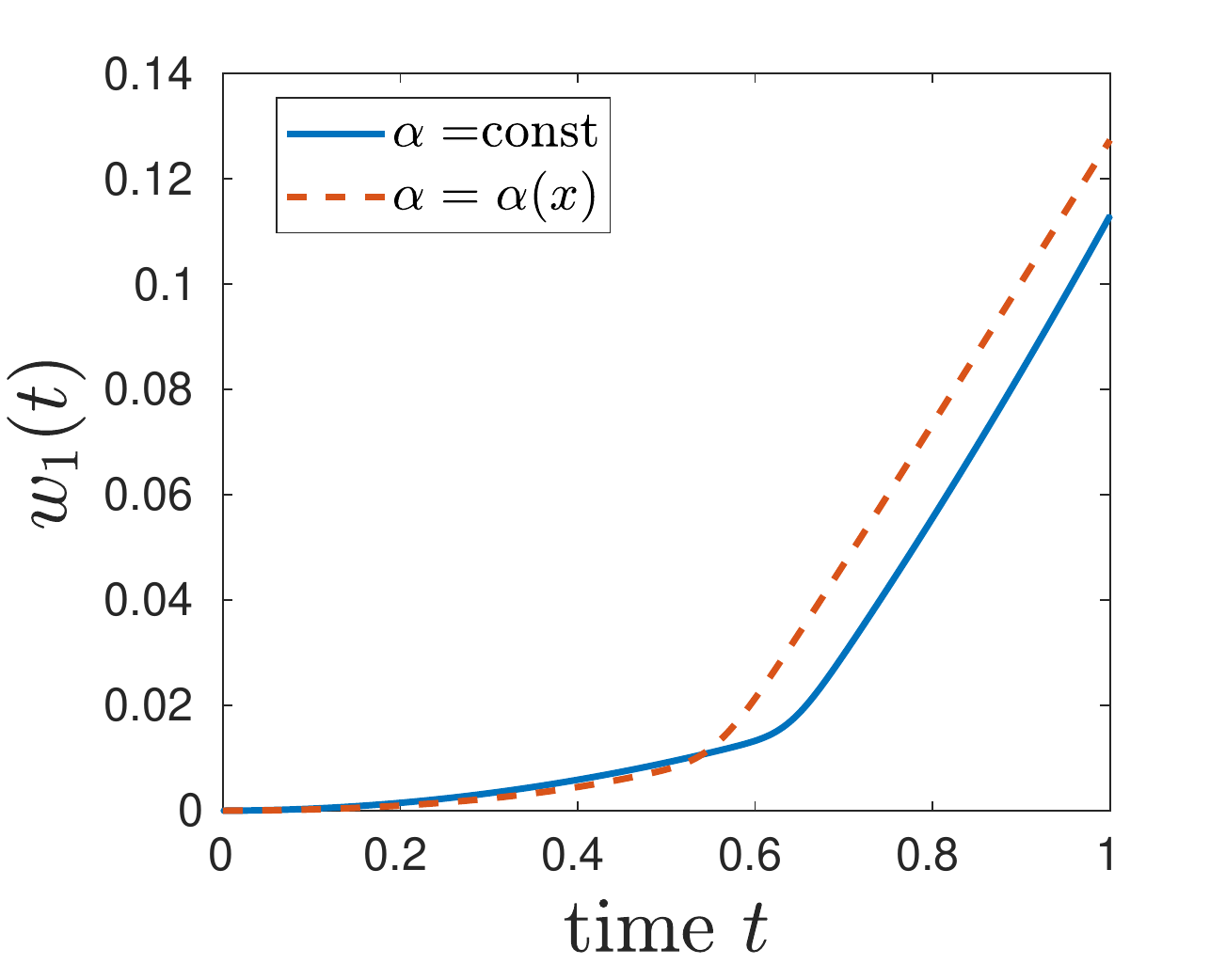}
    \label{fig:wfiguresA}
    \caption{$\omega_1$ at $z = 0.5$}
    \end{subfigure}
    \begin{subfigure}[t]{0.32\textwidth}
            \centering
            \includegraphics[width=\textwidth]{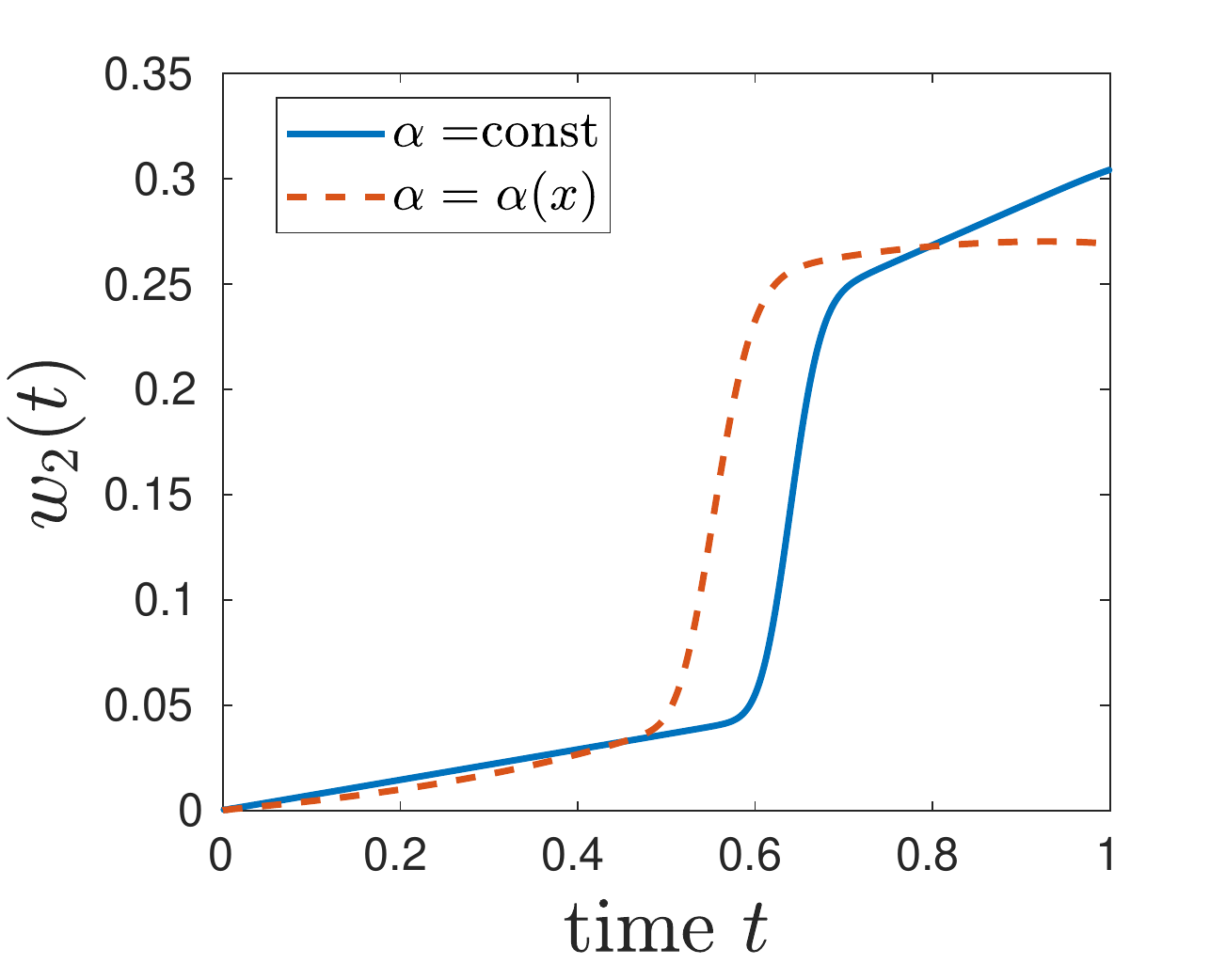}
    \label{fig:wfiguresB}
    \caption{$\omega_2$ at $z = 0.5$}
    \end{subfigure}
    \begin{subfigure}[t]{0.32\textwidth}
            \centering
            \includegraphics[width=\textwidth]{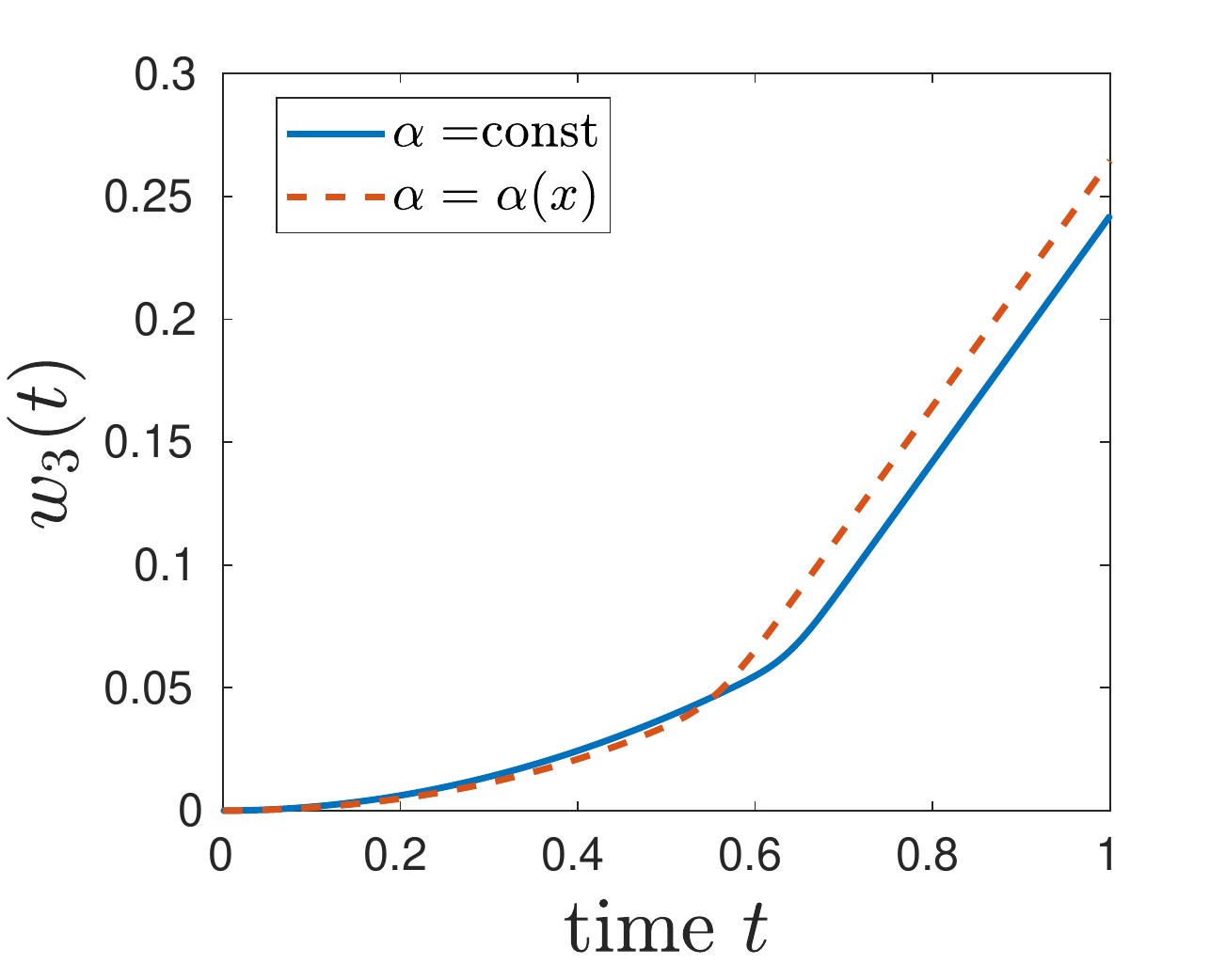}
    \label{fig:wfiguresC}
    \caption{$\omega_3$ at $z = 0.5$}
    \end{subfigure}\\
    \begin{subfigure}[t]{0.32\textwidth}
            \centering
            \includegraphics[width=\textwidth]{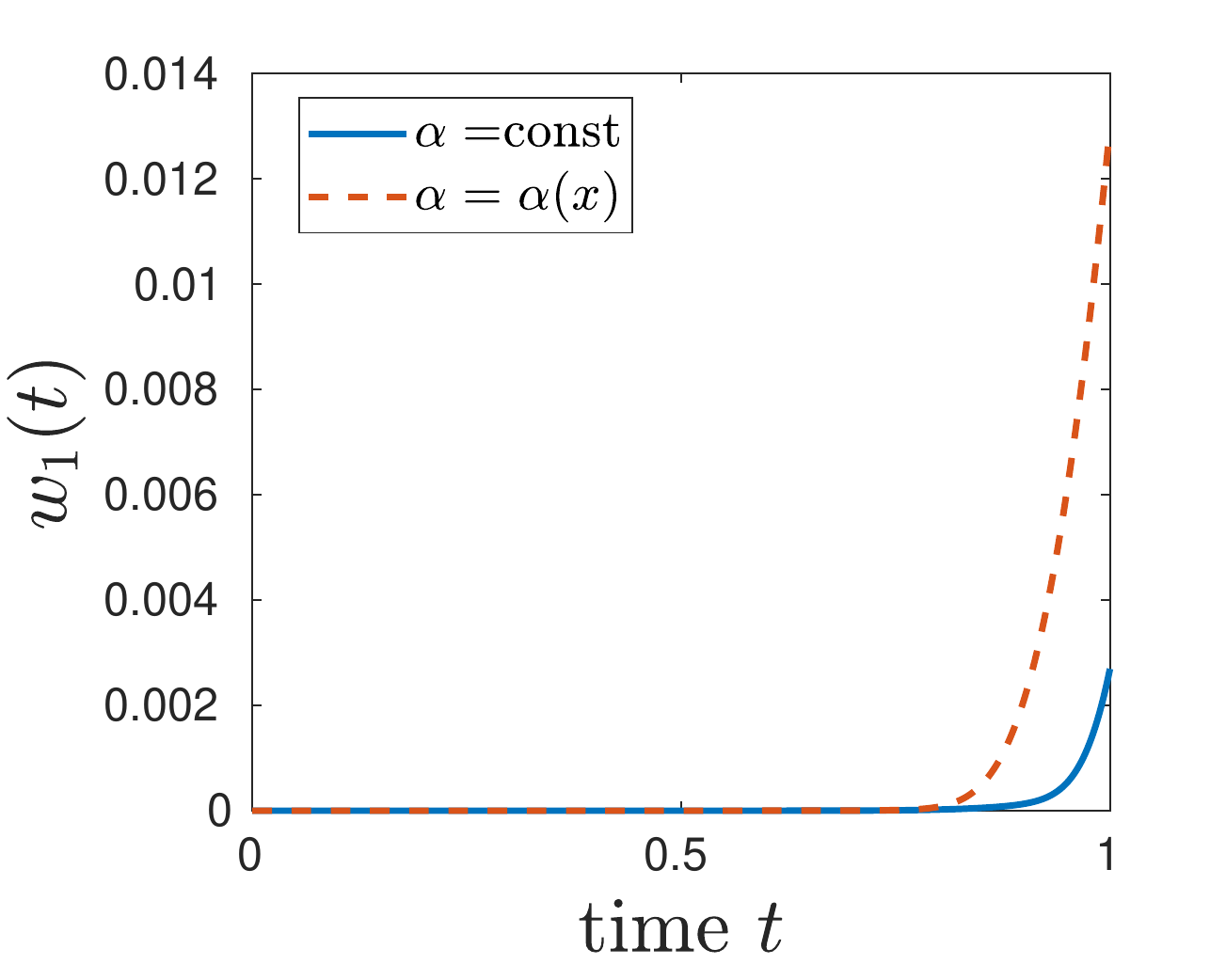}
    \label{fig:wfiguresD}
    \caption{$\omega_1$ at $z = 0.75$}
    \end{subfigure}
    \begin{subfigure}[t]{0.32\textwidth}
            \centering
            \includegraphics[width=\textwidth]{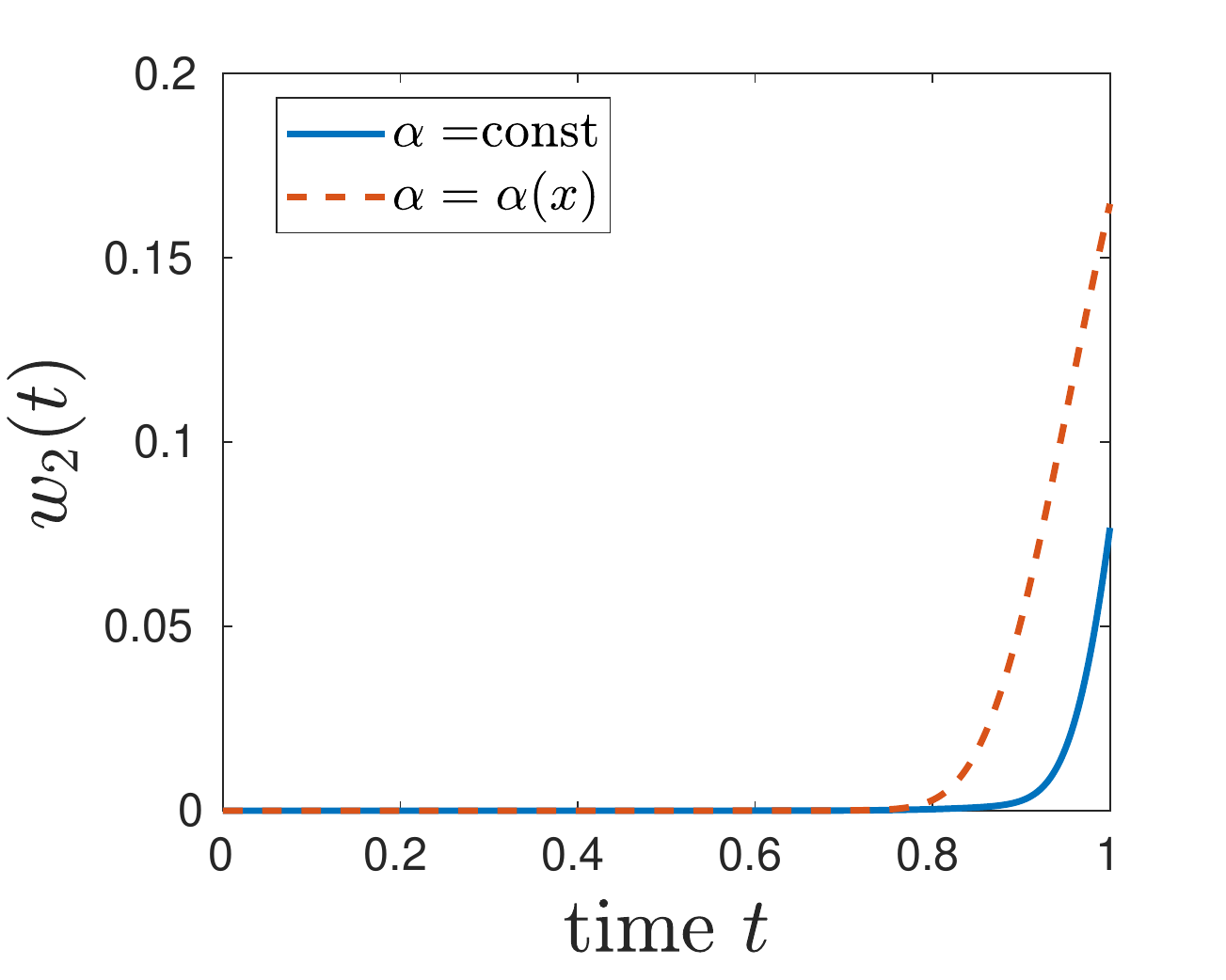}
    \label{fig:wfiguresE}
    \caption{$\omega_2$ at $z = 0.75$}
    \end{subfigure}
    \begin{subfigure}[t]{0.32\textwidth}
            \centering
            \includegraphics[width=\textwidth]{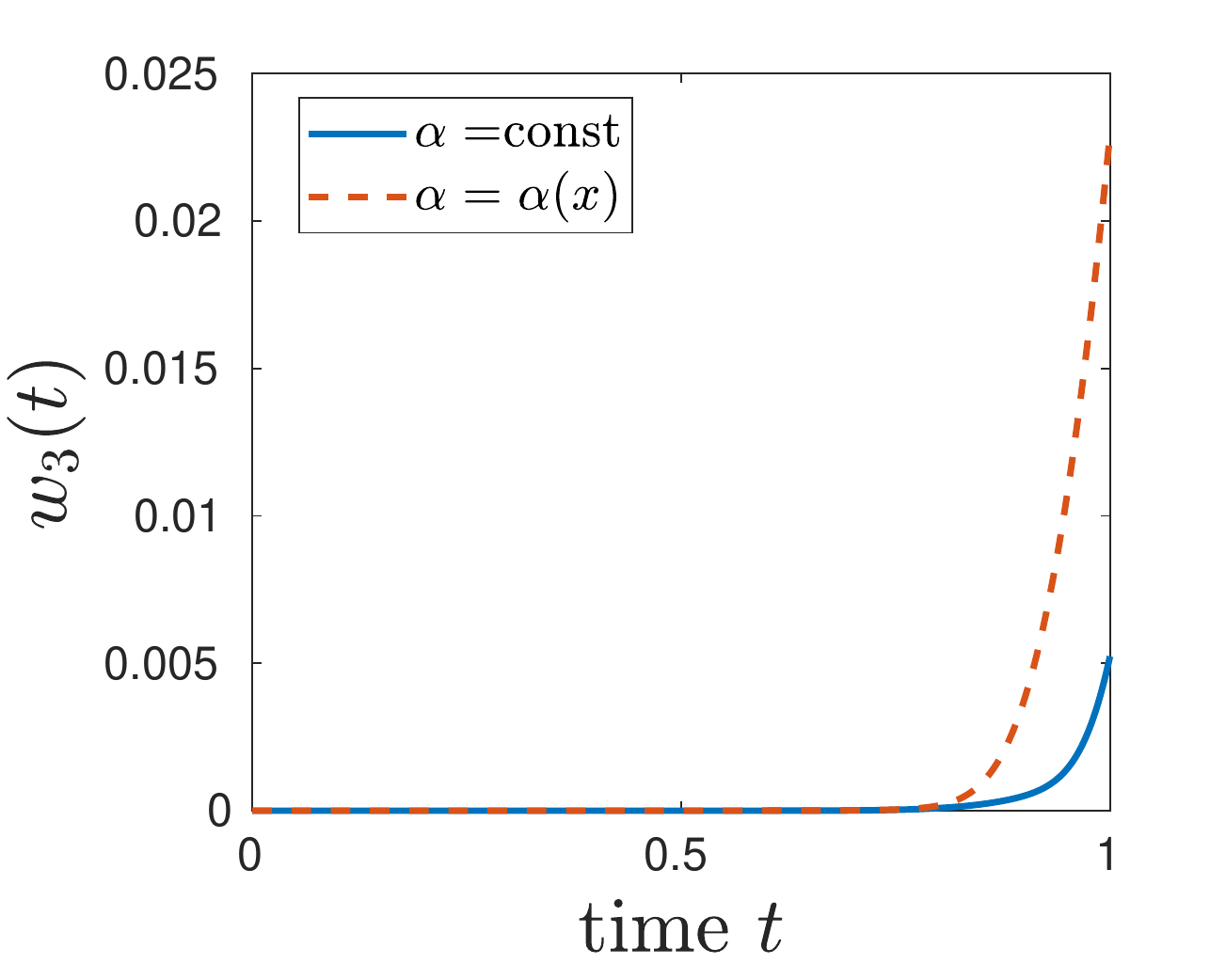}
    \label{fig:wfiguresF}
    \caption{$\omega_3$ at $z = 0.75$}
    \end{subfigure}\\

	\caption{Comparison of the quantities of interest $w_i, i=1,2,3$ given in~\eqref{quantities}. } 
	\label{fig:wfigures}
\end{figure}
\section{Summary}

In this paper, we have analyzed a recently derived mathematical model for the evolution of processed data in large-scale, asynchronous computers \cite{Barnard2019}. 
After the introduction of an auxiliary variable, the model is expressed as a system of partial differential equations.  It is possible to  prove existence of discontinuous solutions to this system for a particular class of initial and boundary conditions.  These solution describes the flow of information as a series of propagating fronts.  A numerical scheme for the reformulated system has also been designed based on a relaxation approximation.  Numerical simulations based on this scheme demonstrate qualitative agreement with theoretical findings. 

We have also briefly explored the effects of local processor speed on quantities of interest predicted by the model.  In future work, we intend to investigate more extensively control mechanisms for optimizing important objectives related to performance of the large-scale computers. Further it is planned to validate the model based on experimental data.

\subsection*{Acknowledgment}
This work has been supported by HE5386/14,15,18-1, ID390621612 Cluster of Excellence Internet of Production (IoP), the US National Science Foundation, RNMS (KI-Net) grant 11-07444, and the U.S. Department of Energy, Office of Advanced Scientific Computing Research. The work of CDH was performed at the Oak Ridge National Laboratory, which is managed by UT-Battelle, LLC under Contract No. De-AC05-00OR22725.

\bibliographystyle{plain}
\bibliography{references}

 \end{document}